\documentclass[english, 11pt]{amsart}

\usepackage{geometry}
\usepackage{amsmath, amssymb, amsthm}
\usepackage{mathtools}
\usepackage[all]{xy}
\usepackage{needspace}
\usepackage{mathrsfs}
\usepackage{booktabs}
\usepackage{longtable}

\usepackage[colorlinks=true,linkcolor=blue,citecolor=blue,urlcolor=blue]{hyperref}

\newcommand{\NN}{\mathbb{N}}
\newcommand{\ZZ}{\mathbb{Z}}
\newcommand{\CC}{\mathbb{C}}
\newcommand{\PP}{\mathbb{P}}
\newcommand{\OO}{\mathcal{O}}
\newcommand{\QQ}{\mathbb{Q}}
\newcommand{\FF}{\mathbb{F}}

\DeclareMathOperator{\Sym}{Sym}

\newcommand{\style}[1]{{\sf #1}}
\newcommand{\Twist}{\style{Twist}}

\theoremstyle{plain}
\newtheorem{thm}{Theorem}[section]
\newtheorem{lem}[thm]{Lemma}
\newtheorem{prop}[thm]{Proposition}
\newtheorem{cor}[thm]{Corollary}

\theoremstyle{definition}

\newtheorem{exa}[thm]{Example}
\newtheorem{rem}[thm]{Remark}
\newtheorem{conv}[thm]{Convention}

\theoremstyle{remark}

\numberwithin{equation}{section}

\title[A Symmetry Method for Vanishing Lemmas and Hyperbolicity]
{A Symmetry Method for Key Vanishing Lemmas\\
 and the $2$-Jet Riemann--Roch Thresholds}

\author{Tao Cui}
\address{Academy of Mathematics and Systems Science, Chinese Academy of Sciences, Beijing 100190, China}
\email{tcui@lsec.cc.ac.cn}

\author{Lei Hou}
\address{Academy of Mathematics and Systems Science, Chinese Academy of Sciences, Beijing 100190, China}
\email{houlei@amss.ac.cn}

\author{Jianjun Liu}
\address{School of Mathematical Sciences, University of Chinese Academy of Sciences, Beijing 100049, China}
\email{liujj@lsec.cc.ac.cn}

\author{Song-Yan Xie}
\address{State Key Laboratory of Mathematical Sciences, Academy of Mathematics and Systems Science, Chinese Academy of Sciences, Beijing 100190, China; School of Mathematical Sciences, University of Chinese Academy of Sciences, Beijing 100049, China}
\email{xiesongyan@amss.ac.cn}

\date{\today}

\begin{document}
\raggedbottom
\begin{abstract}
We prove Kobayashi hyperbolicity for general surfaces of degrees $15$
and $16$ in $\PP^3$, and for complements of general smooth plane curves
of degree $11$.  In the respective general-type ranges, $15$ and $11$
are the first degrees for which the leading coefficient in Demailly's
invariant $2$-jet Riemann--Roch formula is positive.  We introduce a
symmetry method to establish the Key Vanishing Lemmas.  The proof relies
on involution covariance in both settings and, in compact degrees $15$
and $16$, on a commuting action by roots of unity; parallel modular rank
computations are used in the finite vanishing arguments.  The numerical
bounds follow from these vanishing statements and the geometric $2$-jet
argument.  Two relative $1$-jet locus theorems then allow the boundedness
method of Xie and Zhao to replace very general parameters by general
parameters, without changing the vanishing computations.
\end{abstract}

\maketitle


\section{Introduction}
\label{sec:intro}

\subsection{Motivation and background}
\label{sec:motivation}

Kobayashi conjectured that general hypersurfaces of sufficiently large
degree in $\PP^n$, as well as their complements, are
hyperbolic~\cite{Kobayashi1970}.  The expected degree bounds are
$2n-1$ for the hypersurface when $n\geqslant3$ and $2n+1$ for its
complement; see, for example,~\cite{Demailly2020}.  In particular, both
the compact problem for surfaces in $\PP^3$ and the logarithmic problem
for plane curves have conjectured threshold degree $5$.
Throughout this paper, \emph{general} means outside a proper
Zariski-closed subset of the parameter space, while \emph{very general}
means outside a countable union of proper Zariski-closed subsets.

This paper proves hyperbolicity for
general surfaces of degree $15$ and $16$ in $\PP^3$, and for complements
of general smooth plane curves of degree $11$.  With the very-general
bounds of Hou--Huynh--Merker--Xie for surfaces of degree at least
$17$~\cite{HHMX2026} and the boundedness method of Xie and
Zhao~\cite{XieZhao2026}, which together cover all higher degrees, this
gives hyperbolicity for every general surface in $\PP^3$ of degree at
least $15$.  A gap to the conjectured degree $5$ remains.

The degrees $15$ and $11$ arise from Demailly's Riemann--Roch estimate
for invariant $2$-jet differentials.  On a smooth surface in $\PP^3$,
the vanishing of negatively twisted symmetric differentials~\cite{Sakai1979}
motivates the passage to jets of order two.  Let $X\subset\PP^3$ be a smooth
surface of degree $d$, and let $C_d\subset\PP^2$ be a smooth plane curve
of degree $d$.  We write $E_{2,m}T_X^*$ for the bundle of invariant
$2$-jet differentials of weight $m$ on $X$, and
$E_{2,m}T_{\PP^2}^*(\log C_d)$ for the corresponding logarithmic bundle
on $(\PP^2,C_d)$.  Demailly's
calculation~\cite[Section~13, formulas~(13.8) and~(13.16)]{Demailly1997}
gives, as $m\to\infty$,
\begin{align*}
 \chi\bigl(X,E_{2,m}T_X^*\bigr)
   &=\frac{d(4d^2-68d+154)}{648}\,m^4+O(m^3),\\
 \chi\bigl(\PP^2,E_{2,m}T_{\PP^2}^*(\log C_d)\bigr)
   &=\frac{4d^2-51d+90}{648}\,m^4+O(m^3).
\end{align*}
In the general-type ranges, this coefficient is positive precisely for
$d\geqslant15$ in the compact case and $d\geqslant11$ in the logarithmic
case.  We call $15$ and $11$ the \emph{$2$-jet Riemann--Roch positivity
thresholds}.  Below these degrees, deciding whether such sections exist
on a particular fibre or on a general fibre requires a separate argument.

The differential equations used here are defined by sections of the
negatively twisted bundles
\[
 E_{2,m}T_X^*\otimes\OO_X(-t)
 \qquad\text{and}\qquad
 E_{2,m}T_{\PP^2}^*(\log C_d)\otimes\OO_{\PP^2}(-t),
 \qquad t>0,
\]
in the compact and logarithmic settings, respectively.  The idea goes
back to Bloch~\cite{Bloch26} and to
Green and Griffiths~\cite{Green-Griffiths1980}, and takes its invariant
form in Demailly's work~\cite{Demailly1997}.  Let $X_1$ and $X_2$
be the first two levels of the Demailly--Semple tower of $X$, or of the
logarithmic tower of $(\PP^2,C_d)$, respectively.  Write $\pi_{2,0}$
for the projection from $X_2$ to $X$ or $\PP^2$.  Demailly's direct-image
theorem~\cite[Theorem~6.8(ii) and formula~(10.3)]{Demailly1997} gives
\[
 (\pi_{2,0})_*\OO_{X_2}(m)\simeq
 \begin{cases}
 E_{2,m}T_X^* &\text{in the compact case},\\
 E_{2,m}T_{\PP^2}^*(\log C_d) &\text{in the logarithmic case}.
 \end{cases}
\]
By the projection formula, a nonzero section $\omega_1$ of either
negatively twisted bundle above corresponds to a nonzero section
$\widetilde\omega_1$ of
\[
 \OO_{X_2}(m)\otimes\pi_{2,0}^*\OO_X(-t)
 \quad\text{or}\quad
 \OO_{X_2}(m)\otimes\pi_{2,0}^*\OO_{\PP^2}(-t),
\]
respectively.

The fundamental vanishing theorem for entire curves, due to
Green--Griffiths, Siu--Yeung and
Demailly~\cite{Green-Griffiths1980,SiuYeung1996Bloch,Demailly1997},
implies $f^*\omega_1\equiv0$ for every nonconstant entire curve
$f\colon\CC\to X$ (respectively, $f\colon\CC\to\PP^2\setminus C_d$).
Write $f_{[i]}\colon\CC\to X_i$ for its lifted curves, $i=1,2$.
The direct-image correspondence then gives
\[
 f_{[2]}(\CC)\subset\{\widetilde\omega_1=0\}\subset X_2.
\]

For the differentials selected for the geometric argument, the zero
divisor of $\widetilde\omega_1$ has a unique irreducible component $Z$
other than the vertical divisor $\Gamma_2$.  This component dominates
$X_1$ and contains $f_{[2]}(\CC)$.

The strategy seeks a further nontrivial constraint on the lift
$f_{[2]}$: an equation satisfied by the lift that does not hold
identically on $Z$.  Its solution set in $Z$ has codimension at least
two in $X_2$ away from $\Gamma_2$.  This constraint may come from a
section defined only on $Z$ or from a second global jet differential on
the base.  In the degree ranges considered here, a Riemann--Roch
argument supplies a negatively twisted differential, but does not by
itself control its zero locus or furnish an independent second constraint.

Constructing a second independent equation has been a central difficulty
in the jet-differential approach to the hyperbolicity of surfaces.
McQuillan~\cite{Mcquillan1999}
proved hyperbolicity for very general surfaces of degree at least $36$.
Demailly and El~Goul~\cite{Demailly-Elgoul2000} lowered the bound
to $21$ with Miyaoka's $2$-jet zero-locus method~\cite{Miyaoka} (see
also~\cite{Lu-Yau}), and P\u{a}un~\cite{mihaipaun2008} reached $18$ with
Siu's slanted vector fields~\cite[Section~2.2]{Siu2004}.  These three
bounds concern very general surfaces.  In the logarithmic setting, Siu
and Yeung~\cite{siu_yeung1996} proved hyperbolicity for complements of
generic plane curves of sufficiently large degree; for very general
plane curves, El~Goul~\cite{ElGoul2003} reached degree $15$ and
Rousseau~\cite{Rousseau2009} reached degree $14$.  Together, these works
developed complementary geometric and jet-theoretic tools for the
compact and logarithmic settings.

To control the weights and twists of factors of a negatively
twisted jet differential, the argument uses vanishing statements of the
form
\[
 H^0\bigl(X,E_{2,m}T_X^*\otimes\OO_X(-t)\bigr)=0
\]
for suitable pairs $(m,t)$ of positive integers, together with analogous
logarithmic statements.  Here $m$ is the weight and $-t$ is the twist.
The \emph{key vanishing lemmas} establish such
vanishings for the finitely many pairs $(m,t)$ needed in the argument,
thereby excluding negatively twisted factors of the corresponding
weights and twists.  Hou, Huynh, Merker and Xie recast such
vanishing problems as finite algebraic problems~\cite{Hou-Huynh-Merker-Xie2025}.
In the sequel~\cite{HHMX2026} they prove hyperbolicity for very general
surfaces of degree at least $17$ and for complements of very general
smooth plane curves of degree at least $12$.  Within the general-type
range, the preceding results cover all degrees for which the leading
coefficient of the invariant $2$-jet differential Riemann--Roch calculation is
positive, except $d=15,16$ in the compact case and $d=11$ in the
logarithmic case.  For these remaining degrees, we prove hyperbolicity
for general surfaces in $\PP^3$ and for complements of general smooth
plane curves, respectively; see Theorem~\ref{thm:main-thresholds}.

The proofs combine three ingredients.  The first is an independence
theorem for pole-three slanted vector fields.
Applied to the selected families of $\omega_1$, these fields produce a
second negatively twisted $2$-jet differential whose equation is
nontrivial on $Z$ when $m\geqslant3$ and $t\geqslant4$.  It supplies
the additional constraint on $f_{[2]}$ discussed above; see
Theorems~\ref{thm:apply O(3) fields cpt}
and~\ref{thm:apply O(3) fields log}.  The independence argument extends
those of P\u{a}un~\cite[pp.~886--889]{mihaipaun2008} and
Rousseau~\cite[Proposition~17]{Rousseau2009}, which apply for
$3\leqslant m\leqslant5$.  For related low-pole fields in the case of two
boundary curves, see~\cite{HouWangXie2026}.

The second ingredient is a symmetry method for proving the key
vanishing lemmas.  On suitable special fibres, an involution and, in
compact degrees $15$ and $16$, a commuting action by roots of unity decompose
$H^0\bigl(X,E_{2,m}T_X^*\otimes\OO_X(-t)\bigr)$ and its logarithmic
analogue into eigenspaces.  The resulting vanishing questions
are expressed as finite linear systems, whose ranks are computed over
finite fields.  Upper semicontinuity then transfers the vanishing to
general parameters.

The third ingredient consists of relative $1$-jet locus theorems.  In
the compact case, the loci contain the first lifts of nonconstant entire
curves; in the logarithmic case, they contain the logarithmic tangent
lifts of rational curves with at most two points over the boundary.
These loci provide the input for the boundedness method of Xie and
Zhao~\cite{XieZhao2026}, which allows us to pass from very general to
general parameters.

\subsection{Main results}
\label{sec:main-results}

\begin{thm}\label{thm:main-thresholds}
\leavevmode
\begin{enumerate}
\item Every general surface $X\subset\PP^3$ of degree $15$ or $16$ is
Kobayashi hyperbolic.
\item The complement $\PP^2\setminus C$ of a general smooth plane
curve $C\subset\PP^2$ of degree $11$ is Kobayashi hyperbolic.
\end{enumerate}
\end{thm}

For $d\geqslant17$ hyperbolicity of general surfaces follows from the
relative $1$-jet locus of~\cite[Theorem~1.5 and Remark~7.5]{HHMX2026}
together with Xie--Zhao's Theorems~B and~D~\cite{XieZhao2026}.
Combined with Theorem~\ref{thm:main-thresholds}, this gives the compact
range $d\geqslant15$.

The proof of Theorem~\ref{thm:main-thresholds} combines the geometric tools described below
with the finite vanishing statements in Lemma~\ref{lem:finite-kvl-main}.
Fix a weight $m$ and a twist $-t$, and call the quotient $t/m$ the
\emph{slope}.  The
zero-locus inequality on which the proof rests gives positivity only
when $t/m$ lies below a threshold $\rho_d$.  In our three cases
\[
 \rho_{16}=\frac58,\qquad \rho_{15}=\frac{17}{66},\qquad
 \rho_{11}^{\log}=\frac{13}{96}.
\]
Below the threshold, Miyaoka's zero-locus method in the form developed by
Demailly--El~Goul, together with its logarithmic analogue, supplies a
second constraint on the relevant zero divisor; see~\cite[Propositions~6.1
and~6.5]{HHMX2026}.  At slope at least $\rho_d$, we use a slanted vector
field to obtain a second global jet differential when its twist remains
negative.  Under the hypotheses of the corresponding independence
results, differentiating by a suitable field of pole order $k$ produces
a section of twist $-(t-k)$, which is negative when $t\geqslant k+1$.
The fields themselves are constructed independently of $t$.
Differentiation by the pole-three fields preserves negative twist when
$t\geqslant4$.  The required independence follows from
Theorems~\ref{thm:apply O(3) fields cpt}
and~\ref{thm:apply O(3) fields log} under their respective hypotheses.
Vanishing at twist $-s$ implies vanishing at every twist $-t$
with $t\geqslant s$, so for each weight it is enough to rule out the
least value of $t$ for which $t/m\geqslant\rho_d$,
\[
 t_m=\lceil\rho_d m\rceil .
\]
The remaining cases have $t_m\leqslant3$ and are covered by the key
vanishing lemma below.  Some degree-$15$ vanishings are stated at twist
orders below $t_m$ and extend to $(m,t_m)$ by monotonicity.

\begin{lem}[Key Vanishing Lemma]\label{lem:finite-kvl-main}
The following vanishings hold.
\begin{enumerate}
\item If $X\subset\PP^3$ is a general surface of degree $d\in\{15,16\}$, then
\[
H^0\!\left(X,E_{2,m}T_X^*\otimes\OO_X(-t)\right)=0
\]
for the following pairs:
\begin{equation}\label{eq:compact16-targets}
 (m,t)=(3,2),(4,3) \qquad\text{if } d=16,
\end{equation}
and
\begin{equation}\label{eq:compact15-targets}
 \begin{aligned}
 (m,t)={}&(3,1),(4,1),(5,1),(6,1),(7,2),\\
          &(8,3),(9,3),(10,3),(11,3) \qquad\text{if } d=15.
 \end{aligned}
\end{equation}
\item If $C\subset\PP^2$ is a general smooth curve of degree $11$, then
\[
H^0\!\left(\PP^2,E_{2,m}T_{\PP^2}^*(\log C)\otimes\OO_{\PP^2}(-t)\right)=0
\]
for $3\leqslant m\leqslant7$ and $t=1$, for
$8\leqslant m\leqslant14$ and $t=2$, and for
$15\leqslant m\leqslant22$ and $t=3$.
\end{enumerate}
\end{lem}

The logarithmic case has a Nevanlinna-theoretic consequence.
Here $T_f(r)$ denotes the Nevanlinna characteristic, $N_f^{[1]}(r,C)$ the
level-one truncated counting function, and $\|$ indicates that the
estimate holds outside a set of finite Lebesgue measure.

\begin{thm}
\label{thm:log-11-endpoint}
For a general smooth plane curve $C\subset\PP^2$ of degree $11$ and every
algebraically nondegenerate entire curve $f\colon\CC\to\PP^2$, one has
\[
 T_f(r)\leqslant1422\,N_f^{[1]}(r,C)
          +o\bigl(T_f(r)\bigr)\ \|.
\]
\end{thm}

Theorem~\ref{thm:log-11-endpoint} applies to algebraically nondegenerate
entire curves for a general boundary.  For a very general boundary,
combining it with Lemma~\ref{lem:log-algebraically-degenerate}
(\cite[Lemma~2.8]{HHMX2026}), which
treats the algebraically degenerate case, gives the following estimate
for all nonconstant entire curves whose images are not contained in the
boundary.

\begin{cor}
\label{cor:log-11-very-general}
For a very general smooth plane curve $C\subset\PP^2$ of degree $11$,
every nonconstant entire curve $f\colon\CC\to\PP^2$ with
$f(\CC)\not\subset C$ satisfies
\[
 T_f(r)\leqslant1422\,N_f^{[1]}(r,C)
          +o\bigl(T_f(r)\bigr)\ \|.
\]
\end{cor}

For general smooth plane curves of degree
$d\geqslant12$ the preceding paper~\cite[Theorem~1.9]{HHMX2026} proves a
Second Main Theorem of the same kind, with a constant depending on $d$.

The assertions of Theorem~\ref{thm:main-thresholds} are proved
first for very general parameters.  Passing from very general to general
parameters uses the following two relative $1$-jet locus theorems.  The
compact locus contains the first lifts of all entire curves.  The
logarithmic locus contains the tangent lifts of rational curves that
meet the boundary in at most two points.

\begin{thm}
\label{thm:compact-general-upgrade}
For each $d\in\{15,16\}$, there is a nonempty Zariski-open subset
$U_d$ of the parameter space of smooth degree-$d$ surfaces, with
universal family $p_d\colon\mathcal X_d\to U_d$, and a closed algebraic
subset
\[
 \mathcal S_d\subset\PP(T_{\mathcal X_d/U_d})
\]
which is flat and projective over $U_d$ and satisfies
$\dim(\mathcal S_d)_a\leqslant2$ for every $a\in U_d$.
For every nonconstant entire curve $f\colon\CC\to X_a$, its first lift
satisfies $f_{[1]}(\CC)\subset(\mathcal S_d)_a$.
\end{thm}

\begin{thm}
\label{thm:log-general-upgrade}
There is a nonempty Zariski-open subset $U_{11}$ of the space of smooth
degree-$11$ plane curves, with universal boundary
$\mathcal C\subset\PP^2\times U_{11}$, and a closed algebraic subset
\[
 \mathcal S_{11}^{\log}\subset
 \PP\!\left(T_{(\PP^2\times U_{11})/U_{11}}(-\log\mathcal C)\right)
\]
which is flat and projective over $U_{11}$ and satisfies
$\dim(\mathcal S_{11}^{\log})_a\leqslant2$ for every $a\in U_{11}$.
If $\Gamma\not\subset C_a$ is an integral rational curve with
normalization $\nu\colon\PP^1\to\Gamma$ and
$\#\operatorname{Supp}\nu^*C_a\leqslant2$, its logarithmic tangent lift
is contained in $(\mathcal S_{11}^{\log})_a$.
\end{thm}

\subsection{Outline of the proofs}
\label{sec:outlook}
\label{sec:symmetry-idea}

The geometric argument follows the strategy described in detail
in~\cite{HHMX2026}; its central step is the construction of a second
independent equation satisfied by the second lift $f_{[2]}$ of an entire
curve $f$.  We begin with a family of negatively
twisted $2$-jet differentials supplied
by~\cite[Propositions~5.6 and~5.7]{HHMX2026}.  A member $\omega_1$ of
weight $m$ and twist $-t$ determines the section $\widetilde\omega_1$
on $X_2$ and its residual divisor $Z$ introduced in
Section~\ref{sec:motivation}.

The slope determines which mechanism applies.  If $t/m<\rho_d$, we use
Miyaoka's zero-locus method~\cite{Miyaoka}, as developed by Demailly and
El~Goul~\cite[Proposition~3.4]{Demailly-Elgoul2000} and adapted to the
logarithmic setting by El~Goul~\cite[Theorem~1.3.3]{ElGoul2003}.  The
compact and logarithmic bigness criteria used here are
Propositions~\ref{prop:compact-zero-locus}
and~\ref{prop:log-restricted-bigness}, respectively, quoted
from~\cite[Propositions~6.1 and~6.5]{HHMX2026}.  Their bigness
conclusions on $Z$ yield a negatively twisted section on $Z$, which
supplies the additional constraint.  If $t/m\geqslant\rho_d$, the extra
equation is a second twisted differential $\omega_2$ obtained from a slanted vector field,
and its zero locus meets $Z$ in codimension at least two away from the
exceptional divisor $\Gamma_2\subset X_2$.  In either case the second
lift of the curve is confined to a surface in $X_2$, and its first lift
to the image of that surface in $X_1$.  These images are the relative
$1$-jet loci of Theorems~\ref{thm:compact-general-upgrade}
and~\ref{thm:log-general-upgrade}.

At the two thresholds, suppose one used only a pole-seven field, as
in~\cite[Propositions~5.8 and~5.9]{HHMX2026}.  The last weights with
$t_m\leqslant7$ are $27$ in compact degree $15$ and $51$ in logarithmic
degree $11$.  The resulting systems at $(27,7)$ and $(51,7)$ are not
solved in this paper; exact elimination at their scale may be very
difficult even with highly parallel computation.

The difficulty is resolved by lowering the pole order in Siu's slanted
vector field method~\cite[Section~2.2]{Siu2004}.  Differentiation by
pole-three fields preserves negative twist when $t\geqslant4$, but the
prior independence arguments of
P\u{a}un~\cite{mihaipaun2008} in the compact case and
Rousseau~\cite[Proposition~17]{Rousseau2009} in the logarithmic case
apply in the range $3\leqslant m\leqslant5$.  Extending these arguments
to all weights $m\geqslant3$ in the families of Section~\ref{sec:o3-low-pole}
(Theorems~\ref{thm:apply O(3) fields cpt}
and~\ref{thm:apply O(3) fields log}) removes the obstruction.  Under the
stated hypotheses, the resulting fields give an independent equation for
$m\geqslant3$ and $t\geqslant4$.  Consequently, in the large-slope range,
the only cases requiring a separate vanishing argument have
$t\leqslant3$.  The remaining weights are $3$ and $4$ for surfaces of
degree $16$, $3\leqslant m\leqslant11$ for surfaces of degree $15$, and
$3\leqslant m\leqslant22$ for complements of plane curves of degree $11$.
Twist monotonicity reduces these cases to Lemma~\ref{lem:finite-kvl-main}.
In degrees $15$ and $11$, before imposing the symmetry
equations and the holomorphicity conditions at infinity, the largest
candidate lists for the
pole-seven route have about $78$ million and $34$ million coefficients
in the compact and logarithmic cases, respectively, compared with
$2.7$ million and $1.4$ million for the pole-three route; see
\eqref{eq:o3-compact-variable-count} and~\eqref{eq:o3-log-variable-count}.
The low-pole
independence theorem reduces the range of cases to be checked, symmetry
decomposes the resulting systems, and parallel computations complete the
remaining finite verification.

The vanishing calculations are performed on smooth surfaces and plane
curves defined by two-term perturbations of Fermat equations, invariant
under the involution exchanging two coordinates
(Lemmas~\ref{lem:compact-special-fibres} and~\ref{lem:log-special-fibre}).
For a fixed weight and negative twist, we express a hypothetical invariant
(logarithmic) $2$-jet differential locally in terms of jet variables and
polynomial coefficient numerators in the affine coordinates.
The degree bounds in Subsections~\ref{sec:log-degree}
and~\ref{sec:cpt-degree} leave finitely many unknown coefficients.
The involution acts on the space of negatively twisted invariant
(logarithmic) $2$-jet differentials, splitting it into $\pm1$ eigenspaces.
We represent an eigensection by an untwisted differential obtained by
multiplication by a power of a coordinate function.  Since the involution
exchanges this coordinate with another, the untwisted representative
satisfies the covariance identity~\eqref{eq:log-twist-covariance} or
\eqref{eq:compact-twist-covariance}.  Comparing the jet polynomial on
two coordinate charts gives the coefficient equations
\eqref{eq:log-K-symmetry} and~\eqref{eq:coefficient_match}.  Together with
holomorphicity at infinity (see Subsection~\ref{sec:log-computation}
and~\eqref{eq:trunc}), these give a homogeneous linear system in the
unknown coefficients.  After denominators are cleared, the system has
integer coefficients; see Sections~\ref{sec:log-computation}
and~\ref{sec:cpt-computation}.
For a compact surface of degree $d\in\{15,16\}$, each $d$-th root of
unity $\zeta$ acts by $(x,y,z)\mapsto(\zeta x,y,z)$, preserving the
chosen equation~\eqref{eq:compact-special-polynomial}; see
\eqref{eq:compact-cyclic-action}.  This action commutes with the involution
and decomposes each involution eigenspace into $d$ character spaces.
The corresponding congruence condition on the coefficients
\eqref{eq:compact-cyclic-ansatz} reduces the size of each system.
Appendices~\ref{app:compact-computations} and~\ref{app:log-computations}
report the sparse modular rank computations over $\FF_{32749}$ and
describe the GPU-accelerated Wiedemann workflow with randomized
preconditioners and projections.\footnote{\label{fn:computational-resources}
The code used for the computations is available at
\url{https://github.com/liujjmath/comp_hyper/tree/master/user}.
The supporting computational evidence is available at
\url{https://github.com/LeiHou-code/Hyperbolicity-dimension-2-threshold-evidence}.}
If a matrix has full column rank
modulo a good prime, a maximal minor is nonzero over $\QQ$, so its kernel vanishes
over $\CC$.  Upper semicontinuity transfers the vanishing from these
smooth examples to general parameters.

The proof of the Second Main Theorem follows the strategy
of~\cite[Section~6.5]{HHMX2026} and requires vanishing for slopes
$t/m>2/15$.
For $t=1,2,3$, this requires the same vanishing results as the slope
condition $t/m\geqslant13/96$.  Thus the logarithmic clause of
Lemma~\ref{lem:finite-kvl-main} also suffices for the Second Main Theorem.

A very general surface of degree $15$ or $16$ contains no rational or
elliptic curves by the Clemens--Xu theorem~\cite{Clemens1986,Xu1994}.
Together with the jet-differential arguments above and McQuillan's
theorem~\ref{thm:McQuillan}, this gives the compact assertion of
Theorem~\ref{thm:main-thresholds} for very general surfaces.  In the
logarithmic case, the Second Main Theorem and the exclusion of rational
curves with at most two points over a very general degree-$11$ boundary
\cite[Lemma~2.8]{HHMX2026} give the corresponding assertion for very
general plane curves.

To establish the main theorem for general parameters, the relative
$1$-jet loci of
Theorems~\ref{thm:compact-general-upgrade}
and~\ref{thm:log-general-upgrade} yield relative symmetric differentials
whose zero loci contain the lifts of the relevant curves; see
Section~\ref{sec:compact-endpoints-proof} and
\cite[Proposition~7.3.1]{XieZhao2026}.  Applying the boundedness results
of Xie--Zhao~\cite[Theorem~6.1.1 and Sections~7.2--7.4]{XieZhao2026}
gives uniform degree bounds for rational and elliptic curves in the
compact case and for rational curves with at most two points over the
boundary in the logarithmic case.
In the compact case, the relative symmetric differentials and the absence
of rational and elliptic curves on very general surfaces verify the
hypotheses of Theorem~B of~\cite{XieZhao2026}.  Theorem~D of the same
paper gives hyperbolicity on a nonempty Zariski-open set.  In the
logarithmic case, the degree bounds reduce the incidence argument to
finitely many constructible families.  None has dense image in the
parameter space, since a dense image would contain a very general
parameter, which the exclusion above rules out.  Removing the closures
of these images leaves the nonempty Zariski-open set on which the
logarithmic theorem holds.

Below these positivity thresholds, extending the present argument with
invariant $2$-jets would require a separate construction or estimate
for negatively twisted sections on general fibres.  Higher-order jets,
in particular $3$-jets, could also be used if suitable existence and
key vanishing results can be established.

\subsection{Structure of the paper}
\label{sec:structure}

The paper is organized as follows.
Sections~\ref{sec:o3-low-pole}--\ref{sec:compact-symmetry-proof} contain
the computational core.  Their main statements are used in
Section~\ref{sec:hyperbolicity}, which contains the geometric argument.
Section~\ref{sec:o3-low-pole} constructs the low-pole slanted vector
fields and proves the independence theorems in the compact and
logarithmic settings.
Section~\ref{sec:log-symmetry-proof} proves the logarithmic key vanishing
lemma by symmetry reduction and exact computation, and
Section~\ref{sec:compact-symmetry-proof} treats both compact cases.
The degree-$16$ and degree-$15$ character blocks are analysed in
Sections~\ref{sec:cpt16-computation} and~\ref{sec:cpt15-certificates},
respectively.
Section~\ref{sec:hyperbolicity} proves the relative $1$-jet locus
theorems, the hyperbolicity theorems, and the degree-$11$ Second Main
Theorem.  Appendix~\ref{sec:sl2-preliminaries} collects the
$\mathfrak{sl}_2$ representation results used in the independence proof.
Appendices~\ref{app:compact-computations} and~\ref{app:log-computations}
describe the computational methods and modular rank results for the
compact and logarithmic cases, respectively.

\section{Low-pole Slanted Vector Fields}
\label{sec:o3-low-pole}

We first construct the compact pole-three slanted vector fields and prove
the independence of the resulting jet differentials.  We then treat their
logarithmic counterparts in the planar and cyclic coordinates, referring
explicitly to the compact calculations whenever the same argument applies.

\Needspace{8\baselineskip}
\subsection{Compact Case}
\label{sec:o3-compact}

\subsubsection{The compact universal surface and its vertical $2$-jets}

Consider the {\sl universal hypersurface}\, $\mathcal{ X} \subset \mathbb{P}^{3} \times \mathbb{P}^{ N_d^3}$ parametrizing
all algebraic surfaces of a fixed degree $d
\geqslant 1$ in $\mathbb{P}^{3}$.  In the homogeneous coordinates
\[
\aligned
\left[Z\right]
&
=
[Z_0:Z_1:Z_2:Z_{3}]
\in\mathbb{P}^{3}
\\
[A]
&
=
\big[(A_\alpha)_{\alpha\in\mathbb{N}^{4},\,\vert\alpha\vert=d}\big]
\in\mathbb{P}^{N_d^3},
\endaligned
\]
it is the zero locus of the general homogeneous degree $d$ polynomial:
\[
\mathcal{X} =
\Big\{
0
=
\sum_{\substack{\alpha\in\mathbb{N}^{4}\\ \vert\alpha\vert=d}}\,
A_\alpha\,Z^\alpha
\Big\},
\]
where a multiindex $\alpha = ( \alpha_0, \alpha_1, \alpha_2, \alpha_{3}) \in
\mathbb{N}^{4}$ has {\sl length} defined by $\vert \alpha \vert := \alpha_0 + \alpha_1 + \alpha_2 + \alpha_{3}$ and we abbreviate $Z^\alpha = Z_0^{\alpha_0} Z_1^{\alpha_1} Z_{2}^{\alpha_2} Z_{3}^{ \alpha_{3}}$.

Let \(N_d^3 \coloneqq \binom{d+3}{3}-1\). We have two canonical projections:
\begin{equation*}
    \xymatrix{
     & \mathcal{X} \ar[dl]_{\mathrm{pr}_1} \ar[dr]^{\mathrm{pr}_2} & \\
    \mathbb{P}^3 & & \mathbb{P}^{N_d^3} 
    }
\end{equation*}
Now, let $J_2 (\mathcal{X}) \to \mathcal{X}$ be the bundle of $2$-jets of germs of holomorphic curves in $\mathcal{X}$.
The space of \emph{vertical $2$-jets} is the subspace 
$$J_2^{v} (\mathcal{X}) \coloneqq  \ker\left(\mathrm{pr}_2\right)_* \subset J_2 (\mathcal{X}),$$ 
consisting of jets tangent to the fibres of the second projection $\mathrm{pr}_2$.

As in~\cite{mihaipaun2008,Rousseau2009}, we shall mainly work in inhomogeneous coordinates on the chart $\{ Z_0 \neq 0 \} \times \{
A_{0d00} \neq 0\}$, a copy of $\mathbb{C}^{3} \times \mathbb{C}^{ N_d^3}$.  Dividing by $(Z_0)^d$ and by $A_{0 d 0 0}$, and setting $z_i := Z_i / Z_0$, the equation of
$\mathcal{ X}$ then transfers to:
\[
\mathcal{X}_0 =
\Big\{
0
=
z_1^d
+
\sum_{\substack{\alpha\in\mathbb{N}^{3}\\ \vert\alpha\vert\leqslant d,\,
\alpha_1<d}}\,\,
a_\alpha\,z^\alpha
\Big\},
\]
with new coefficients $a_{ \alpha_1 \alpha_2 \alpha_{3}} := {
A_{ \alpha_0 \alpha_1 \alpha_2 \alpha_{3}}}/{ A_{ 0 d 0 0}}$ in
which $\alpha_0 := d - \alpha_1 - \alpha_2 - \alpha_{3}$. By
convention, we shall set $a_{ d 0 0} = 1$.

On the chart $\{ Z_0 \neq 0 \} \times \{
A_{0d00} \neq 0\}$, write
\[
 F(z,a)=
 z_1^d
+
\sum_{\substack{\alpha\in\mathbb{N}^{3}\\ \vert\alpha\vert\leqslant d,\,
\alpha_1<d}}\,\,
a_\alpha\,z^\alpha.
\]
The projection \(\mathrm{pr}_2 |_{\mathcal{X}_0} \colon \mathcal{X}_0\to \{A_{0d00} \neq 0\}\subset\PP^{N_d^3}\) sends \((z,a)\) to \(a\). Then a vertical $2$-jet is represented by
\[
  (z,a,\xi^{(1)},\xi^{(2)}) \in\mathbb{C}^{3}\times\mathbb{C}^{N_d^3}\times\mathbb{C}^{3}\times\mathbb{C}^{3},
\]
where
\begin{equation*}
  \xi^{(1)}=(\xi^{(1)}_1,\xi^{(1)}_2,\xi^{(1)}_3) \coloneqq (z_1' , z_2' , z_3'), \quad
  \xi^{(2)}=(\xi^{(2)}_1,\xi^{(2)}_2,\xi^{(2)}_3) \coloneqq (z_1'' , z_2'' , z_3'').
\end{equation*}
With this convention,
the vertical $2$-jet space is cut out by
\begin{align}
 E_0&=F=0, \label{eq:compact-E0}\\
 E_1&=\sum_jF_j\xi^{(1)}_j=0, \label{eq:compact-E1}\\
 E_2&=\sum_jF_j\xi^{(2)}_j+
       \sum_{j,k}F_{jk}\xi^{(1)}_j\xi^{(1)}_k=0. \label{eq:compact-E2}
\end{align}
Here \(F_j=\partial F/\partial z_j\) and
\(F_{jk}=\partial^2F/\partial z_j\partial z_k\).

\begin{rem}
Equations \eqref{eq:compact-E0}--\eqref{eq:compact-E2} are obtained by
differentiating \(F(z(t),a)=0\) once and twice while keeping \(a\) fixed.
\end{rem}

\subsubsection{The tangency test}

Consider an ambient vector field
\[
 V=\sum_\alpha v_\alpha\frac{\partial}{\partial a_\alpha}
   +\sum_j u_j\frac{\partial}{\partial z_j}
   +\sum_{r=1}^2\sum_jw_j^{(r)}
       \frac{\partial}{\partial\xi_j^{(r)}}.
\]
The parameter sum runs over \(|\alpha|\le d\), \(\alpha_1<d\), since
\(a_{d00}=1\). The field is tangent to the vertical \(2\)-jet space
defined by \eqref{eq:compact-E0}--\eqref{eq:compact-E2} if
\(V(E_i)\in(E_0,E_1,E_2)\) for \(i=0,1,2\).

For later use, associate with the parameter part of \(V\) the polynomial
\[
  G_V(y;z,a)=\sum_\alpha v_\alpha(z,a)y^\alpha,
  \qquad y=(y_1,y_2,y_3).
\]
Here is the full calculation behind that shorthand.  Regard the differential
of \(F\) as a symmetric multilinear form and abbreviate
\(G=G_V(\,\cdot\,;z,a)\).  The vector $u\coloneqq(u_1 , u_2 , u_3)$
is the base-coordinate component of $V$.  Direct differentiation gives
\begin{align}
 V(E_0)
 &=G(z)+\mathrm{d}_z F(u),                                      \label{eq:full-tan-0}\\
 V(E_1)
 &=\mathrm{d}_y G(z)(\xi^{(1)})
   +\mathrm{d}_z^2F(u,\xi^{(1)})+\mathrm{d}_z F(w^{(1)}),              \label{eq:full-tan-1}\\
 V(E_2)
 &=\mathrm{d}_y G(z)(\xi^{(2)})
   +\mathrm{d}_y^2 G(z)(\xi^{(1)},\xi^{(1)})
   +\mathrm{d}_z^2 F(u,\xi^{(2)})
   +\mathrm{d}_z^3 F(u,\xi^{(1)},\xi^{(1)}) \label{eq:full-tan-2}\\
 &\hspace{2.2em}
   +\mathrm{d}_z F(w^{(2)})
   +2\mathrm{d}_z^2 F(w^{(1)},\xi^{(1)}). \notag
\end{align}
Thus every later tangency assertion can be checked by substituting three
Taylor coefficients into
\eqref{eq:full-tan-0}--\eqref{eq:full-tan-2}. 

If \(V\) has only a parameter component, the sufficient conditions
\begin{equation}\label{condition:only parameter component}
     G_V(z;z,a)=0,\qquad \mathrm{d}_yG_V(z;z,a)=0,\qquad \mathrm{d}_y^2G_V(z;z,a)=0
\end{equation}
make \(V\) tangent to all three jet equations, independently of
\(\xi^{(1)}\) and \(\xi^{(2)}\).

\begin{rem}\label{rem:invariance of fields cpt}
The fields below commute with source reparametrization by \(\mathbb G_2\)
as in \cite[Section~4.3]{Merker2009}. Each has weight zero for
\(\operatorname{wt}(\xi^{(r)})=r\), so differentiation preserves
invariant jet differentials of weighted degree \(m\).
\end{rem}

\subsubsection{Finite differences of arbitrary order and their pointwise span}
\label{subsubsec:first package cpt}

For multi-indices \(\alpha,\gamma\in\NN^3\), assume that
\(\alpha\ge\gamma\) componentwise, \(|\alpha|\le d\), and
\(\alpha_1<d\). For
\(q=|\gamma|\ge3\), define
\begin{equation}
 D_{\alpha,\gamma}^{[q]}
 =\sum_{\beta\le\gamma}(-1)^{|\beta|}
    \binom{\gamma}{\beta}z^\beta
    \frac{\partial}{\partial a_{\alpha-\beta}}.
 \label{eq:finite-difference}
\end{equation}
Notice that the assumptions \(\alpha\ge\gamma\) and \(|\alpha|\le d\)
automatically imply \(d\ge q\).

\begin{prop}[{\cite[Section~3]{Merker2009}}]
\label{prop:first-package}
For indices satisfying the conditions above, the vector field
\(D_{\alpha,\gamma}^{[q]}\) is tangent to the vertical \(2\)-jet
equations. Its coefficients have degree at most \(q\) in \(z\), and it
homogenizes as a section with an \(\mathcal{O}_{\PP^3}(q)\)-twist.
\end{prop}

\begin{proof}
Introduce a dummy variable \(y\), distinct from the base point \(z\).
Applying \eqref{eq:finite-difference} to the universal polynomial in the
\(y\)-variable gives
\[
 D_{\alpha,\gamma}^{[q]}\bigl(F(y,a)\bigr)
 =\sum_{\beta\le\gamma}(-1)^{|\beta|}
    \binom{\gamma}{\beta}z^\beta y^{\alpha-\beta}
 =y^{\alpha-\gamma}(y-z)^\gamma.
\]
Since \(q=|\gamma|\ge3\), the polynomial on the right-hand side and all
of its \(y\)-derivatives of order at most \(2\) vanish at \(y=z\).
Consequently, condition
\eqref{condition:only parameter component} is satisfied, which proves
the required tangency.

The coefficients in \eqref{eq:finite-difference} have \(z\)-degree at
most \(|\gamma|=q\). After homogenization, the vector field therefore
acquires an \(\mathcal{O}_{\PP^3}(q)\)-twist.
\end{proof}

\begin{exa}
For \(q=3\), the possible shapes of \(\gamma\), up to permutation, are
\[
 (3,0,0),\qquad (2,1,0),\qquad (1,1,1).
\]
For instance, if \(\alpha\ge(3,0,0)\), \(|\alpha|\le d\), and
\(\alpha_1<d\), then
\[
 D_{\alpha,(3,0,0)}^{[3]}
 =\partial_{a_\alpha}
  -3z_1\partial_{a_{\alpha-e_1}}
  +3z_1^2\partial_{a_{\alpha-2e_1}}
  -z_1^3\partial_{a_{\alpha-3e_1}},
\]
which is P\u aun's vector field \(V_\alpha^{300}\).
\end{exa}

We next identify the parameter directions generated by the
\(q\)-th-difference package at a fixed base point. Fix
\(z\in\CC^3\), and let \(\mathfrak{m}_z\) denote the maximal ideal of
\(z\) in the polynomial ring \(\CC[y_1,y_2,y_3]\).
For a polynomial \(G\) in \(y\), write \([y^\alpha]G\) for the coefficient
of the monomial \(y^\alpha\) in \(G\), with any parameters held fixed.
Thus \(\mathbb C[y_1,y_2,y_3]_{\le d,\,[y_1^d]=0}\) denotes the space
of polynomials \(G\) of degree at most \(d\) satisfying \([y_1^d]G=0\).

\begin{lem}
\label{lem:difference-spans}
For every \(3\le q\le d\), the parameter variations associated with the
fields \(D_{\alpha,\gamma}^{[q]}\), evaluated at the fixed point \(z\),
span
\[
 \mathfrak{m}_z^q\cap
 \mathbb C[y_1,y_2,y_3]_{\le d,\,[y_1^d]=0}.
\]
\end{lem}

\begin{proof}
The proof of Proposition~\ref{prop:first-package} shows that the
polynomials associated with these parameter variations are precisely
\[
 y^\delta(y-z)^\gamma,
 \qquad |\gamma|=q,\qquad |\delta|+q\le d,
 \qquad (\delta+\gamma)_1<d,
\]
where \(\delta=\alpha-\gamma\). These polynomials all lie in
\(\mathfrak m_z^q\cap
\mathbb C[y_1,y_2,y_3]_{\le d,\,[y_1^d]=0}\).

On the other hand, the Taylor monomials \((y-z)^\eta\) with
\(q\le|\eta|\le d\) form a basis of
\(\mathfrak m_z^q\cap\mathbb C[y_1,y_2,y_3]_{\le d}\). Among them,
only \((y_1-z_1)^d\) has a nonzero coefficient of \(y_1^d\), and that
coefficient is one. Hence the Taylor monomials
\[
 (y-z)^\eta,\qquad q\le|\eta|\le d,\qquad \eta_1<d,
\]
form a basis of \(\mathfrak m_z^q\cap
\mathbb C[y_1,y_2,y_3]_{\le d,\,[y_1^d]=0}\).

To express each of these basis elements in terms of the parameter
variations, choose \(\gamma\le\eta\) with \(|\gamma|=q\). Then
\[
 (y-z)^\eta=(y-z)^\gamma
 \prod_{i=1}^3(y_i-z_i)^{\eta_i-\gamma_i}.
\]
Expanding the last product in powers of \(y\), with coefficients
depending on the fixed point \(z\), expresses \((y-z)^\eta\) as a
linear combination of polynomials
\[
 y^\delta(y-z)^\gamma,\qquad \delta\le\eta-\gamma.
\]
For each such \(\delta\), we have \(|\delta|+q\le d\) and
\((\delta+\gamma)_1\le\eta_1<d\). Thus each polynomial in the
expansion is the variation associated with the finite-difference field
\(D_{\delta+\gamma,\gamma}^{[q]}\). The parameter variations therefore
span every basis element of \(\mathfrak m_z^q\cap
\mathbb C[y_1,y_2,y_3]_{\le d,\,[y_1^d]=0}\), proving the lemma.
\end{proof}

In particular, for \(d\ge3\), Lemma~\ref{lem:difference-spans} shows
that the \(q=3\) package spans all parameter variations whose defining
polynomials vanish to order at least three at \(z\). Thus, if a
slanted vector field \(V\) has parameter variation \(G_V\) satisfying
\eqref{condition:only parameter component}, then \(G_V\) belongs to
\(\mathfrak m_z^3\cap
\mathbb C[y_1,y_2,y_3]_{\le d,\,[y_1^d]=0}\). The third-difference
package can therefore be used to cancel the parameter component of
\(V\) at the fixed point \(z\).
It is important, however, that condition
\eqref{condition:only parameter component}, namely \(j_z^2G_V=0\), is
stronger than tangency at a single prescribed vertical \(2\)-jet
\(
 (z,a,\xi^{(1)},\xi^{(2)}).
\)
Indeed, by \eqref{eq:full-tan-0}--\eqref{eq:full-tan-2}, a
parameter-only vector field with associated polynomial \(G\) is tangent
at this jet precisely when
\begin{equation}\label{eq:parameter-only-tangency}
 G(z)=0,
 \qquad
 \mathrm{d}G(z)\bigl(\xi^{(1)}\bigr)=0,
 \qquad
 \mathrm{d}G(z)\bigl(\xi^{(2)}\bigr)
 +\mathrm{d}^2G(z)
   \bigl(\xi^{(1)},\xi^{(1)}\bigr)=0.
\end{equation}
For \(d\ge3\), the space
\(\mathbb C[y_1,y_2,y_3]_{\le d,\,[y_1^d]=0}\) of parameter
variation polynomials contains every polynomial of degree at most two.
The Taylor \(2\)-jets of these polynomials at \(z\) therefore fill the
space of Taylor polynomials of degree at most two,
which has dimension \(1+3+6=10\).

At a nondegenerate jet, \(\xi^{(1)}\ne0\), the three tangency conditions
in~\eqref{eq:parameter-only-tangency} are independent. The first condition involves the value \(G(z)\),
the second involves the first derivative, and the third can be varied
independently through \(\mathrm d^2G(z)(\xi^{(1)},\xi^{(1)})\).
Thus parameter-only tangency cuts out a subspace of codimension three in
the space of Taylor \(2\)-jets.

By Lemma~\ref{lem:difference-spans}, the third-difference package spans
exactly the parameter variations satisfying
\eqref{condition:only parameter component}, namely those with vanishing
Taylor \(2\)-jet at \(z\). Consequently, within the parameter-only tangent directions at the
chosen jet, the span of this package has codimension \(10-3=7\). These
seven missing parameter directions explain why the first package alone
does not provide a full tangent frame.

\subsubsection{The first \texorpdfstring{$\mathcal{O}(3)$}{O(3)} package:
prescribed linear motion of the jet variables}

Fix \(d\ge3\). For a constant matrix \(A=(A_j^k)\in M_3(\CC)\), set
\[
  \widetilde{V}_A
  =\sum_{r=1}^2\sum_j(A\xi^{(r)})_j
       \frac{\partial}{\partial\xi_j^{(r)}},
\]
where \(A\) acts on the column vector \(\xi^{(r)}\). We seek a lift
\[
 V_A=\sum_{\substack{|\alpha|\le d\\\alpha_1<d}}
       v_\alpha(z,a)\frac{\partial}{\partial a_\alpha}
       +\widetilde V_A,\qquad
 v_\alpha(z,a)=\sum_{|\beta|\le3}v^\alpha_\beta(a)z^\beta.
\]
Write \(G_A(y;z,a)=\sum_\alpha v_\alpha(z,a)y^\alpha\) for its
parameter variation. Differentiate in \(y\) with \(z,a\) fixed, then
evaluate at \(y=z\). The coefficientwise tangency conditions are
\begin{align}
 G_A(z;z,a)&=0,                                                       \label{eq:matrix-0}\\
 \partial_{y_j}G_A(z;z,a)+\sum_p F_p(z,a)A_p^j&=0,                  \label{eq:matrix-1}\\
 \partial_{y_j}\partial_{y_k}G_A(z;z,a)
 +\sum_p\bigl(F_{jp}(z,a)A_p^k+F_{pk}(z,a)A_p^j\bigr)&=0
 \quad(1\le j\le k\le3).                                            \label{eq:matrix-2}
\end{align}

\begin{prop}[{\cite[Lemma~1.1]{mihaipaun2008}}]\label{prop:matrix-lift}
For \(d\ge3\) and every constant \(A\in M_3(\CC)\), there is a
parameter variation \(G_A\) satisfying
\eqref{eq:matrix-0}--\eqref{eq:matrix-2}.
Its coefficients have degree at most three in \(z\) and total degree
at most one in the parameters \(a_\alpha\).
The resulting \(V_A\) annihilates the three vertical \(2\)-jet equations
and extends with twist \(\mathcal O_{\PP^3}(3)\). 
\end{prop}

For the construction of the matrix lift, see P\u{a}un~\cite[proof of Lemma~1.1]{mihaipaun2008}.
To obtain its parameter coefficients, put \(B_A(y,z)=A(y-z)\) and set
\[
\begin{aligned}
 G_A^{\mathrm{unif}}(y;z,a)
   &=-\mathrm d_yF(y,a)\bigl(B_A(y,z)\bigr),\\
 G_A(y;z,a)
   &=G_A^{\mathrm{unif}}(y;z,a)
     -\bigl([y_1^d]G_A^{\mathrm{unif}}\bigr)y_1^{d-3}(y_1-z_1)^3.
\end{aligned}
\]
Both polynomials have \(y\)-degree at most \(d\). The only term of
\(G_A^{\mathrm{unif}}\) involving \(z\) is \(\mathrm d_yF(y,a)(Az)\),
of \(y\)-degree at most \(d-1\). Thus the coefficient \([y_1^d]G_A^{\mathrm{unif}}\)
of \(y_1^d\) is independent of \(z\), and \(G_A\) has \(z\)-degree at most three.
Since \(y_1^{d-3}(y_1-z_1)^3\) is monic of degree \(d\) in \(y_1\),
\([y_1^d]G_A=0\). The construction is linear in \(F\), so each
coefficient of \(G_A\) has total degree at most one in the parameters \(a_\alpha\).

The correction vanishes to order three at \(y=z\). For \(u,v\in\CC^3\),
differentiating in \(y\) with \(z,a\) fixed and then setting \(y=z\) gives
\begin{align*}
 G_A(z;z,a)&=0,\\
 \mathrm d_yG_A(z;z,a)(u)&=-\mathrm d_zF(z,a)(Au),\\
 \mathrm d_y^2G_A(z;z,a)(u,v)
 &=-\mathrm d_z^2F(z,a)(u,Av)-\mathrm d_z^2F(z,a)(v,Au).
\end{align*}
These are \eqref{eq:matrix-0}--\eqref{eq:matrix-2} in multilinear form.
The field \(V_A\) fixes \(z\) and satisfies \(V_A(\xi^{(r)})=A\xi^{(r)}\)
for \(r=1,2\). Substitution in
\eqref{eq:full-tan-0}--\eqref{eq:full-tan-2}, with the derivatives of
\(F\) evaluated at \((z,a)\), gives the ambient identities
\begin{align*}
 V_A(E_0)&=G_A(z;z,a)=0,\\
 V_A(E_1)&=-\mathrm d_zF(A\xi^{(1)})+\mathrm d_zF(A\xi^{(1)})=0,\\
 V_A(E_2)&=-\mathrm d_zF(A\xi^{(2)})
            -2\mathrm d_z^2F(\xi^{(1)},A\xi^{(1)})\\
         &\quad+\mathrm d_zF(A\xi^{(2)})
            +2\mathrm d_z^2F(\xi^{(1)},A\xi^{(1)})=0.
\end{align*}
Expanding \(G_A\) in the monomial basis gives \(v_\alpha(z,a)=[y^\alpha]G_A\)
for every \(|\alpha|\le d\) with \(\alpha_1<d\).

\subsubsection{A genuinely new \texorpdfstring{$\mathcal{O}(3)$}{O(3)} package:
quadratic acceleration}

Assume \(d\ge2\), as in all applications considered here.  Let
\[
 C:\Sym^2\CC^3\longrightarrow\CC^3,\qquad
 C^r_{jk}=C^r_{kj},
\]
be fixed.  We prescribe a field $V_{C}$ which fixes the $1$-jets and accelerates
the $2$-jets:
\[
 V_{C}(\xi^{(1)})=0,\qquad
 V_{C}(\xi^{(2)})=C(\xi^{(1)},\xi^{(1)})=\big( C^{1}(\xi^{(1)},\xi^{(1)}) , C^{2}(\xi^{(1)},\xi^{(1)}) , C^{3}(\xi^{(1)},\xi^{(1)}) \big).
\]
Its parameter correction is built from
\begin{equation}\label{eq:acc-Q0}
 Q_0(y;z,a)
 =-\frac12\sum_{r,j,k}C^r_{jk}F_r(y,a)
       (y_j-z_j)(y_k-z_k).                               
\end{equation}
At \(y=z\),
\begin{equation}
 Q_0(z;z,a)=0,\qquad \mathrm{d}_y Q_0(z;z,a)=0,\qquad
 \partial_{y_p}\partial_{y_q}Q_0(z;z,a)
 =-\sum_rF_r(z,a)C^r_{pq}.                               \label{eq:acc-jet}
\end{equation}
The factor \(1/2\) in \eqref{eq:acc-Q0} is exactly cancelled by the two
orders in which the last two factors can be differentiated.

To remove the terms of \(y\)-degree \(d+1\), write
\[
 [Q_0]_{d+1}=\sum_{|\alpha|=d+1}h_\alpha(a)y^\alpha.
\]
For each \(|\alpha|=d+1\), fix \(\gamma(\alpha)\le\alpha\) of length
three, independently of \(C\) and the coefficients of \(F\). Set
\[
 R(y;z,a)=\sum_{|\alpha|=d+1}h_\alpha(a)
 (y-z)^{\gamma(\alpha)}y^{\alpha-\gamma(\alpha)}.
\]
Then \(R\) has $z$-degree at most $3$ and its top homogeneous \(y\)-part is
\([Q_0]_{d+1}\).  Consequently
\[
 Q_C:=Q_0-R=\sum_{|\alpha|\le d}q_\alpha(z,a)y^\alpha
\]
has \(y\)-degree at most \(d\). Since \(R\) vanishes to order three
at \(y=z\), the Taylor identities \eqref{eq:acc-jet} hold for \(Q_C\).
The coefficients of \(Q_C\) have degree at most three in \(z\) and total
degree at most one in the parameters \(a_\alpha\).

Since \(a_{d00}=1\), the parameter variation must have zero \(y_1^d\)
coefficient. In the matrix lift of Proposition~\ref{prop:matrix-lift},
subtracting a multiple of \(y_1^{d-3}(y_1-z_1)^3\) preserves the Taylor
\(2\)-jet at \(y=z\). The multiplier \([y_1^d]G_A^{\mathrm{unif}}\) is
independent of \(z\), so the correction has \(z\)-degree at most three.
For \(Q_C\), the coefficient to be cancelled can depend on \(z\).
Multiplying this coefficient by the same cubic factor can therefore
give the parameter variation a pole of order greater than three along \(Z_0=0\).

For every \(d\ge2\), \(F(y,a)\) has \(y_1^d\) coefficient one and its
coefficients are independent of \(z\). Thus subtracting
\(\bigl([y_1^d]Q_C\bigr)F(y,a)\) cancels the \(y_1^d\) term of \(Q_C\)
and keeps the \(z\)-degree at most three. The coefficients of the
corrected polynomial give the parameter components in \eqref{eq:acc-field}.
In the tangency formulas \eqref{eq:full-tan-0}--\eqref{eq:full-tan-2},
the subtracted term contributes \(-\bigl([y_1^d]Q_C\bigr)E_i\) to
\(V_C(E_i)\) for \(i=0,1,2\). These contributions vanish on the
vertical \(2\)-jet space.

\begin{thm}\label{thm:acc-package}
For \(d\ge2\) and every symmetric \(C:\Sym^2\CC^3\to\CC^3\), the vector field
\begin{equation}
 V_C=\sum_{\substack{|\alpha|\le d\\\alpha_1<d}}
       \bigl(q_\alpha(z,a)-a_\alpha\,[y_1^d]Q_C\bigr)\partial_{a_\alpha}
 +\sum_r C^r(\xi^{(1)},\xi^{(1)})
       \partial_{\xi_r^{(2)}}                           \label{eq:acc-field}
\end{equation}
is tangent to the compact vertical $2$-jet space.  Over the smooth-family
locus, it is \(\mathbb G_2\)-equivariant and 
\(\mathcal{O}_{\PP^3}(3)\)-twisted.
\end{thm}

\begin{proof}
For the field \(V_C\) in \eqref{eq:acc-field}, substitute
\(G(y)=Q_C(y;z,a)-\bigl([y_1^d]Q_C\bigr)F(y,a)\) into the tangency
formulas \eqref{eq:full-tan-0}--\eqref{eq:full-tan-2}. By \eqref{eq:acc-jet},
the value and first derivatives of \(Q_C\) vanish at \(y=z\), and its
Hessian cancels \(\mathrm d_zF\bigl(C(\xi^{(1)},\xi^{(1)})\bigr)\)
in \(V_C(E_2)\). Thus
\[
 V_C(E_i)=-\bigl([y_1^d]Q_C\bigr)E_i,\qquad i=0,1,2,
\]
which proves tangency to \eqref{eq:compact-E0}--\eqref{eq:compact-E2}.

For fixed \(C\), both \(Q_0\) in \eqref{eq:acc-Q0} and the correction
\(R\) are linear in the full homogeneous coefficient vector
\(A=(A_\alpha)_{|\alpha|=d}\). Thus the coefficient vector of \(Q_0-R\)
has the form \(B(z)A\), where the entries of \(B(z)\) are polynomials
of degree at most three in \(z\). Homogenizing each entry to degree
three gives the matrix
\[
 \widetilde B(Z)=Z_0^3B\!\left(\frac{Z_1}{Z_0},
                         \frac{Z_2}{Z_0},\frac{Z_3}{Z_0}\right).
\]
The linear field \(A\mapsto\widetilde B(Z)A\) descends to the projective
parameter space with base twist \(\mathcal O_{\PP^3}(3)\). Its affine
components are the coefficients of \(\partial_{a_\alpha}\) in
\eqref{eq:acc-field}. These components
are independent of the jets. The acceleration fixes the first jets
and is quadratic in them, so \(V_C\) commutes with reparametrization
as in Remark~\ref{rem:invariance of fields cpt}.

For the jet components, take the standard overlap \(U_0\cap U_v\) with
\begin{equation}\label{equ:coord. change from w to z}
    w_0=\frac1{z_v},\qquad w_i=\frac{z_i}{z_v}\quad(i\ne v).
\end{equation}
Write \(\partial w/\partial z=(\partial w_i/\partial z_j)_{i,j}\) for the
Jacobian matrix of \(z\mapsto w\), with rows indexed by
\(i\in\{0,1,2,3\}\setminus\{v\}\) and columns by \(j\in\{1,2,3\}\).
If \(\zeta_0^{(1)},\zeta_i^{(1)}\) are the $1$-jet coordinates in the \(w\)-chart, then
\begin{equation}
 \xi_v^{(1)}=-\frac{\zeta_0^{(1)}}{w_0^2},\qquad
 \xi_i^{(1)}=\frac{w_0\zeta_i^{(1)}-w_i \zeta_0^{(1)}}{w_0^2}.                    \label{eq:acc-invjac}
\end{equation}
Put \(A_z^r=C^r(\xi^{(1)},\xi^{(1)})\). By \eqref{eq:acc-invjac},
\(\widehat A_z^r=w_0^4 A_z^r\) is polynomial in \(w,\zeta^{(1)}\).
The second-order chain rule gives
\[
 \zeta_i^{(2)}
 =\sum_j\frac{\partial w_i}{\partial z_j}\,\xi_j^{(2)}
  +\sum_{j,k}\frac{\partial^2 w_i}{\partial z_j\partial z_k}\,
       \xi_j^{(1)}\xi_k^{(1)}.
\]
The coefficients of the first sum depend only on \(z\), and the second
sum depends only on \(z\) and \(\xi^{(1)}\). By \eqref{eq:acc-field},
\(V_C(z)=V_C(\xi^{(1)})=0\) and \(V_C(\xi^{(2)})=A_z\).
Applying \(V_C\) to the chain rule therefore gives
\(A_w=V_C(\zeta^{(2)})=(\partial w/\partial z)A_z\).
Substituting the derivatives from \eqref{equ:coord. change from w to z} yields
\begin{equation}\label{eq:acc-overlap}
\begin{split}
 A_w^0
 &=
 \frac{\partial w_0}{\partial z_v} A_z^{v}=(-w_0^2) A_z^v=-\frac{\widehat A_z^v}{w_0^2},
 \\
 A_w^i
 &=
 \frac{\partial w_i}{\partial z_i} A_z^i + \frac{\partial w_i}{\partial z_v} A_z^v=w_0 A_z^i - w_0 w_i A_z^v=\frac{\widehat A_z^i-w_i\widehat A_z^v}{w_0^3}
\end{split}
\end{equation}
The component along \(\zeta_0^{(2)}\) has pole order at most two,
and the other components have order at most three. This gives the
claimed \(\mathcal O_{\PP^3}(3)\)-twist.
\end{proof}

\subsubsection{Extension and irreducible components}

We use the compact extension statement of
\cite[Proposition~5.6]{HHMX2026}, with its notation for the divisor classes.
The corresponding proof is not repeated here.

\begin{prop}[{\cite[Proposition~5.6]{HHMX2026}}]
\label{prop:extension and irred. component cpt case}
\label{prop:o3-compact-relative-family}
For $d\geqslant15$, write $S^{\rm sm}\subset\PP^{N_d^3}$ for the smooth
surface parameter locus.  There are $m,t\in\mathbb Z_{>0}$,
$b_1,b_2\in\mathbb Z$ and a nonempty Zariski-open subset
$U^{m,t}\subset S^{\rm sm}$ with
\[
 b_1\geqslant2b_2>0,\qquad b_1+b_2=m,
\]
carrying an algebraic, hence holomorphic, family
\[
 0\not\equiv\omega_{1,a}\in
 H^0\!\left(X_a,E_{2,m}T_{X_a}^*\otimes\OO_{X_a}(-t)\right),
 \qquad a\in U^{m,t}.
\]
Let $\pi_{2,0}\colon X_{a,2}\to X_a$ be the projection from the second
level of the Demailly--Semple tower.  The jet differential
$\omega_{1,a}$ corresponds to a tautological line-bundle section
\[
 \widetilde\omega_{1,a}\in
 H^0\!\left(X_{a,2},\OO_{X_{a,2}}(m)
                    \otimes\pi_{2,0}^*\OO_{X_a}(-t)\right).
\]
Its zero divisor on $X_{a,2}$ satisfies
\begin{equation}
 \operatorname{div}_{X_{a,2}}(\widetilde\omega_{1,a})
       =Z_a+b_1\Gamma_{a,2},\qquad
 Z_a\sim b_1u_{1,a}+b_2u_{2,a}-t\pi_{2,0}^*h_a,
 \label{eq:o3-relative-divisor}
\end{equation}
with $Z_a$ reduced, irreducible and distinct from $\Gamma_{a,2}$.  Here
$u_{1,a}=[\pi_{2,1}^*\OO_{X_{a,1}}(1)]$,
$u_{2,a}=[\OO_{X_{a,2}}(1)]$ and $h_a=[\OO_{X_a}(1)]$.
\end{prop}

For an initial section of weight $M$ and twist $-T$, the maximal-slope
component construction in the proof of the cited proposition gives
$t/m\geqslant T/M$.  Put
$U=U^{m,t}$ and regard $P=\{\omega_{1,a}\}_{a\in U}$ as the equivariant map
\begin{equation}
 P:J_2^v(\mathcal X)|_{\mathrm{pr}_2^{-1}(U)}
 \longrightarrow
 p^*\mathrm{pr}_1^*\OO_{\PP^3}(-t)|_{\mathrm{pr}_2^{-1}(U)},
 \label{P holo map cpt}
\end{equation}
where $p:J_2^v(\mathcal X)\to\mathcal X$ is the jet projection.
For $a\in U$, the restriction of $p$ to $J_2X_a$ is the base-point projection
\begin{equation}\label{eq:o3-compact-jet-projection}
 p_a\colon J_2X_a\longrightarrow X_a,\qquad
 (x,\xi^{(1)},\xi^{(2)})\longmapsto x.
\end{equation}
Evaluating $\omega_{1,a}$ on jets gives the section $P|_a$ of
$p_a^*\OO_{X_a}(-t)$.
We use a tilde for the tautological section on $X_{a,2}$ associated
with an invariant jet differential on $X_a$.
The zero set of $P|_a$ on regular jets descends to
\[
 \{\widetilde\omega_{1,a}=0\}\setminus\Gamma_{a,2}
   =Z_a\setminus\Gamma_{a,2}\subset X_{a,2}\setminus\Gamma_{a,2}.
\]
By contrast, the zero locus of the vector-bundle section
$\omega_{1,a}$ on the surface $X_a$ consists of the points $x$ for
which $P(x,a,\,\cdot\,)\equiv0$ as a jet polynomial.
The residual divisor $Z_a$ and the stationary component
$b_1\Gamma_{a,2}$ in~\eqref{eq:o3-relative-divisor} both lie on $X_{a,2}$.

\subsubsection{Absorbing the parameter component and simultaneous reduction}

\begin{lem}
\label{lem:absorb-parameter-component}
Let \(V\) be a slanted vector field with parameter part
\[
 V_{\mathrm{par}}
 =
 \sum_\alpha v_\alpha \frac{\partial}{\partial a_\alpha},
\]
associated to the parameter variation
\[
 G_V(y)=\sum_\alpha v_\alpha y^\alpha .
\]
Fix a point
\(
 p=(z,a)\in \mathcal X.
\)
Assume that \(\deg_yG_V\le d\) and that \(G_V\) satisfies \eqref{condition:only parameter component}:
\[
 G_V(z)=0,\qquad \mathrm{d}_yG_V(z)=0,\qquad \mathrm{d}_y^2G_V(z)=0.
\]
Then, at the fixed point \(p\), the parameter component of \(V\) can be
absorbed by the third-difference package $D_{\alpha,\gamma}^{[3]}$ in Proposition~\ref{prop:first-package}.  More precisely, there are
constants \(c_{\alpha,\gamma}\), depending on \(p\), such that the vector
field
\[
 \widetilde V
 =
 V-\sum_{\alpha,\gamma}
 c_{\alpha,\gamma}D_{\alpha,\gamma}^{[3]}
\]
has zero parameter component at \(p\), while its \(z\)- and jet-variable
components at \(p\) are the same as those of \(V\). 
\end{lem}

\begin{proof}
Since \(G_V\) is the variation of the normalized parameters, its
coefficient of \(y_1^d\) is zero. The degree assumption and
condition~\eqref{condition:only parameter component} therefore imply
\[
 G_V\in \mathfrak m_z^3\cap
 \mathbb C[y_1,y_2,y_3]_{\le d,\,[y_1^d]=0}.
\]

If \(d<3\), the degree and vanishing assumptions give \(G_V=0\), so
the conclusion holds with all coefficients zero. For \(d\ge3\),
Lemma~\ref{lem:difference-spans} shows that, after fixing \(z\), the
third-difference fields \(D_{\alpha,\gamma}^{[3]}\) span this space.
Hence there exist constants \(c_{\alpha,\gamma}\), depending on the fixed point $(z,a)$, such that
\[
 G_V(y)
 =
 \sum_{\alpha,\gamma}
 c_{\alpha,\gamma}
 D_{\alpha,\gamma}^{[3]}F(y,a).
\]
Equivalently, at \(p=(z,a)\),
\[
 V_{\mathrm{par}}(p)
 =
 \sum_{\alpha,\gamma}
 c_{\alpha,\gamma}
 D_{\alpha,\gamma}^{[3]}(p).
\]

Since the value of a derivation at a point depends only on the tangent vector at that point, the contribution of the parameter component of \(V\) can be reproduced, at the fixed point \(p\), by a linear combination of third-difference fields.

Therefore we may replace \(V\), at the level of its action at \(p\), by
\[
 \widetilde V
 =
 V-\sum_{\alpha,\gamma}
 c_{\alpha,\gamma}D_{\alpha,\gamma}^{[3]}.
\]
The new field \(\widetilde V\) has zero parameter component at \(p\).  Since
the finite-difference fields have no \(z\)- or jet-variable components,
\(\widetilde V\) keeps the same non-parameter components as \(V\) at \(p\).
\end{proof}

To compare the actions of the matrix and acceleration fields on \(P\),
we apply Lemma~\ref{lem:absorb-parameter-component} to both at the same
point of the fibre.

\begin{lem}
\label{lem:simultaneous-open}
Let \(d\ge3\), and let \(S^{\rm sm}\subset\PP^{N_d^3}\) be the smooth
surface parameter locus.  There exists a nonempty Zariski-open subset
\(S^{\rm red}\subset S^{\rm sm}\) such that, for every \(a\in S^{\rm red}\),
the fibre \(X_a\) contains a nonempty Zariski-open subset
\(U_a\subset X_a\) with the following property.  At every fixed point
\(p=(z,a)\) with $z\in U_a$, a matrix-lift
field from Proposition~\ref{prop:matrix-lift} and an acceleration field
from Theorem~\ref{thm:acc-package} can be chosen, both with twist
\(\mathcal O(3)\), such that their parameter variation polynomials
satisfy~\eqref{condition:only parameter component} at \(p\).
Consequently, both parameter components can be absorbed simultaneously
by the third-difference package at \(p\).
\end{lem}

\begin{proof}
In each smooth fibre
take \(U_a\) to be the open set where \(F_1(x,a)\ne0\) and the
restriction of \(\mathrm d_z^2F(x,a)\) to \(T_xX_a\) is nondegenerate.
Set
\[
 S^{\rm red}=\{\,a\in S^{\rm sm}\mid U_a\ne\varnothing\,\}.
\]
It is straightforward to show that \(S^{\rm red}\) is a nonempty
Zariski-open subset of \(S^{\rm sm}\).

Fix \(x\in U_a\) and put \(p=(x,a)\). Keep the second and third affine
coordinates and take \(\mathrm d_zF(p)/F_1(p)\) as the differential of the first.
The first coordinate vector is still \(\partial_{z_1}\), so the
coefficient of \(y_1^d\) is still one. In this adapted frame,
\begin{equation}\label{normal direction assumption-cpt}
    \mathrm d_zF(p)=F_1(p)\,\mathrm dz_1,\qquad F_1(p)\ne0.
\end{equation}
For the matrix lift, write the Hessian in this frame as
\[
 \mathrm{d}_z^2F(x,a)=
 \begin{pmatrix}r&s^{\mathsf T}\\s&H\end{pmatrix},
\]
where \(H\) is the nondegenerate restriction to \(T_xX_a\).
By Proposition~\ref{prop:matrix-lift} and
\eqref{eq:matrix-0}--\eqref{eq:matrix-2}, the normalized variation
\(G_A\) satisfies \eqref{condition:only parameter component} at \(p\)
exactly when
\[
  \mathrm d_zF(p)A=0,
  \qquad
  A^{\mathsf T}\mathrm d_z^2F(p)+\mathrm d_z^2F(p)A=0.
\]
The first equation forces the first row of \(A\) to vanish, so write
\[
 A=\begin{pmatrix}0&0\\a&B\end{pmatrix}.
\]
Here \(a\) in the block matrix denotes its lower-left column.
Block multiplication gives
\begin{align}
 B^{\mathsf T}H+HB&=0,                                   \label{eq:compact-B}\\
 Ha+B^{\mathsf T}s&=0.                                   \label{eq:compact-a}
\end{align}
Equations \eqref{eq:compact-B}--\eqref{eq:compact-a} give
\(a=BH^{-1}s\) and \((BH^{-1})^{\mathsf T}=-BH^{-1}\).
Thus the remaining scalar equation \(s^{\mathsf T}a=0\) holds.
The infinitesimal orthogonal algebra of \(H\) is one-dimensional,
so we can choose \(B\ne0\) in \eqref{eq:compact-B}, with \(a\)
then determined by \eqref{eq:compact-a}.

For the acceleration field of Theorem~\ref{thm:acc-package}, choose the
constant tensor \(C\) whose only nonzero component in the frame of
\eqref{normal direction assumption-cpt} is \(C^3_{22}=1\). In this frame,
\[
 C(\xi^{(1)},\xi^{(1)})=\bigl(0,0,(\xi_2^{(1)})^2\bigr),
 \qquad \mathrm d_zF(p)\circ C=0.
\]
Thus \eqref{eq:acc-jet} gives \(j_x^2Q_C=0\). Since the second coordinate
is unchanged, \eqref{eq:acc-Q0} expresses \(Q_0\) in either frame as
\((y_2-z_2)^2\) times a polynomial of \(y\)-degree at most \(d-1\).
It follows that \([y_1^d]Q_0=0\) and that every monomial of \(y\)-degree
\(d+1\) in \(Q_0\) is divisible by \(y_2^2\). A term of \(y\)-degree \(d\)
in \(R\) arises by replacing one \(y\)-factor of such a monomial by the
corresponding \(z\)-factor, so it is still divisible by \(y_2\).
Hence \([y_1^d]Q_C=0\), and
the parameter variation in \eqref{eq:acc-field} equals \(Q_C\).

Both \(G_A\) and \(Q_C\) satisfy \eqref{condition:only parameter component}
at \(p\), so Lemma~\ref{lem:absorb-parameter-component} allows us to
cancel both parameter components there.
\end{proof}

The fields of Proposition~\ref{prop:matrix-lift} and
Theorem~\ref{thm:acc-package} depend linearly on \(A\) and \(C\),
respectively. Fix bases of the constant matrices and tensors, and include
the finitely many third differences of Proposition~\ref{prop:first-package}.
At each \(p\), the two fields obtained in
Lemma~\ref{lem:simultaneous-open} after cancellation are scalar linear
combinations of this fixed family. Their parameter components vanish
on every jet over \(p\).

\subsubsection{Independence via pole-three slanted vector fields}

\begin{thm}
\label{thm:apply O(3) fields cpt}
\label{thm:o3-differentiation}
Let $\{\omega_{1,a}\}_{a\in U^{m,t}}$ be the family in
Proposition~\ref{prop:o3-compact-relative-family}, with $m\geqslant3$
and $t\geqslant4$, and let $S^{\rm red}$ be the open set of parameters
in Lemma~\ref{lem:simultaneous-open}.
For every $a_0\in U^{m,t}\cap S^{\rm red}$, there exist a Zariski-open
neighbourhood $U'\subset U^{m,t}\cap S^{\rm red}$ of $a_0$ and a regular
slanted vector field $V$
with base pole order at most $3$, such that, for every $a\in U'$, the differential
\[
 \omega_{2,a}\in
 H^0\!\left(X_a,E_{2,m}T_{X_a}^*
                         \otimes\OO_{X_a}(-(t-3))\right)
\]
with evaluation $V(P)|_a$ on jets satisfies
$\widetilde\omega_{2,a}|_{Z_a}\not\equiv0$ on $X_{a,2}$.  Consequently,
\[
 \{\widetilde\omega_{1,a}=0\}\cap\{\widetilde\omega_{2,a}=0\}
       \subset X_{a,2}
\]
has codimension at least two outside $\Gamma_{a,2}$.
\end{thm}

\begin{proof}
Fix $a=a_0$.  For any field $V$ in the fixed finite packages, differentiation of
\eqref{P holo map cpt} gives the evaluation $V(P)|_a$ of an invariant
jet differential $\omega_{V,a}$ on $X_a$, of weight $m$ and twist $3-t$.
Here the twist follows because the fields fix the base coordinates,
and the weight follows from reparametrization invariance.  Its
corresponding section on the second Demailly--Semple level is
\[
 \widetilde\omega_{V,a}\in
 H^0\!\left(X_{a,2},\OO_{X_{a,2}}(m)
                   \otimes\pi_{2,0}^*\OO_{X_a}(3-t)\right).
\]
We must find $V$ for which
\[
 \bigl(\{\widetilde\omega_{1,a}=0\}
       \cap\{\widetilde\omega_{V,a}=0\}\bigr)\setminus\Gamma_{a,2}
       \subset X_{a,2}\setminus\Gamma_{a,2}
\]
has codimension at least two.
Since $Z_a\subset X_{a,2}$ is reduced and irreducible, this is
equivalent to $\widetilde\omega_{V,a}|_{Z_a}\not\equiv0$.
Suppose, to the contrary, that
$\widetilde\omega_{V,a}|_{Z_a}\equiv0$ for every such $V$.

The class in~\eqref{eq:o3-relative-divisor} has degree $b_2>0$ on the
fibres of $\pi_{2,1}$, so $Z_a$ dominates $X_{a,1}$ and hence $X_a$.
Since $Z_a$ is reduced, its generic fibre over $X_a$ is reduced and,
in characteristic zero, geometrically reduced
\cite[Chapter~II, Exercise~3.15(b)]{Hartshorne1977}.
The projection $Z_a\to X_a$ is projective and becomes flat over $U_a$
after shrinking $U_a$ by generic flatness.  Openness of geometric
reducedness for proper flat families
\cite[Chapter~14 and Appendix~E]{GoertzWedhorn2010} then allows us to
shrink $U_a$ further so that the fibres of $Z_a\to X_a$ are reduced and
$P(x,a,\,\cdot\,)\not\equiv0$ for every $x\in U_a$.

At a general $x\in U_a$, choose coordinates as in
\eqref{normal direction assumption-cpt}.  Equations
\eqref{eq:compact-E1}--\eqref{eq:compact-E2} give $\xi_1^{(1)}=0$
and express $\xi_1^{(2)}$ in $1$-jets alone.  Put
$W_{23}=\xi_2^{(1)}\xi_3^{(2)}-\xi_3^{(1)}\xi_2^{(2)}$.
The prescribed jet components give
\[
 \begin{aligned}
 V_A(W_{23})
   &=(A_2^2+A_3^3)W_{23}
     +(A_3^1\xi_2^{(1)}-A_2^1\xi_3^{(1)})\xi_1^{(2)},\\
 V_C(W_{23})
   &=\xi_2^{(1)}C^3(\xi^{(1)},\xi^{(1)})
     -\xi_3^{(1)}C^2(\xi^{(1)},\xi^{(1)}).
 \end{aligned}
\]
The $1$-jet components are independent of $\xi^{(2)}$. For $j=2,3$,
the formulas above give a polynomial $V(W_{23}/(\xi_j^{(1)})^3)$ of degree
at most one in $W_{23}/(\xi_j^{(1)})^3$, with regular coefficients
independent of $\xi^{(2)}$. Over each chart
$\{\xi_j^{(1)}\ne0\}$ of $X_{a,1}$, the divisor
$\Gamma_{a,2}\subset X_{a,2}$ is locally defined by
\[
 \frac{(\xi_j^{(1)})^3}{W_{23}}=0.
\]
Moreover, $V((\xi_j^{(1)})^3/W_{23})$ is divisible by
$(\xi_j^{(1)})^3/W_{23}$.
The third-difference fields fix the jets.  Hence differentiation
preserves the multiplicity $b_1\Gamma_{a,2}$ on $X_{a,2}$.

Together with the assumed vanishing on $Z_a$, this accounts for the
whole zero divisor of $\widetilde\omega_{1,a}$ in
\eqref{eq:o3-relative-divisor}.  Therefore
\[
 \frac{\widetilde\omega_{V,a}}{\widetilde\omega_{1,a}}
 \in H^0\!\left(X_{a,2},\pi_{2,0}^*\OO_{X_a}(3)\right).
\]
Since $(\pi_{2,0})_*\OO_{X_{a,2}}=\OO_{X_a}$, the quotient is
$\pi_{2,0}^*\rho_V$ for some $\rho_V\in H^0(X_a,\OO_{X_a}(3))$.
Thus $\widetilde\omega_{V,a}=(\pi_{2,0}^*\rho_V)\widetilde\omega_{1,a}$
on $X_{a,2}$.  Returning to evaluation on $J_2X_a$ gives
\begin{equation}
 V(P)|_a=(p_a^*\rho_V)P|_a,\qquad
 \rho_V\in H^0(X_a,\OO_{X_a}(3)).
 \label{equ:V(P)=rho P cpt}
\end{equation}
Here $p_a$ is the projection defined in~\eqref{eq:o3-compact-jet-projection}.
In particular, at a base point $x\in X_a$ the multiplier is
$\rho_V(x)$, independent of the jet variables.

In the invariant jet variables at this $(x,a)$ write
\begin{equation}
 P(x,a,\xi^{(1)},\xi^{(2)})
   =\sum_{0\leqslant k\leqslant\lfloor m/3\rfloor}
       R_k(x,a,\xi^{(1)})W_{23}^k,
 \label{equ:P local cpt}
\end{equation}
where $R_k$ is homogeneous of degree $m-3k$ in
$\xi_2^{(1)},\xi_3^{(1)}$.  The weighted degree of this nonzero fibre
polynomial is
$m=b_1+b_2\geqslant3$, with weights $1,1,3$ on
$\xi_2^{(1)},\xi_3^{(1)},W_{23}$, respectively.

Choose the two fields of Lemma~\ref{lem:simultaneous-open} at \((x,a)\).
Apply Lemma~\ref{lem:absorb-parameter-component} to cancel their
parameter components, and denote the resulting fields by \(V_A,V_C\).
Equation~\eqref{equ:V(P)=rho P cpt} holds for both by linearity.
The third differences fix all jet coordinates. In the frame of
\eqref{normal direction assumption-cpt}, the chosen \(V_C\) therefore
acts by \((\xi_2^{(1)})^2\partial_{\xi_3^{(2)}}\). Applying it to
\(W_{23}=\xi_2^{(1)}\xi_3^{(2)}-\xi_3^{(1)}\xi_2^{(2)}\) gives
\begin{equation}
 \left.V_C\right|_{(x,a)}(W_{23})=(\xi_2^{(1)})^3,\qquad
 \left.V_C\right|_{(x,a)}
     =(\xi_2^{(1)})^3\frac{\partial}{\partial W_{23}}
 \label{eq:acc-nil}
\end{equation}
where the second equality describes the action on invariant polynomials
in $\xi_2^{(1)},\xi_3^{(1)},W_{23}$ over $(x,a)$.  Hence
\[
 \begin{aligned}
 \left.V_C\right|_{(x,a)}
     \bigl(P(x,a,\xi^{(1)},\xi^{(2)})\bigr)
   &=\sum_{k\geqslant1}kR_k(x,a,\xi^{(1)})
        (\xi_2^{(1)})^3W_{23}^{k-1}\\
   &=\rho_{V_C}(x)\sum_{k\geqslant0}R_k(x,a,\xi^{(1)})W_{23}^k.
 \end{aligned}
\]
The highest power of $W_{23}$ first gives $\rho_{V_C}(x)=0$; comparison
of the remaining coefficients gives $R_k=0$ for every $k\geqslant1$.
Thus
\begin{equation}
 P(x,a,\xi^{(1)},\xi^{(2)})
 =\sum_{j=0}^{m}R_0^j(x,a)
             (\xi_2^{(1)})^{m-j}(\xi_3^{(1)})^j.
 \label{equ:P local after V_C cpt}
\end{equation}

This binary form is square-free: a repeated linear factor would give
a nonreduced component of the fibre of $Z_a\to X_a$ over $x$ in
$X_{a,2}\setminus\Gamma_{a,2}$, contradicting the choice of $x$.

For the matrix direction, Lemma~\ref{lem:simultaneous-open} gives
$A=\left(\begin{smallmatrix}0&0\\a&B\end{smallmatrix}\right)$ with
$B^{\mathsf T}H+HB=0$.  Choose
\[
 B=\operatorname{adj}(H)\begin{pmatrix}0&1\\-1&0\end{pmatrix}.
\]
Then $\operatorname{tr}B=0$ and $\det B=\det H\ne0$.  On the $1$-jet
polynomial~\eqref{equ:P local after V_C cpt}, the normal part contributes
nothing because $\xi_1^{(1)}=0$.  The restriction
$\left.V_A\right|_{(x,a)}$ acts on this binary form by $\pi_m(-B)$,
with the convention of Appendix~\ref{sec:sl2-preliminaries}.  Equation
\eqref{equ:V(P)=rho P cpt} makes the binary form an eigenvector.
By Corollary~\ref{cor:sl2-sharp-degree}, squarefreeness forces the
exponents in the eigenvector factorisation to satisfy
\[
 m-j\leqslant1,\qquad j\leqslant1,\qquad\text{hence }m\leqslant2,
\]
contrary to $m\geqslant3$.

Therefore one field from the original finite regular packages has
$\widetilde\omega_{V,a}|_{Z_a}\not\equiv0$ on $X_{a,2}$.
Taking $\omega_{2,a}=\omega_{V,a}$ gives the required common zero set
of codimension at least two in $X_{a,2}\setminus\Gamma_{a,2}$.
The point-dependent combinations were used only in the contradiction.
This nonvanishing persists on a Zariski-open
parameter neighbourhood in the relative divisor family of
\cite[Proposition~5.6]{HHMX2026}.  Both twists are negative, proving the
theorem.
\end{proof}

\begin{rem}
\label{rem:o3-compact-first-order}
For a smooth surface $X\subset\PP^3$ and $m,t\geqslant1$, one has
\[
 H^0\!\left(X,\Sym^mT_X^*\otimes\OO_X(-t)\right)=0.
\]
Indeed, Sakai's theorem gives the untwisted vanishing~\cite{Sakai1979};
multiplication by a nonzero section of $\OO_X(t)$ gives the displayed
twisted vanishing.  At weights one and two, invariant jet polynomials
involve only first derivatives, so
$E_{2,m}T_X^*=\Sym^mT_X^*$ $(m=1,2)$.  In particular, a nonzero negatively twisted
invariant $2$-jet differential has weight at least three.

There is also a global proof of the contradiction in
Theorem~\ref{thm:apply O(3) fields cpt}. The acceleration identity
\eqref{eq:acc-nil} gives \eqref{equ:P local after V_C cpt} on a dense
open subset of \(X_a\), where \(P|_a\) depends only on first jets.
This property is
intrinsic under changes of tangent coordinates, so the global section
belongs to $\Sym^mT_{X_a}^*\otimes\OO_{X_a}(-t)$, contradicting the
vanishing above.  This argument does not use the fibrewise degree bound.
\end{rem}

\Needspace{8\baselineskip}
\subsection{Logarithmic Case}
\label{sec:o3-logarithmic}

\subsubsection{The logarithmic universal family and its vertical $2$-jets}
Fix $d\geqslant3$.  Use homogeneous coordinates $[Z]=[Z_0:Z_1:Z_2]$ on $\PP^2$ and
$[A]=[A_\alpha]_{|\alpha|=d}$ on the coefficient space $\PP^{N_d^2}$.
Here $\alpha=(\alpha_0,\alpha_1,\alpha_2)\in\NN^3$,
$|\alpha|=\alpha_0+\alpha_1+\alpha_2$ and
$Z^\alpha=Z_0^{\alpha_0}Z_1^{\alpha_1}Z_2^{\alpha_2}$.
The universal degree-$d$ curve is
\begin{equation}
 \mathcal X=\left\{\sum_{|\alpha|=d}A_\alpha Z^\alpha=0\right\}
 \subset\PP^2\times\PP^{N_d^2},\qquad
 N_d^2=\binom{d+2}{2}-1.
 \label{family of curves in P^2}
\end{equation}
Let $\mathrm{pr}_1:\PP^2\times\PP^{N_d^2}\to\PP^2$ and
$\mathrm{pr}_2:\PP^2\times\PP^{N_d^2}\to\PP^{N_d^2}$ be the projections;
their restrictions to $\mathcal X$ are denoted by the same symbols.
For $a=[A]$, write $\mathcal C_a\subset\PP^2$ for the corresponding
curve.  Let $S^{\rm sm}_{\log}\subset\PP^{N_d^2}$ be the smooth-curve
locus; all relative logarithmic jet bundles below are considered over
this locus.

Following the logarithmic jet convention of~\cite{Rousseau2009}, set
\begin{equation}\label{eq:log-vertical-jets}
 \begin{aligned}
 \overline{J_2}(\PP^2\times\PP^{N_d^2})
   &:=J_2(\PP^2\times\PP^{N_d^2},-\log\mathcal X),\\
 \overline{J_2^v}(\PP^2\times\PP^{N_d^2})
   &:=\ker(\mathrm{pr}_2)_*
       \subset\overline{J_2}(\PP^2\times\PP^{N_d^2}).
 \end{aligned}
\end{equation}
Here $(\mathrm{pr}_2)_*$ is the induced map on $2$-jets: its kernel
consists of jets whose parameter $2$-jet is constant, as in the compact
case.  Denote the base-point projection by
\[
 p:\overline{J_2^v}(\PP^2\times\PP^{N_d^2})
       \longrightarrow\PP^2\times\PP^{N_d^2}.
\]

The auxiliary family is
\begin{equation}
 \mathcal Y=\left\{A_{\mathbf0}Z_3^d+
                 \sum_{|\alpha|=d}A_\alpha Z^\alpha=0\right\}
 \subset\PP^3\times\PP^{N_d^2+1}.
 \label{eq:auxiliary-family-one-component}
\end{equation}
Its homogeneous coordinates are $[Z:Z_3]=[Z_0:Z_1:Z_2:Z_3]$ and
$[A:A_{\mathbf0}]$; the auxiliary coefficient $A_{\mathbf0}$ is
distinct from every $A_\alpha$ with $|\alpha|=d$.
On the open subset
\[
 \mathcal Y^*=\mathcal Y\setminus
     \bigl(\{A_{\mathbf0}=0\}\cup\{A=\mathbf0\}\bigr)
\]
the forgetful map
\[
 \pi:\mathcal Y^*\longrightarrow\PP^2\times\PP^{N_d^2},\qquad
 ([Z:Z_3],[A:A_{\mathbf0}])\longmapsto([Z],[A])
\]
is well defined: $Z=\mathbf0$ on $\mathcal Y$ would force
$A_{\mathbf0}=0$.  Put $\mathcal D=\mathcal Y^*\cap\{Z_3=0\}$.
The defining equation gives $\pi^*\mathcal X=d\mathcal D$ as Cartier
divisors, so $\pi$ is a map of logarithmic pairs
$(\mathcal Y^*,\mathcal D)\to(\PP^2\times\PP^{N_d^2},\mathcal X)$.
Similarly, $\overline{J_2^v}(\mathcal Y^*)$ denotes the subspace of
$J_2(\mathcal Y^*,-\log\mathcal D)$ with constant coefficient $2$-jet
in $\PP^{N_d^2+1}$.  The map of pairs therefore induces
\begin{equation}
 \pi_{[2]}:\overline{J_2^v}(\mathcal Y^*)\longrightarrow
 \overline{J_2^v}(\PP^2\times\PP^{N_d^2}).
 \label{eq:pi-two-jets-one-component}
\end{equation}

On $\{Z_0\ne0\}\times\{A_{\mathbf0}\ne0\}$ take affine coordinates
\begin{equation}\label{eq:log-affine-coordinates}
 z_i=\frac{Z_i}{Z_0}\quad(i=1,2,3),\qquad
 a_{\alpha_1,\alpha_2}
   =\frac{A_{d-\alpha_1-\alpha_2,\alpha_1,\alpha_2}}{A_{\mathbf0}}
   \quad(\alpha_1+\alpha_2\leqslant d).
\end{equation}
In affine expressions, $\alpha=(\alpha_1,\alpha_2)\in\NN^2$ and
$z^\alpha=z_1^{\alpha_1}z_2^{\alpha_2}$.  Define
\[
 F(z,a)=\sum_{|\alpha|\leqslant d}a_\alpha z^\alpha,
 \qquad z=(z_1,z_2).
\]
Then the equation of $\mathcal Y^*$ becomes
\begin{equation}
 z_3^d+F(z_1,z_2,a)=0.
 \label{eq:inhomogeneous-equation-one-component}
\end{equation}
The nonzero coefficient tuple $a$ represents the plane parameter $[A]$;
we use $a$ for this parameter as well.  In particular,
$\mathcal C_a\cap\{Z_0\ne0\}=\{F(\,\cdot\,,a)=0\}$.
The ambient vertical logarithmic jet coordinates are
$(z,z_3,a,\xi^{(1)},\xi^{(2)},\xi_3^{(1)},\xi_3^{(2)})$, where
\begin{equation}\label{eq:coordinates-log-two-jets-one-component}
 \xi_i^{(r)}=z_i^{(r)}\ (i=1,2),\qquad
 \xi_3^{(r)}=(\log z_3)^{(r)},\qquad
 \xi^{(r)}=(\xi_1^{(r)},\xi_2^{(r)})\quad(r=1,2).
\end{equation}
The logarithmic jet coordinates $\xi_3^{(r)}$ extend across $z_3=0$.
As in~\eqref{eq:compact-E0}--\eqref{eq:compact-E2}, differentiating
\eqref{eq:inhomogeneous-equation-one-component} with $a$ fixed gives
the equations for $\overline{J_2^v}(\mathcal Y^*)$ on this chart:
\begin{align}
 E_0^{\log}&=z_3^d+F=0,\label{eq:log0}\\
 E_1^{\log}&=dz_3^d\xi_3^{(1)}
                  +\sum_{j=1}^2F_j\xi_j^{(1)}=0,\label{eq:log1}\\
 E_2^{\log}&=dz_3^d\xi_3^{(2)}+d^2z_3^d(\xi_3^{(1)})^2
     +\sum_{j=1}^2F_j\xi_j^{(2)}
     +\sum_{j,k=1}^2F_{jk}\xi_j^{(1)}\xi_k^{(1)}=0.
     \label{eq:log2}
\end{align}
Here $F_j=\partial F/\partial z_j$ and
$F_{jk}=\partial^2F/\partial z_j\partial z_k$ for $j,k=1,2$.

\subsubsection{The tangency test}

As in the compact tangency test, consider an ambient vector field
\[
 V=\sum_{|\alpha|\leqslant d}v_\alpha(z,a)\partial_{a_\alpha}
   +\sum_{j=1}^3u_j\frac{\partial}{\partial z_j}
   +\sum_{r=1}^2\sum_{j=1}^3w_j^{(r)}
       \frac{\partial}{\partial\xi_j^{(r)}}.
\]
It is tangent to $\overline{J_2^v}(\mathcal Y^*)$ if
$V(E_i^{\log})$ belongs to the ideal
$(E_0^{\log},E_1^{\log},E_2^{\log})$ for $i=0,1,2$.
Associate with its parameter component the polynomial
\[
 G_V(y;z,a)=\sum_{|\alpha|\leqslant d}v_\alpha(z,a)y^\alpha,
 \qquad y=(y_1,y_2),
\]
and abbreviate $G=G_V(\,\cdot\,;z,a)$.  As in the compact case,
regard the differentials of $F$ and $G$ as symmetric multilinear forms.
Write $u=(u_1,u_2)$ and $w^{(r)}=(w_1^{(r)},w_2^{(r)})$ for
$r=1,2$; thus $u$ is the planar base-coordinate component, whereas
$u_3$ is the coefficient of $\partial/\partial z_3$.
The $y$-derivatives of $G$ are taken with $z,a$ fixed, before setting
$y=z$.  Direct differentiation of~\eqref{eq:log0}--\eqref{eq:log2}
gives the full identities
\begin{align}
 V(E_0^{\log})
 &=G(z)+\mathrm d_zF(u)+dz_3^{d-1}u_3,
       \label{eq:full-tan-0-log}\\
 V(E_1^{\log})
 &=\mathrm d_yG(z)(\xi^{(1)})
       +\mathrm d_z^2F(u,\xi^{(1)})
       +\mathrm d_zF(w^{(1)})\notag\\
 &\quad+d^2z_3^{d-1}\xi_3^{(1)}u_3+dz_3^dw_3^{(1)},
       \label{eq:full-tan-1-log}\\
 V(E_2^{\log})
 &=\mathrm d_yG(z)(\xi^{(2)})
       +\mathrm d_y^2G(z)(\xi^{(1)},\xi^{(1)})\notag\\
 &\quad+\mathrm d_z^2F(u,\xi^{(2)})
       +\mathrm d_z^3F(u,\xi^{(1)},\xi^{(1)})\notag\\
 &\quad+\mathrm d_zF(w^{(2)})
       +2\mathrm d_z^2F(w^{(1)},\xi^{(1)})\notag\\
 &\quad+d^2z_3^{d-1}
       \bigl(\xi_3^{(2)}+d(\xi_3^{(1)})^2\bigr)u_3
       +2d^2z_3^d\xi_3^{(1)}w_3^{(1)}+dz_3^dw_3^{(2)}.
       \label{eq:full-tan-2-log}
\end{align}
These identities hold on the whole affine chart, including $z_3=0$.
On $\{z_3\ne0\}$, equations~\eqref{eq:log0}--\eqref{eq:log2}
give, modulo the jet ideal,
\[
 \begin{aligned}
 z_3^d&\equiv-F,\\
 dz_3^d\xi_3^{(1)}&\equiv-\mathrm d_zF(\xi^{(1)}),\\
 dz_3^d\bigl(\xi_3^{(2)}+d(\xi_3^{(1)})^2\bigr)
   &\equiv-\mathrm d_zF(\xi^{(2)})
            -\mathrm d_z^2F(\xi^{(1)},\xi^{(1)}).
 \end{aligned}
\]
Substitution yields the reduced tangency expressions
\begin{align}
 V(E_0^{\log})
 &\equiv G(z)+\mathrm d_zF(u)-dF\frac{u_3}{z_3},
       \label{eq:full-tan-0-log-without3}\\
 V(E_1^{\log})
 &\equiv\mathrm d_yG(z)(\xi^{(1)})
       +\mathrm d_z^2F(u,\xi^{(1)})
       +\mathrm d_zF(w^{(1)})\notag\\
 &\quad-d\frac{u_3}{z_3}\mathrm d_zF(\xi^{(1)})-dF w_3^{(1)},
       \label{eq:full-tan-1-log-without3}\\
 V(E_2^{\log})
 &\equiv\mathrm d_yG(z)(\xi^{(2)})
       +\mathrm d_y^2G(z)(\xi^{(1)},\xi^{(1)})\notag\\
 &\quad+\mathrm d_z^2F(u,\xi^{(2)})
       +\mathrm d_z^3F(u,\xi^{(1)},\xi^{(1)})\notag\\
 &\quad+\mathrm d_zF(w^{(2)})
       +2\mathrm d_z^2F(w^{(1)},\xi^{(1)})\notag\\
 &\quad-d\frac{u_3}{z_3}
       \bigl(\mathrm d_zF(\xi^{(2)})
             +\mathrm d_z^2F(\xi^{(1)},\xi^{(1)})\bigr)\notag\\
 &\quad-2d\,\mathrm d_zF(\xi^{(1)})w_3^{(1)}-dF w_3^{(2)}.
       \label{eq:full-tan-2-log-without3}
\end{align}
For the packages constructed below, $u_1=u_2=u_3=0$.
After this specialization the reduced expressions contain no division by
$z_3$ and remain valid modulo the jet ideal on the whole chart.
Thus every later tangency assertion can be checked by substituting the
Taylor coefficients of $G$ and the prescribed jet components into
\eqref{eq:full-tan-0-log}--\eqref{eq:full-tan-2-log}, or into their
reduced forms~\eqref{eq:full-tan-0-log-without3}--\eqref{eq:full-tan-2-log-without3}.
For parameter-only fields the sufficient condition remains
\begin{equation}
 G_V(z;z,a)=0,\qquad\mathrm d_yG_V(z;z,a)=0,\qquad
 \mathrm d_y^2G_V(z;z,a)=0.
 \label{condition:only parameter component log}
\end{equation}

\subsubsection{Logarithmic third differences and their pointwise span}
\begin{lem}
\label{lem:log-third-differences}
Let $d\geqslant3$ and \(\alpha,\gamma\in\NN^2\), with
\(\alpha\ge\gamma\), \(|\alpha|\le d\), and \(|\gamma|=3\), and set
\[
 D_{\alpha,\gamma}^{\log}
 =\sum_{\beta\le\gamma}(-1)^{|\beta|}
   \binom{\gamma}{\beta}z^\beta
   \frac{\partial}{\partial a_{\alpha-\beta}}.
\]
Its parameter variation in the dummy variable \(y=(y_1,y_2)\) is
\[
 G_{\alpha,\gamma}^{\log}(y;z)
 =\sum_{\beta\le\gamma}(-1)^{|\beta|}
  \binom{\gamma}{\beta}z^\beta y^{\alpha-\beta}
 =y^{\alpha-\gamma}(y-z)^\gamma.
\]
Each field annihilates the three jet equations
\eqref{eq:log0}--\eqref{eq:log2} by
\eqref{condition:only parameter component log} and has base twist
$\OO_{\PP^2}(3)$.  For fixed $z$, the parameter values of the family
$\{D_{\alpha,\gamma}^{\log}\}$ span
\[
 \mathfrak{m}_z^3\cap \mathbb{C}[y_1,y_2]_{\leqslant d}.
\]
Here $\mathfrak m_z=(y_1-z_1,y_2-z_2)$.
\end{lem}

\begin{proof}
The displayed identity for the parameter variation proves tangency and
the base twist as in Proposition~\ref{prop:first-package}. All planar
coefficients \(a_\alpha\), \(|\alpha|\le d\), are independent. For fixed
\(z\), the Taylor monomials \((y-z)^\eta\) with \(3\le|\eta|\le d\)
form a basis of the stated space. Choosing \(\gamma\le\eta\) with
\(|\gamma|=3\) and expanding \((y-z)^{\eta-\gamma}\) in powers of \(y\)
expresses each basis element as a linear combination of the displayed
parameter variations.
\end{proof}

\subsubsection{The logarithmic matrix package}

As in Proposition~\ref{prop:matrix-lift}, prescribe linear motion at
both jet orders, now with a constant matrix having zero third column:
\begin{equation}\label{eq:matrix A log}
 A=\begin{pmatrix}T&0\\ b&0\end{pmatrix},\qquad
 T\in M_2(\CC),\quad b\in M_{1,2}(\CC).
\end{equation}
Thus, with the planar notation $\xi^{(r)}$ from
\eqref{eq:coordinates-log-two-jets-one-component}, set
\[
 \widetilde V_A
 =\sum_{r=1}^2\left(
     \sum_{j=1}^2(T\xi^{(r)})_j\partial_{\xi_j^{(r)}}
       +(b\xi^{(r)})\partial_{\xi_3^{(r)}}\right).
\]

\begin{prop}[{\cite[Lemma 5]{Rousseau2009}}]\label{prop:matrix-lift log}
For every $A$ in~\eqref{eq:matrix A log}, there are coefficients
$v_\alpha(z,a)$, linear in $a$ and of $z$-degree at most three, such that
\[
 V_A=\sum_{|\alpha|\leqslant d}v_\alpha(z,a)\partial_{a_\alpha}
           +\widetilde V_A
\]
is tangent to $\overline{J_2^v}(\mathcal Y^*)$ and has projective-base
twist $\OO_{\PP^2}(3)$. \qed
\end{prop}

As in the compact case, the following single formula treats every
multi-index, including those with entries equal to zero or one:
\begin{equation}
 G_{A,0}^{\log}(y;z,a)
 =-\mathrm d_yF(y,a)\bigl(T(y-z)\bigr)
       +dF(y,a)b(y-z).
 \label{eq:log-matrix-unif}
\end{equation}
At $y=z$, for $u,v\in\CC^2$, its Taylor coefficients are
\begin{align}
 G_{A,0}^{\log}(z)&=0,\notag \\
 \mathrm{d}_y G_{A,0}^{\log}(z)(u)&=-\mathrm{d}_z F(z)(T u)+ d F(z)bu,\label{eq:log-matrix-1} \\
 \mathrm{d}_y^2 G_{A,0}^{\log}(z)(u,v)
  &=-\mathrm{d}_z^2F(z)(u,Tv)-\mathrm{d}_z^2F(z)(v,Tu)\notag\\
  &\quad+d\,\mathrm d_zF(z)(u)bv+d\,\mathrm d_zF(z)(v)bu.
  \label{eq:log-matrix-2}
\end{align}

As in the degree correction in Theorem~\ref{thm:acc-package}, write
$[G_{A,0}^{\log}]_{d+1}=\sum_{|\alpha|=d+1}h_\alpha(a)y^\alpha$ and choose
$\gamma(\alpha)\leqslant\alpha$ with $|\gamma(\alpha)|=3$.  Then
\[
 \begin{aligned}
 G_A^{\log}
   &:=G_{A,0}^{\log}-\sum_{|\alpha|=d+1}h_\alpha(a)
             (y-z)^{\gamma(\alpha)}y^{\alpha-\gamma(\alpha)}\\
   &=\sum_{|\alpha|\leqslant d}v_\alpha(z,a)y^\alpha
 \end{aligned}
\]
has $y$-degree at most $d$ and retains
\eqref{eq:log-matrix-1}--\eqref{eq:log-matrix-2}.
Substitution in~\eqref{eq:full-tan-0-log-without3}--\eqref{eq:full-tan-2-log-without3}
proves tangency, including all low multi-indices.  The coefficients
$v_\alpha$ are linear in $a$ and have $z$-degree at most three.
Using \eqref{equ:coord. change from w to z}--\eqref{eq:acc-invjac}
for the two planar variables and \(w_3=z_3/z_v\) for the third coordinate
gives the \(\OO_{\PP^2}(3)\) bound by the chain rule.

\subsubsection{The logarithmic quadratic acceleration}

Theorem~\ref{thm:acc-package} has the following logarithmic analogue.
Fix a symmetric bilinear map
\[
 C:\Sym^2\CC^2\longrightarrow\CC^2,\qquad
 C^r_{jk}=C^r_{kj}.
\]
We prescribe
\[
 \begin{aligned}
 V_C^{\log}(\xi^{(1)})&=0,\qquad
     V_C^{\log}(\xi_3^{(r)})=0\quad(r=1,2),\\
 V_C^{\log}(\xi^{(2)})&=C(\xi^{(1)},\xi^{(1)}).
 \end{aligned}
\]
As in~\eqref{eq:acc-Q0}, start with the parameter polynomial
\[
 Q_{C,0}^{\log}(y;z,a)
 =-\frac12 \sum_{r,j,k}C^r_{jk}F_r(y,a)
       (y_j-z_j)(y_k-z_k).
\]
For $u,v\in\CC^2$, its Taylor $2$-jet at $y=z$ is
\begin{equation}\label{eq:acc-jet-log}
 \begin{aligned}
 Q_{C,0}^{\log}(z;z,a)&=0,\qquad
     \mathrm d_yQ_{C,0}^{\log}(z;z,a)=0,\\
 \mathrm d_y^2Q_{C,0}^{\log}(z;z,a)(u,v)
     &=-\mathrm d_zF(z,a)\bigl(C(u,v)\bigr).
 \end{aligned}
\end{equation}
The degree is at most \(d+1\).  Removing the top homogeneous part by the
same $\mathfrak m_z^3$ correction used in
Proposition~\ref{prop:matrix-lift log} gives
$Q_C^{\log}=\sum_{|\alpha|\leqslant d}q_\alpha^{\log}(z,a)y^\alpha$,
with the same Taylor $2$-jet.  Its coefficients are linear in $a$ and
have $z$-degree at most three.

\begin{thm}
\label{thm:log-acc}
The following vector field
\[
 V_{C}^{\log}
  =\sum_{|\alpha|\leqslant d}q_\alpha^{\log}(z,a)\partial_{a_\alpha}
   +\sum_{r=1}^2 C^r(\xi^{(1)},\xi^{(1)})
        \partial_{\xi_r^{(2)}}
\]
annihilates \eqref{eq:log0}--\eqref{eq:log2}, and has projective-base twist at most
\(\mathcal{O}_{\PP^2}(3)\).
\end{thm}

\begin{proof}
Substitute~\eqref{eq:acc-jet-log} in
\eqref{eq:full-tan-0-log}--\eqref{eq:full-tan-2-log}.
The value and first derivative of \(Q_C^{\log}\) vanish at \(y=z\).
Its Hessian in \eqref{eq:acc-jet-log} cancels the acceleration term
in \eqref{eq:full-tan-2-log}, proving that all three logarithmic
equations are annihilated.
For the planar jet components the pole bounds follow from
\eqref{eq:acc-invjac}--\eqref{eq:acc-overlap}.
On \(\{Z_0Z_v\ne0\}\), \(v=1,2\), the change of coordinates
\eqref{equ:coord. change from w to z} gives \(w_0=1/z_v\) and
\(w_3=z_3/z_v\). Thus \(\mathrm d\log w_3=\mathrm d\log z_3-\mathrm dz_v/z_v\).
Evaluating this identity along a curve and differentiating once more
gives, in the jet coordinates of~\eqref{eq:coordinates-log-two-jets-one-component},
\[
 \begin{aligned}
  (\log w_3)^{(1)}
    &=\xi_3^{(1)}-\frac{\xi_v^{(1)}}{z_v},\\
  (\log w_3)^{(2)}
    &=\xi_3^{(2)}-\frac{\xi_v^{(2)}}{z_v}
                  +\frac{(\xi_v^{(1)})^2}{z_v^2}.
 \end{aligned}
\]
The field \(V_C^{\log}\) fixes \(z_v\), the first jets and \(\xi_3^{(2)}\).
It therefore annihilates \((\log w_3)^{(1)}\), and
\[
 V_C^{\log}\bigl((\log w_3)^{(2)}\bigr)
 =-\frac{V_C^{\log}(\xi_v^{(2)})}{z_v}
 =-\frac{C^v(\xi^{(1)},\xi^{(1)})}{z_v}.
\]
Since \(1/z_v=w_0\), multiplying this component by \(w_0^3\) gives
\(-C^v(w_0^2\xi^{(1)},w_0^2\xi^{(1)})\). By~\eqref{eq:acc-invjac},
both components of \(w_0^2\xi^{(1)}\) are regular in the \(w\)-chart,
so the pole along \(w_0=0\) has order at most three. Together with the
planar pole bounds and the \(z\)-degree bound on \(q_\alpha^{\log}\),
this proves the claimed twist.
\end{proof}

\begin{rem}\label{rem:invariance of fields log}
In the logarithmic jet coordinates of~\eqref{eq:coordinates-log-two-jets-one-component},
the coefficients of the three packages are polynomials in the planar
coordinates, the parameters and the jets, so they are regular along
\(\mathcal D\). Tangency there follows from
\eqref{condition:only parameter component log},
\eqref{eq:log-matrix-1}--\eqref{eq:log-matrix-2} and
\eqref{eq:acc-jet-log}.

For descent to projective parameters, replace the third-difference
fields by the homogeneous representatives
$a_\nu D_{\alpha,\gamma}^{\log}$, $|\nu|\leqslant d$.
Their values have the same span at each nonzero $a$, since some
$a_\nu\ne0$.  The parameter components of all three packages are then
linear in $a$, as in the projectivization argument in
Theorem~\ref{thm:acc-package}.
All three packages are invariant under
\((a,z_3)\mapsto(c^d a,cz_3)\), \(c\in\CC^*\), as is the map
\(\pi_{[2]}\) in~\eqref{eq:pi-two-jets-one-component}.

On \(a_\nu\ne0\), the local equation \(F(z,a)/a_\nu\) of \(\mathcal X\)
pulls back to \(-z_3^d/a_\nu\) by~\eqref{eq:log0}. Since the parameters
are constant along vertical jets, \eqref{eq:coordinates-log-two-jets-one-component}
gives, for every such choice of \(\nu\),
\[
 \pi_{[2]}^*\!\left(\bigl(\log(F(z,a)/a_\nu)\bigr)^{(r)}\right)
 =d\,\xi_3^{(r)},\qquad r=1,2.
\]
The planar coordinates, their jets and the ratios \(a_\alpha/a_\nu\)
are also pullbacks. Writing \(a_\alpha=a_\nu(a_\alpha/a_\nu)\), we see
from \eqref{eq:log0}--\eqref{eq:log2} that every regular function on
this source chart can be written as a polynomial in \(z_3,a_\nu,a_\nu^{-1}\)
with coefficients pulled back from the target. Since \(z_3\) has weight
\(1\) and \(a_\nu\) has weight \(d\), the weight-zero monomials in
\(z_3,a_\nu^{\pm1}\) are exactly the nonnegative powers of
\(z_3^d/a_\nu\). Thus the invariant regular functions are precisely the
pullbacks of regular functions on the target, even over \(z_3=0\).
Pullback is injective by the dominance of \(\pi_{[2]}\)~\cite[p.~25]{Rousseau2009}.
Thus, in each local trivialization of the base twist, restriction of an
invariant field \(V\) to this subalgebra defines a unique derivation on
the target. These derivations glue to the
\(\OO_{\PP^2}(3)\)-twisted field
\begin{equation}\label{eq:log-descended-field}
 \widehat V=(\pi_{[2]})_*V,\qquad
 \pi_{[2]}^*\circ\widehat V=V\circ\pi_{[2]}^*.
\end{equation}
In particular, the coefficients of \(\widehat V\) in logarithmic jet
coordinates are regular along \(\mathcal C_a\) for every smooth fibre.
The calculation in Remark~\ref{rem:invariance of fields cpt} also gives
\(\mathbb G_2\)-equivariance and weight zero for \(\widehat V\).
\end{rem}

\subsubsection{Extension and irreducible components}

The logarithmic extension statement is recorded separately, as in
\cite[Proposition~5.7]{HHMX2026}; we refer there for its proof.

\begin{prop}[{\cite[Proposition~5.7]{HHMX2026}}]
\label{prop:extension and irred. component log case}
\label{prop:o3-log-relative-family}
For $d\geqslant11$, an algebraic, hence holomorphic, family
\[
 0\not\equiv\omega_{1,a}\in
 H^0\!\left(\PP^2,E_{2,m}T_{\PP^2}^*(\log\mathcal C_a)
                             \otimes\OO_{\PP^2}(-t)\right)
\]
exists on a nonempty Zariski-open subset $U^{m,t}\subset S^{\rm sm}_{\log}$ for some
$m,t\in\mathbb Z_{>0}$ and $b_1,b_2\in\mathbb Z$ satisfying
\[
 b_1\geqslant2b_2>0,\qquad b_1+b_2=m.
\]
Let $\pi_{2,0}\colon X_{a,2}\to\PP^2$ be the projection from the
second level of the logarithmic Demailly--Semple tower of
$(\PP^2,\mathcal C_a)$.  The corresponding tautological section is
\[
 \widetilde\omega_{1,a}\in
 H^0\!\left(X_{a,2},\OO_{X_{a,2}}(m)
                    \otimes\pi_{2,0}^*\OO_{\PP^2}(-t)\right),
\]
and its zero divisor on $X_{a,2}$ satisfies
\begin{equation}
 \operatorname{div}_{X_{a,2}}(\widetilde\omega_{1,a})
       =Z_a+b_1\Gamma_{a,2},\qquad
 Z_a\sim b_1u_{1,a}+b_2u_{2,a}-t\pi_{2,0}^*h.
 \label{eq:o3-log-relative-divisor}
\end{equation}
The divisor $Z_a$ is reduced, irreducible and distinct from
$\Gamma_{a,2}$; here
$u_{1,a}=[\pi_{2,1}^*\OO_{X_{a,1}}(1)]$,
$u_{2,a}=[\OO_{X_{a,2}}(1)]$ and $h=[\OO_{\PP^2}(1)]$.
\end{prop}

For a selected maximal-slope factor of an initial section of weight $M$
and twist $-T$, the construction in the cited proof again gives
$t/m\geqslant T/M$.  Set $U=U^{m,t}$.  As in~\eqref{P holo map cpt},
$P=\{\omega_{1,a}\}_{a\in U}$ is an equivariant map on
$\overline{J_2^v}(\PP^2\times\PP^{N_d^2})|_{\mathrm{pr}_2^{-1}(U)}$
with values in $p^*\mathrm{pr}_1^*\OO_{\PP^2}(-t)$, where $p$ is the
base-point projection introduced after~\eqref{eq:log-vertical-jets}.
Restricting $\mathrm{pr}_1\circ p$ to parameter $a$ gives
\begin{equation}\label{eq:o3-log-jet-projection}
 p_a^{\log}\colon
 \left.\overline{J_2^v}(\PP^2\times\PP^{N_d^2})\right|_{\mathrm{pr}_2^{-1}(a)}
 \longrightarrow\PP^2.
\end{equation}
As in the compact case, the zero set of
$P|_a$ on regular logarithmic jets descends to
$Z_a\setminus\Gamma_{a,2}\subset X_{a,2}\setminus\Gamma_{a,2}$.
The zero locus of the vector-bundle section $\omega_{1,a}$ on
$\PP^2$ instead consists of the points $x$ for which
$P(x,a,\,\cdot\,)\equiv0$.
Over $\PP^2\setminus\mathcal C_a$, pullback by
$\pi_{[2]}$ in~\eqref{eq:pi-two-jets-one-component} leaves the planar
jet expression of $P$ unchanged.  In this expression $\pi_{[2]}^*P$
is independent of $z_3,\xi_3^{(1)},\xi_3^{(2)}$.

\subsubsection{Absorbing the parameter component and simultaneous reduction}

\begin{lem}
\label{lem:absorb-parameter-component log}
Let $V$ be a logarithmic slanted vector field with parameter part
\[
 V_{\mathrm{par}}
   =\sum_{|\alpha|\leqslant d}v_\alpha(z,a)
        \frac{\partial}{\partial a_\alpha},
 \qquad
 G_V(y;z,a)=\sum_{|\alpha|\leqslant d}v_\alpha(z,a)y^\alpha.
\]
Assume that the coefficients $v_\alpha(z,a)$ are linear in $a$.
Fix $(z,a)$, and assume that $\deg_yG_V\leqslant d$ and
\eqref{condition:only parameter component log} holds there.
Choose $\nu$ with $a_\nu\ne0$.
There are constants $c_{\alpha,\gamma}$, depending on $(z,a)$, such that
\[
 \widetilde V
   =V-\sum_{\alpha,\gamma}c_{\alpha,\gamma}
                   a_\nu D_{\alpha,\gamma}^{\log}
\]
has zero parameter component on all jets based in $\pi^{-1}(z,a)$,
while its base-coordinate and jet-variable components are unchanged.
\end{lem}

\begin{proof}
The proof follows the same argument as in
Lemma~\ref{lem:absorb-parameter-component}, now with two variables.
\end{proof}

As in the compact case, the matrix and acceleration packages must admit
this cancellation at the same base point.  The following lemma gives
the required simultaneous reduction.

\begin{lem}
\label{lem:simultaneous-open log}
Let $d\geqslant11$.
There exists a nonempty Zariski-open subset
$S^{\rm red}_{\log}\subset S^{\rm sm}_{\log}$ such that, for every
$a\in S^{\rm red}_{\log}$, the complement $\PP^2\setminus\mathcal C_a$
contains a nonempty Zariski-open subset $U_a$ with the following property.
At every fixed $(x,a)$ with $x\in U_a$, a matrix-lift field from
Proposition~\ref{prop:matrix-lift log} and an acceleration field from
Theorem~\ref{thm:log-acc} can be chosen, both with base twist
$\OO_{\PP^2}(3)$, so that their parameter variation polynomials satisfy
\eqref{condition:only parameter component log} at $(x,a)$.
The matrix block $T$ can be chosen with $\operatorname{tr}T=0$ and
$\det T\ne0$.
Consequently, both parameter components can be absorbed simultaneously
by Lemma~\ref{lem:absorb-parameter-component log}.
\end{lem}

\begin{proof}
Regard $\mathrm d_zF$ as a row vector and $\mathrm d_z^2F$ as its
$2\times2$ Hessian matrix, and put
\[
 H^{\log}(x,a)=F(x,a)\,\mathrm d_z^2F(x,a)
                   -(\mathrm d_zF(x,a))^{\mathsf T}\mathrm d_zF(x,a).
\]
The condition $\det H^{\log}(\,\cdot\,,a)\not\equiv0$ defines a
Zariski-open subset $S^{\rm red}_{\log}\subset S^{\rm sm}_{\log}$.
For $d\geqslant11$ it is nonempty: the Fermat polynomial
$F=1+z_1^d+z_2^d$ gives
$\det H^{\log}(1,1)=3d^2(d-1)(d-3)\ne0$.
For each $a\in S^{\rm red}_{\log}$, set
\begin{equation}\label{eq:log-simultaneous-open}
 U_a:=\{x\in\CC^2:\det H^{\log}(x,a)\ne0,\ 
                         \mathrm d_zF(x,a)\ne0\}.
\end{equation}
This open set is nonempty since $\det H^{\log}(\,\cdot\,,a)\not\equiv0$
and $F(\,\cdot\,,a)$ is nonconstant.
On $U_a$ one has $F\ne0$: if $F=0$, the matrix $H^{\log}$ has rank
at most one.

Fix $x\in U_a$ and choose affine linear coordinates such that
\begin{equation}
 \mathrm d_zF(x,a)=F_1(x,a)\,\mathrm dz_1,\qquad F_1(x,a)\ne0.
 \label{normal direction assumption-log}
\end{equation}
All matrices in the following calculation are evaluated at this fixed
$(x,a)$.
For $A=\left(\begin{smallmatrix}T&0\\b&0\end{smallmatrix}\right)$,
\eqref{eq:log-matrix-1}--\eqref{eq:log-matrix-2} show that vanishing of
the matrix correction's Taylor $2$-jet is equivalent to
\begin{align}
 \mathrm d_zF\,T&=dF\,b,\label{eq:condition T,b-1}\\
 T^{\mathsf T}\mathrm d_z^2F+\mathrm d_z^2F\,T
   &=d\bigl(b^{\mathsf T}\mathrm d_zF+(\mathrm d_zF)^{\mathsf T}b\bigr).
   \label{eq:condition T,b-2}
\end{align}
Eliminating $b$ by~\eqref{eq:condition T,b-1} gives
\begin{equation}
 T^{\mathsf T}H^{\log}+H^{\log}T=0.
 \label{eq:logarithmic-T}
\end{equation}
Take
\[
 T=\operatorname{adj}(H^{\log})
        \begin{pmatrix}0&1\\-1&0\end{pmatrix},\qquad
 b=\frac{\mathrm d_zF\,T}{dF}.
\]
This solves~\eqref{eq:logarithmic-T} and gives
$\operatorname{tr}T=0$ and $\det T=\det H^{\log}\ne0$.
For the acceleration choose
\[
 C^2(\xi^{(1)},\xi^{(1)})=(\xi_1^{(1)})^2,\qquad C^1=0.
\]
Condition~\eqref{normal direction assumption-log} gives
$\mathrm d_zF\circ C=0$, so~\eqref{eq:acc-jet-log} implies that
$Q_C^{\log}$ also satisfies~\eqref{condition:only parameter component log}
at $(x,a)$.
Lemma~\ref{lem:absorb-parameter-component log} now cancels both parameter
components at this same point.
\end{proof}

As in the compact case, fix constant bases of the matrix and acceleration
packages and use the finitely many homogeneous third-difference fields.
Once $(x,a)$ is selected, $T,b,C$ and the cancelling coefficients are
held constant.  Each resulting field is a scalar linear combination
of these fixed packages, and the cancellation holds for every jet
based in $\pi^{-1}(x,a)$.

\subsubsection{Independence via pole-three slanted vector fields}

\begin{thm}
\label{thm:apply O(3) fields log}
Let $\{\omega_{1,a}\}_{a\in U^{m,t}}$ be the family in
Proposition~\ref{prop:o3-log-relative-family}, with $m\geqslant3$
and $t\geqslant4$, and let $S^{\rm red}_{\log}$ be the open set of parameters
in Lemma~\ref{lem:simultaneous-open log}.
For every $a_0\in U^{m,t}\cap S^{\rm red}_{\log}$, there exist a
Zariski-open neighbourhood
$U'\subset U^{m,t}\cap S^{\rm red}_{\log}$ of $a_0$ and a regular
logarithmic slanted vector field $V$ on $\overline{J_2^v}(\mathcal Y^*)$
with base pole order at most $3$ and induced field
$\widehat V=(\pi_{[2]})_*V$ (see~\eqref{eq:log-descended-field}),
such that, for every $a\in U'$, the differential
\[
 \omega_{2,a}
 \in H^0\!\left(\PP^2,E_{2,m}T_{\PP^2}^*(\log\mathcal C_a)
                    \otimes\OO_{\PP^2}(-(t-3))\right)
\]
with evaluation $\widehat V(P)|_a$ on logarithmic jets satisfies
$\widetilde\omega_{2,a}|_{Z_a}\not\equiv0$ on $X_{a,2}$.  Consequently,
\[
 \{\widetilde\omega_{1,a}=0\}\cap\{\widetilde\omega_{2,a}=0\}
       \subset X_{a,2}
\]
has codimension at least two outside $\Gamma_{a,2}$.
\end{thm}

\begin{proof}
Fix $a=a_0$.  Use the descended fields $\widehat V=(\pi_{[2]})_*V$
defined in~\eqref{eq:log-descended-field} for the packages of
Lemma~\ref{lem:log-third-differences},
Proposition~\ref{prop:matrix-lift log} and Theorem~\ref{thm:log-acc}.
For each such field, $\widehat V(P)|_a$ is the evaluation of a global invariant
logarithmic jet differential $\omega_{V,a}$ on $\PP^2$, of weight $m$
and twist $3-t$.
Its associated section on the logarithmic second level is
\[
 \widetilde\omega_{V,a}\in
 H^0\!\left(X_{a,2},\OO_{X_{a,2}}(m)
                   \otimes\pi_{2,0}^*\OO_{\PP^2}(3-t)\right).
\]
As in the proof of Theorem~\ref{thm:apply O(3) fields cpt}, the required
codimension statement concerns
the common zero set of $\widetilde\omega_{1,a}$ and
$\widetilde\omega_{V,a}$ in $X_{a,2}\setminus\Gamma_{a,2}$ and is
equivalent to $\widetilde\omega_{V,a}|_{Z_a}\not\equiv0$ for some $V$.
Suppose that every $\widetilde\omega_{V,a}$ vanishes identically on
the reduced irreducible divisor $Z_a\subset X_{a,2}$.

Put $W_{12}=\xi_1^{(1)}\xi_2^{(2)}-\xi_2^{(1)}\xi_1^{(2)}$.
For $A$ in~\eqref{eq:matrix A log}, the zero third column gives
\[
 \begin{aligned}
 V_A(W_{12})&=\operatorname{tr}(T)W_{12},\\
 V_C^{\log}(W_{12})
   &=\xi_1^{(1)}C^2(\xi^{(1)},\xi^{(1)})
     -\xi_2^{(1)}C^1(\xi^{(1)},\xi^{(1)}).
 \end{aligned}
\]
Over $\PP^2\setminus\mathcal C_a$, for $j=1,2$,
$\widehat V(W_{12}/(\xi_j^{(1)})^3)$ is a polynomial of degree at most
one in $W_{12}/(\xi_j^{(1)})^3$, with regular coefficients independent
of $\xi^{(2)}$.  Near $\Gamma_{a,2}$, over each chart
$\{\xi_j^{(1)}\ne0\}$ of $X_{a,1}|_{\PP^2\setminus\mathcal C_a}$,
a local defining equation for $\Gamma_{a,2}\subset X_{a,2}$ is
\[
 \frac{(\xi_j^{(1)})^3}{W_{12}}=0.
\]
Moreover, $\widehat V((\xi_j^{(1)})^3/W_{12})$ is divisible by
$(\xi_j^{(1)})^3/W_{12}$.
Vanishing order along $\Gamma_{a,2}$ is determined on this dense open set.
The sections extend across $\mathcal C_a$ by
the regularity in Remark~\ref{rem:invariance of fields log}, so
$\widetilde\omega_{V,a}$ retains the multiplicity
$b_1\Gamma_{a,2}$ on $X_{a,2}$.  Together with the assumed vanishing
on $Z_a$, equation~\eqref{eq:o3-log-relative-divisor} gives
\[
 \frac{\widetilde\omega_{V,a}}{\widetilde\omega_{1,a}}
 \in H^0\!\left(X_{a,2},\pi_{2,0}^*\OO_{\PP^2}(3)\right).
\]
Since $(\pi_{2,0})_*\OO_{X_{a,2}}=\OO_{\PP^2}$, this quotient is
$\pi_{2,0}^*\rho_V$ for some
$\rho_V\in H^0(\PP^2,\OO_{\PP^2}(3))$.  Thus
$\widetilde\omega_{V,a}=(\pi_{2,0}^*\rho_V)\widetilde\omega_{1,a}$
on $X_{a,2}$.  On the logarithmic jet bundle this becomes
\begin{equation}\label{equ:V(P)=rho P log}
 \widehat V(P)|_a=((p_a^{\log})^*\rho_V)P|_a.
\end{equation}
Here $p_a^{\log}$ is the projection defined in~\eqref{eq:o3-log-jet-projection}.
The multiplier depends only on the base point $x\in\PP^2$.
Pulling this jet-space identity back by $\pi_{[2]}$ gives
\begin{equation}\label{equ:pullback V(P)=rho P log}
 V(\pi_{[2]}^*P)|_{(\mathrm{pr}_2\circ\pi)^{-1}(a)}
   =((p_a^{\log}\circ\pi_{[2]})^*\rho_V)\,
     \pi_{[2]}^*P|_{(\mathrm{pr}_2\circ\pi)^{-1}(a)}.
\end{equation}

The class in~\eqref{eq:o3-log-relative-divisor} has degree $b_2>0$
on the fibres of $\pi_{2,1}$, so $Z_a$ dominates $X_{a,1}$ and hence $\PP^2$.
As in the proof of Theorem~\ref{thm:apply O(3) fields cpt},
the generic fibre of $Z_a\to\PP^2$ is
geometrically reduced in characteristic zero.  This projective morphism
becomes flat over $U_a$ after shrinking $U_a$ by generic flatness.
Openness of geometric reducedness then allows a further shrinking of
$U_a\subset\PP^2\setminus\mathcal C_a$ so that the fibres of
$Z_a\to\PP^2$ are reduced and
$P(x,a,\,\cdot\,)\not\equiv0$ for every $x\in U_a$.
Choose a general complex point $x\in U_a$ and coordinates as in
\eqref{normal direction assumption-log}.
The invariant polynomial expansion is similar in form to
\eqref{equ:P local cpt}. With $\xi^{(1)}$ as in
\eqref{eq:coordinates-log-two-jets-one-component}, we have
\[
 \left.\pi_{[2]}^*P\right|_{\pi^{-1}(x,a)}
   =\sum_{0\leqslant k\leqslant\lfloor m/3\rfloor}
          R_k(x,a,\xi^{(1)})W_{12}^k.
\]
Here $R_k$ is homogeneous of degree $m-3k$ in the two $1$-jet variables.
This nonzero polynomial has weighted degree
$m=b_1+b_2\geqslant3$ in
$\xi_1^{(1)},\xi_2^{(1)},W_{12}$ with weights $1,1,3$.
Choose $T,b,C$ as in the proof of Lemma~\ref{lem:simultaneous-open log}
and cancel the parameter components by
Lemma~\ref{lem:absorb-parameter-component log}.
By linearity,~\eqref{equ:pullback V(P)=rho P log} remains valid for
these combinations at the fixed $(x,a)$.  Retain the notation
$V_A,V_C$ for these combinations on the auxiliary jet space.
Their parameter components vanish on jets with base point in
$\pi^{-1}(x,a)$, and
$\left.V_C\right|_{\pi^{-1}(x,a)}
  =(\xi_1^{(1)})^2\partial_{\xi_2^{(2)}}$.
Thus
\begin{equation}\label{eq:acc-nil-log}
 \left.V_C\right|_{\pi^{-1}(x,a)}(W_{12})=(\xi_1^{(1)})^3,\qquad
 \left.V_C\right|_{\pi^{-1}(x,a)}
     =(\xi_1^{(1)})^3\frac{\partial}{\partial W_{12}},
\end{equation}
where the second equality describes the action on invariant polynomials
in $\xi_1^{(1)},\xi_2^{(1)},W_{12}$ over $\pi^{-1}(x,a)$.
Applying this restricted field to the pulled-back polynomial gives
\[
 \begin{aligned}
 \left.V_C\right|_{\pi^{-1}(x,a)}
     \left(\left.\pi_{[2]}^*P\right|_{\pi^{-1}(x,a)}\right)
   &=\sum_{k\geqslant1}kR_k(x,a,\xi^{(1)})
        (\xi_1^{(1)})^3W_{12}^{k-1}\\
   &=\rho_{V_C}(x)\sum_{k\geqslant0}R_k(x,a,\xi^{(1)})W_{12}^k.
 \end{aligned}
\]
The highest power of $W_{12}$ gives $\rho_{V_C}(x)=0$, and comparison
of the remaining coefficients gives $R_k=0$ for $k\geqslant1$.
Thus the fibre polynomial is a homogeneous binary form of degree
$m$ in
$\xi_1^{(1)},\xi_2^{(1)}$.  The reduced fibre of $Z_a\to\PP^2$ over
$x$ makes this binary form square-free: a repeated linear factor would
give a nonreduced component off $\Gamma_{a,2}$, by the same argument as
in Theorem~\ref{thm:apply O(3) fields cpt}.
Since it is independent of $\xi_3^{(1)},\xi_3^{(2)}$, only the block
$T$ contributes to the matrix action.
The restriction $\left.V_A\right|_{\pi^{-1}(x,a)}$ acts on the binary
form by $\pi_m(-T)$, with
$\operatorname{tr}T=0$ and $\det T\ne0$.
Corollary~\ref{cor:sl2-sharp-degree} therefore gives
\[
 m-j\leqslant1,\qquad j\leqslant1,\qquad\text{hence }m\leqslant2,
\]
contrary to $m\geqslant3$.

Consequently a regular member of the original finite packages satisfies
$\widetilde\omega_{V,a}|_{Z_a}\not\equiv0$ on the logarithmic tower.
Taking $\omega_{2,a}=\omega_{V,a}$ gives the common zero set of
codimension at least two in $X_{a,2}\setminus\Gamma_{a,2}$.
As in the final paragraph of the proof of
Theorem~\ref{thm:apply O(3) fields cpt}, this nonvanishing persists on
a Zariski-open parameter neighbourhood of $a_0$.  Both twists are
negative because $t\geqslant4$.
\end{proof}

\begin{rem}
\label{rem:o3-log-first-order}
For a smooth irreducible curve $\mathcal C\subset\PP^2$ of degree at
least three, one has
\[
 H^0\!\left(\PP^2,\Sym^m\bigl(T_{\PP^2}^*(\log\mathcal C)\bigr)
                           \otimes\OO_{\PP^2}(-t)\right)=0
 \qquad(m,t\geqslant1).
\]
This follows from~\cite[Theorem~4.2 and Corollary~4.1]{Diverio2009}.
As in Remark~\ref{rem:o3-compact-first-order}, weights one and two
involve only first derivatives, so a nonzero negatively twisted
invariant logarithmic $2$-jet differential has $m\geqslant3$.

The acceleration calculation also gives an independent global proof
of the contradiction in Theorem~\ref{thm:apply O(3) fields log}.
On a dense open set, the polynomial depends only on $1$-jets.
This is intrinsic under logarithmic changes
of tangent coordinates.  Hence the global section is a negatively
twisted logarithmic symmetric differential, contradicting the displayed
vanishing.  This argument does not use the fibrewise degree bound.
\end{rem}

\section{Logarithmic symmetry reduction}
\label{sec:log-symmetry-proof}

This section gives the exact reduction underlying the logarithmic clause of
Lemma~\ref{lem:finite-kvl-main}.  Following the three-step strategy
of~\cite[\S3]{HHMX2026}, we encode a hypothetical negatively twisted
logarithmic $2$-jet differential, impose the involution eigencondition, and
obtain a finite integral linear system.  The associated modular rank
analysis is described in Appendix~\ref{app:log-computations}.

The exposition follows the construction of that matrix.  We first fix and
check a smooth symmetric fibre, then write the invariant $2$-jet normal form
and its degree bounds.  The chart transition converts the involution into
coefficientwise covariance equations.  Finally, holomorphicity at infinity
and covariance are assembled into one exact system.  We derive the
normal form and all matrix equations, then apply the completed modular
verification to obtain vanishing.

\subsection{Local description of invariant logarithmic \texorpdfstring{$2$-jet}{2-jet} differentials}
\label{sec:log-local}

Use homogeneous coordinates $[T:X:Y]$ on $\PP^2$. On the affine chart
$\{T\ne0\}$, take $(x,y)=(X/T,Y/T)$. For $d=11$, take $\mathcal C$ to be
the projective closure of the affine curve $R(x,y)=0$, where
\[
R(x,y)=x^{11}+y^{11}+1+x^5+y^5,
\]
which is invariant under exchanging $x$ and $y$. Its partial derivatives are
\[
 R_x=\frac{\partial R}{\partial x}=11x^{10}+5x^4,
 \qquad
 R_y=\frac{\partial R}{\partial y}=11y^{10}+5y^4.
\]

\begin{lem}
\label{lem:log-special-fibre}
The projective curve defined by
\[
 \widehat R=X^{11}+Y^{11}+T^{11}+T^6X^5+T^6Y^5
\]
is smooth. On the affine chart $\{T\ne0\}$, $R_x$ and $R_y$ have no common
zero on $\mathcal C$.
\end{lem}

The assertions follow by direct calculation.

On the open set $\{R_y \neq 0\}$, the local algebra of invariant
logarithmic $2$-jet polynomials is generated by $x'$, $(\log R)' = R'/R$,
and the logarithmic Wronskian
\[
\Delta_{xR}''' := 
\begin{vmatrix}
x' & \frac{R'}{R} \\[2mm]
x'' & \frac{R''}{R} - \frac{R'^2}{R^2}
\end{vmatrix}
= x'\,\frac{R''}{R} - x'\,\frac{R'^2}{R^2} - x''\,\frac{R'}{R},
\]
where
\[
R' = x' R_x + y' R_y,\qquad
R'' = x'' R_x + y'' R_y + x'^2 R_{xx} + 2x'y' R_{xy} + y'^2 R_{yy}.
\]

According to~\cite[Proposition 3.1]{HHMX2026}, any global section $\omega$ of $E_{2,m}T_{\mathbb{P}^2}^*(\log\mathcal{C})$ can be written on $\{R_y \neq 0\}$ as
\begin{equation}\label{eq:local-J-log}
\omega_{xR} = \sum_{k=0}^{\lfloor m/3\rfloor}\;\sum_{j=0}^{m-3k}
\frac{F_{j,k}}{R_y^{m-2k}}\,(x')^{m-3k-j}\Bigl(\frac{R'}{R}\Bigr)^j
\bigl(\Delta_{xR}'''\bigr)^k,
\end{equation}
with polynomials $F_{j,k} \in \mathbb{C}[x,y]$.

\subsection{Degree bounds from vanishing along the ample divisor}
\label{sec:log-degree}

A section of
$E_{2,m}T_{\PP^2}^*(\log\mathcal C)\otimes\OO(-t)$, with
$t\geqslant1$, is represented by a section vanishing to order at least $t$
along $\{X=0\}$.  Thus $F_{j,k}=x^tK_{j,k}$.  The transition to the chart
$\{X\ne0\}$, computed in \cite[\S3.2]{HHMX2026}, gives
\[
\deg K_{j,k} \leqslant (m-2k)(d-1) - m + j + k + \Twist,
\]
where $\Twist=-t$.  Put
\[
 D_{j,k}=(d-2)m-(2d-3)k+j-t.
\]
The number of monomials in $x,y$ of degree at most $D_{j,k}$ is
$\binom{D_{j,k}+2}{2}$, with $\binom{a}{b}=0$ when $a<b$.  Thus the
number of candidate coefficients before imposing the symmetry equations
and the holomorphicity conditions at infinity is
\begin{equation}
 N^{\log}_{d}(m,t)
  =\sum_{k=0}^{\lfloor m/3\rfloor}
    \sum_{j=0}^{m-3k}\binom{D_{j,k}+2}{2}.
  \label{eq:o3-log-variable-count}
\end{equation}
For $d=11$, this gives $N^{\log}_{11}(51,7)=33\,956\,835$ and
$N^{\log}_{11}(22,3)=1\,383\,586$.

\subsection{The eigenvector condition from symmetry}
\label{sec:log-eigen}

The polynomial $R$ is invariant under the involution $\sigma$ that exchanges $x$ and $y$:
\[
\sigma(x,y)=(y,x),\qquad \sigma^*R=R.
\]
This involution lifts to an action on jet coordinates:
\[
\sigma^*(x')=y',\quad \sigma^*(y')=x',\quad \sigma^*(x'')=y'',\quad \sigma^*(y'')=x''.
\]
Consequently, the logarithmic Wronskian transforms as $\sigma^*(\Delta_{xR}''') = \Delta_{yR}'''$, where $\Delta_{yR}'''$ denotes the Wronskian with $x$ and $y$ exchanged.

Suppose that the negatively twisted space is nonzero.  Choose an eigensection
\[
 0\ne s\in
 H^0\!\left(\PP^2,E_{2,m}T_{\PP^2}^*(\log\mathcal C)
                  \otimes\OO(-t)\right),
 \qquad \sigma^*s=\varepsilon s,\quad \varepsilon\in\{\pm1\}.
\]
Inject it into the untwisted bundle by
$\Omega_x:=X^t s$. On $\{T\ne0\}$, trivializing $\OO(t)$ by $T^t$
gives $X^t/T^t=x^t$. Thus the coefficients in
\eqref{eq:local-J-log} are $F_{j,k}=x^tK_{j,k}$. Since
$\sigma^*X=Y$, the injected section is covariant, not invariant:
\begin{equation}
 x^t\sigma^*\Omega_x=\varepsilon y^t\Omega_x.
 \label{eq:log-twist-covariance}
\end{equation}
We impose this identity after writing both sides on the chart
$\{R_x\ne0\}$.  In the transition formulas below, $\omega$ denotes
$\Omega_x$.

\medskip
\noindent\textbf{General transformation to the $\{R_x \neq 0\}$ chart.}
On the one hand, by applying $\sigma^*$ to the local expression \eqref{eq:local-J-log} on $\{R_y \neq 0\}$ we obtain
\begin{equation}\label{eq:log-sigma-pullback}
\sigma^{*} \omega_{xR} = \sum_{k=0}^{\lfloor m/3\rfloor}\;\sum_{j=0}^{m-3k}
\frac{\sigma^{*} (F_{j,k})}{R_x^{m-2k}}\,(y')^{m-3k-j}\Bigl(\frac{R'}{R}\Bigr)^j
\bigl(\Delta_{yR}'''\bigr)^k.
\end{equation}

On the other hand, to apply the covariance identity
\eqref{eq:log-twist-covariance}, we must rewrite $\omega_{xR}$ on the chart
$\{R_x\ne0\}$.  The next proposition gives the required form.

\begin{prop}\label{prop:log-transition}
When we change from the chart $\{R_y \neq 0\}$ to $\{R_x \neq 0\}$, the local expression \eqref{eq:local-J-log} transforms into
\begin{equation}\label{equ:omega_yR}
\omega_{yR} \coloneqq \sum_{\substack{0 \leqslant q \leqslant \lfloor m/3 \rfloor \\ 0 \leqslant p \leqslant m-3q}} 
\frac{\overline{F}_{p,q}}{R_x^{m-2q} R_y^p}\, 
\Bigl(\frac{R'}{R}\Bigr)^{\!p}\,(y')^{m-3q-p}\,
\bigl(\Delta_{yR}'''\bigr)^q,
\end{equation}
where each $\overline{F}_{p,q} \in \mathbb{C}[x,y]$ is a polynomial obtained from the $F_{j,k}$ via the transformation rules given in the proof.
\end{prop}

\begin{proof}
As in the proof of \cite[Proposition~3.1]{HHMX2026}, the total exponent of $R_x$ in the denominator multiplying 
$\bigl(\frac{R'}{R}\bigr)^{\!p}\,(y')^{m-3q-p}\,\bigl(\Delta_{yR}'''\bigr)^q$ 
is at most $m-2q$. We therefore focus on the exponent of $R_y$.

The transformation formulas for the $1$-jet and the logarithmic Wronskian are:
\begin{align*}
x' &= -\frac{R_y}{R_x} y' + \frac{R}{R_x} \frac{R'}{R}, \\[4pt]
\Delta_{xR}'''
&
\,=\,
-
\frac{R_y}{R_x}\,
\Delta_{yR}'''
+
\frac{R}{R_x^3}
(
R_{xx}R
-
R_x^2
)
\bigg(\frac{R'}{R}\bigg)^{\!3}
+
\frac{2R}{R_x^3}
(
R_{xy}R_x
-
R_yR_{xx}
)
\bigg(\frac{R'}{R}\bigg)^{\!2} y'
\\
&
\ \ \ \ \ \ \ \ \ \ \ \ \ \ \ \ \ \ \ \ \ \ \ \ \ \ \ \ \ \ \ \ \ \ \
\ \ \ \ \ \ \ \ \ \ \ \ \
+
\frac{1}{R_x^3}
(
R_y^2R_{xx}
-
2R_xR_yR_{xy}
+
R_x^2R_{yy}
)
\bigg(\frac{R'}{R}\bigg) (y')^2.
\end{align*}

Substituting these into \eqref{eq:local-J-log}, the term 
$\bigl(\frac{R'}{R}\bigr)^{\!p}\,(y')^{m-3q-p}\,\bigl(\Delta_{yR}'''\bigr)^q$ 
arises from expanding 
\begin{equation*}
\frac{F_{j,k}}{R_y^{m-2k}}\,(x')^{m-3k-j}\Bigl(\frac{R'}{R}\Bigr)^{\!j}
\bigl(\Delta_{xR}'''\bigr)^k \quad\text{with } k \geqslant q \text{ and } j \leqslant p.
\end{equation*}

In this expansion, $\bigl(\Delta_{xR}'''\bigr)^k$ contributes a factor $\bigl(\Delta_{yR}'''\bigr)^q$ together with $q$ powers of $-R_y/R_x$ from the leading term. The remaining $k-q$ factors of $\Delta_{xR}'''$ contribute additional powers of $R'/R$ and $y'$. More precisely, suppose that $\bigl(\Delta_{xR}'''\bigr)^k$ contributes $\eta$ extra powers of $R'/R$, where $k-q \leqslant \eta \leqslant 3(k-q)$. The factor $(x')^{m-3k-j}$ then contributes $p-j-\eta$ powers of $R'/R$ and 
$m-3k-j-(p-j-\eta)$ powers of $-\frac{R_y}{R_x} y'$. Consequently, the total exponent of $R_y$ in the denominator of the coefficient multiplying 
$\bigl(\frac{R'}{R}\bigr)^{\!p}\,(y')^{m-3q-p}\,\bigl(\Delta_{yR}'''\bigr)^q$ is
\begin{equation*}
    m-2k - q - \bigl(m-3k-j-(p-j-\eta)\bigr) = p + k - q - \eta \leqslant p.
\end{equation*}

The upper bound $p$ is attained by expanding the single term 
$\frac{F_{p,q}}{R_y^{m-2q}}\,(x')^{m-3q-p}\bigl(\frac{R'}{R}\bigr)^{\!p}
\bigl(\Delta_{xR}'''\bigr)^q$.

Lemma~\ref{lem:log-special-fibre} supplies the required overlap of the two
derivative charts.  The calculation above puts every coefficient over the
common denominator $R_y^p$; any cancellation can only lower its actual pole
order.  This proves the representation \eqref{equ:omega_yR}.
\end{proof}

The transition is linear in the coefficient numerators.  Hence
$F_{j,k}=x^tK_{j,k}$ implies
$\overline F_{j,k}=x^t\overline K_{j,k}$, while
$\sigma^*F_{j,k}=y^t\sigma^*K_{j,k}$.  Comparing
\eqref{eq:log-sigma-pullback} and \eqref{equ:omega_yR} through
\eqref{eq:log-twist-covariance} first gives
\[
 x^tR_y^j\sigma^*F_{j,k}
 =\varepsilon y^t\overline F_{j,k}.
\]
Cancelling $x^ty^t$ in $\mathbb C[x,y]$ yields the computational equations
\begin{equation}
 R_y^j\sigma^*K_{j,k}
 =\varepsilon\overline K_{j,k},
 \quad
 0\leqslant k\leqslant\lfloor m/3\rfloor,\quad
 0\leqslant j\leqslant m-3k.
 \label{eq:log-K-symmetry}
\end{equation}

\subsection{Modular rank computations}
\label{sec:log-computation}

For each admissible pair $(m,t)$ in the logarithmic clause of
Lemma~\ref{lem:finite-kvl-main}, the degree bounds leave finitely many
unknown coefficients in the polynomials $K_{j,k}$.  We impose two sets of
constraints:

\begin{enumerate}
    \item \textbf{Eigenvector equations from symmetry:} 
    \[
    R_y^j\,\sigma^*K_{j,k}
    = \varepsilon\,\overline{K}_{j,k},
    \]
    where $\overline{K}_{j,k}$ is obtained from the linear transformation in
    Proposition~\ref{prop:log-transition} and $\varepsilon=\pm1$.
    
    \item \textbf{Holomorphicity at infinity:} the equations expressing
    regular extension across the line $T=0$.
\end{enumerate}

\medskip
\noindent\textbf{Holomorphicity condition.}
We compactify $\mathbb{C}^2$ to $\mathbb{P}^2$ with homogeneous coordinates $[T:X:Y]$, where $T=0$ corresponds to the line at infinity. Following~\cite[Subsection~3.2]{HHMX2026}, we homogenise the polynomials $F_{j,k} = x^t K_{j,k} \in \mathbb{C}[x,y]$ to
\[
\widehat{F}_{j,k} = X^t \widehat{K}_{j,k} \in \mathbb{C}[T,X,Y],
\]
where each $\widehat{K}_{j,k}(T,X,Y)$ is homogeneous of degree
\[
\deg \widehat{K}_{j,k} = (m-2k)(d-1) - m + j + k - t,
\]
for $k = 0,\dots,\lfloor m/3\rfloor$ and $j = 0,\dots,m-3k$.

The requirement that the jet differential be holomorphic across $T=0$ is equivalent to the divisibility condition
\[
T^{\,m-3q-p} \;\Big\vert\; \mathbf{HEq}_{p,q}(\widehat{K}_{j,k}, T, X)
\quad\text{in } \mathbb{C}[T,X,Y],
\]
for every $q = 0,\dots,\lfloor m/3\rfloor$ and
$p = 0,\dots,m-3q$.  Here $\mathbf{HEq}_{p,q}$ is obtained by expanding the
transformed expression and collecting powers of $T$.  Using
$(-1)^{-3k}=(-1)^k$, it takes the form
\[
\begin{aligned}
\mathbf{HEq}_{p,q} = &
\sum_{k=q}^{\lfloor m/3\rfloor} \;
\sum_{r=\max(0,\,p+q-k)}^{\min(p,\,m-3k)} \;
\sum_{j=r}^{m-3k}
\\
& (-1)^k \; 2^{p-r} \; d^{\,j+k-p-q} \;
\binom{j}{r} \; \frac{k!}{(p-r)!\,q!\,(k+r-p-q)!} \\
& \qquad \cdot \; X^{\,m-q-j-2k} \; T^{\,k-q} \; \widehat{R}_Y^{\,2(k-q)} \; \widehat{K}_{j,k},
\end{aligned}
\]
where $\widehat{R}_Y = \partial \widehat{R}/\partial Y$ in homogeneous coordinates, and $\widehat{R}(T,X,Y)=T^d R(X/T,Y/T)$ is the homogenization of $R$.

Instead of checking divisibility directly, we truncate each $\mathbf{HEq}_{p,q}$ by discarding all terms where the exponent of $T$ is at least $m-3q-p$. More precisely, we define
\[
\mathbf{HEq}_{p,q}^{\text{(trunc)}} := \mathbf{HEq}_{p,q} \pmod{T^{\,m-3q-p}},
\]
i.e., the polynomial obtained by keeping only terms with $T$-degree
$<m-3q-p$.  The original divisibility holds exactly when
$\mathbf{HEq}_{p,q}^{\text{(trunc)}}=0$.  This is a linear system in the
coefficients of the $K_{j,k}$.

\medskip
\noindent\textbf{Modular rank computation.}
For the symmetric curve
$R=x^{11}+y^{11}+1+x^5+y^5$, the equations above define an integral sparse
matrix for each pair in the logarithmic clause of
Lemma~\ref{lem:finite-kvl-main} and each sign
$\varepsilon\in\{\pm1\}$.  The full-column-rank assertions over
$\FF_{32749}$ and their provenance are recorded in
Table~\ref{tab:log-sizes-ranks}.  A nonzero maximal minor modulo a good prime
proves full column rank over $\QQ$, and hence over $\CC$.
Consequently every coefficient $K_{j,k}$ vanishes, so
\[
H^0\bigl(\mathbb{P}^2,\; E_{2,m}T_{\mathbb{P}^2}^*(\log\mathcal{C}) \otimes \mathcal{O}(-t)\bigr) = 0
\]
for the symmetric curve $\mathcal{C}$.

\subsection{Extension to general curves}
\label{sec:log-generic}

By upper semicontinuity of $h^0$ in the universal smooth degree-$11$
family~\cite[Theorem~7.7.5, part~I]{EGA3.2_1963}, each special-fibre
vanishing holds on a Zariski-open neighbourhood of the corresponding
parameter.  The finite intersection of these open sets
gives the logarithmic clause of Lemma~\ref{lem:finite-kvl-main} for a general,
not necessarily symmetric, curve.

\section{Compact symmetry reduction}
\label{sec:compact-symmetry-proof}

For $(d,a)=(16,7)$ and $(15,7)$, consider the symmetric special surface
\begin{equation}\label{eq:compact-special-polynomial}
 R(x,y,z)=x^d+y^d+z^d+1+y^a+z^a.
\end{equation}
These surfaces give the compact cases of Lemma~\ref{lem:finite-kvl-main}
in degrees $16$ and $15$, respectively.
In each degree, the cyclic action $x\mapsto\zeta x$ for $\zeta^d=1$
commutes with the exchange of $y$ and $z$.  We impose the
involution covariance and the holomorphicity equations at infinity
on each cyclic character separately.

\subsection{Local description of invariant \texorpdfstring{$2$-jet}{2-jet} differentials}
\label{sec:cpt-local}

Use homogeneous coordinates $[T:X:Y:Z]$ on $\PP^3$. On the affine chart
$\{T\ne0\}$, take $(x,y,z)=(X/T,Y/T,Z/T)$. Let $X\subset\PP^3$ be the
projective closure of the affine surface $R(x,y,z)=0$, where
\[
 R(x,y,z)=x^d+y^d+z^d+1+y^a+z^a,
\]
which is invariant under exchanging $y$ and $z$. Its partial derivatives are
\[
 R_x=\frac{\partial R}{\partial x}=d x^{d-1},\qquad
 R_y=\frac{\partial R}{\partial y}=d y^{d-1}+a y^{a-1},\qquad
 R_z=\frac{\partial R}{\partial z}=d z^{d-1}+a z^{a-1}.
\]

\begin{lem}
\label{lem:compact-special-fibres}
For each $(d,a)\in\{(15,7),(16,7)\}$, the projective homogenization
\[
 \widehat R=X^d+Y^d+Z^d+T^d+T^{d-a}(Y^a+Z^a)
\]
defines a smooth surface. The zero sets of the restrictions of
$R_x,R_y,R_z$ to $X\cap\{T\ne0\}$ have pairwise zero-dimensional
intersections. Moreover, the common zero set of $R$, $y$ and $R_z$
in $\mathbb C^3$ is finite.
\end{lem}

The assertions follow by direct calculation.

On the open set $\{R_z \neq 0\}$, the local algebra of invariant
$2$-jet polynomials is generated by $x'$, $y'$, and the Wronskian
\[
\Delta_{yx}''' := 
\begin{vmatrix}
y' & x' \\[2mm]
y'' & x''
\end{vmatrix}
= y'\,x'' - y''\,x'.
\]

According to~\cite[Proposition 4.1]{HHMX2026}, any global section $\omega$ of $E_{2,m}T_{X}^*$ can be written on $\{R_z \neq 0\}$ as
\begin{equation}\label{eq:local-J-cpt}
J_{yx} = \sum_{k=0}^{\lfloor m/3\rfloor}\;\sum_{j=0}^{m-3k}
\frac{F_{j,k}}{R_z^{m-2k}}\,(y')^j(x')^{m-3k-j}
\bigl(\Delta_{yx}'''\bigr)^k,
\end{equation}
with polynomials $F_{j,k} \in \mathbb{C}[x,y,z]$, viewed as elements of the quotient ring $\mathfrak{R} = \mathbb{C}[x,y,z]/(R)$ (hence as regular functions on $X \cap \{T \neq 0\}$).

To determine a proper reduced representative of an element in \(\mathfrak{R}\), we employ an algorithmic reduction based on the relation \(x^d = -1 - y^d - z^d - y^a - z^a\). Every element of \(\mathfrak{R}\) can be uniquely represented by a polynomial in \(\mathbb{C}[x,y,z]\) with \(\deg_x < d\). Indeed, for any monomial \(x^k\) with \(k \geqslant d\), writing \(k = dq + r\) where \(0 \leqslant r < d\) via Euclidean division, we may replace
\[
x^k = (x^d)^q \cdot x^r = (-1 - y^d - z^d - y^a - z^a)^q \cdot x^r,
\]
which involves only powers of \(x\) with exponent \(r < d\). Applying this substitution to each monomial yields a reduced representative \(\operatorname{red}(F) \in \mathbb{C}[x,y,z]\) satisfying \(\deg_x \operatorname{red}(F) < d\) and \(\operatorname{red}(F) \equiv F \pmod{R}\).

The uniqueness of this reduced form follows from the observation that if two polynomials \(F_1, F_2\) with \(\deg_x < d\) satisfy \(F_1 \equiv F_2 \pmod{R}\), then their difference is a multiple of \(R\). Any nonzero multiple of \(R\) must have \(x\)-degree at least \(d\), since \(R\) itself contains the monomial \(x^d\) with coefficient \(1\) and no other term can cancel it after multiplication. Hence \(F_1 - F_2 = 0\), establishing uniqueness.

Therefore, we may assume without loss of generality the following:

\begin{conv}\label{convention} All polynomials $F_{j,k}$ are assumed to have $x$-degree strictly less than $d$; equivalently, every monomial appearing in them satisfies $\deg_x < d$.
\end{conv}

\subsection{Degree bounds from vanishing along the ample divisor}
\label{sec:cpt-degree}

We require the jet differential to be a section of
$E_{2,m}T_{X}^*\otimes\mathcal{O}_{X}(-t)$ with $t\geqslant1$.  It may be
represented by a section vanishing to order at least $t$ along the
hyperplane section $\{Y=0\}$.

\begin{lem}
The rational function 
$\frac{F_{j,k}}{R_{z}^{m-2k}}$ 
vanishes along the divisor $X \cap \{ y = 0 \}$ with multiplicity at least $t$ 
if and only if 
$F_{j,k} \equiv y^{t} K_{j,k} \pmod{R}$
for some polynomial $K_{j,k} \in \mathbb{C}[x,y,z]$ satisfying $\deg_x K_{j,k} < d$.
\end{lem}

\begin{proof}
By Lemma~\ref{lem:compact-special-fibres}, the common zero set of $R$, $y$
and $R_z$ is finite.  Thus $R_z$ does not vanish identically on any
irreducible component of $X\cap\{y=0\}$, so the denominator does not
affect vanishing orders along this divisor.

($\Leftarrow$) If $F_{j,k}\equiv y^tK_{j,k}\pmod R$, then the numerator
vanishes to order at least $t$ along $X\cap\{y=0\}$.  The same is therefore
true of $F_{j,k}/R_z^{m-2k}$.

($\Rightarrow$) By Convention~\ref{convention}, we may assume that
$\deg_xF_{j,k}<d$.  We first treat $t=1$.  The vanishing condition implies
that $F_{j,k}(x,0,z)$ vanishes on the plane curve
$x^d+z^d+1+z^a=0$, and hence
\[
 F_{j,k}(x,0,z)\equiv0\pmod{x^d+z^d+1+z^a}.
\]
Since $\deg_xF_{j,k}(x,0,z)<d$, this forces
$F_{j,k}(x,0,z)\equiv0$.  Thus $F_{j,k}$ is divisible by $y$ in
$\mathbb C[x,y,z]$, and its quotient still has $x$-degree less than $d$.
This proves the claim for $t=1$.

For $t>1$, the rational function $(F_{j,k}/y)/R_z^{m-2k}$ vanishes along
$X\cap\{y=0\}$ to order at least $t-1$.  Repeating the preceding argument
therefore gives
\[
 F_{j,k}=y^tK_{j,k},\qquad
 K_{j,k}\in\mathbb C[x,y,z],\qquad \deg_xK_{j,k}<d,
\]
which implies the required congruence modulo $R$.
\end{proof}

The transition to the chart $\{X\ne0\}$ gives the degree bounds for
$F_{j,k}=y^tK_{j,k}$; see \cite[\S4.4]{HHMX2026}.  Explicitly,
\[
\deg K_{j,k} \leqslant (m-2k)(d-1) - m + \Twist
\quad
\text{and}
\quad
\deg_x K_{j,k} < d,
\]
where $\Twist=-t$.  Put
\[
 D_k=(d-2)m-2(d-1)k-t.
\]
For fixed $k$, the number of monomials in $x,y,z$ of total degree at most
$D_k$ and $x$-degree less than $d$ is
$\binom{D_k+3}{3}-\binom{D_k-d+3}{3}$, where $\binom{a}{b}=0$ when
$a<b$.  The subtraction imposes the normal form $\deg_xK_{j,k}<d$.
Summing over the $m-3k+1$ possible values of $j$ gives the number of
candidate coefficients before imposing the symmetry equations and the
holomorphicity conditions at infinity:
\begin{equation}
 N^{\mathrm{cpt}}_{d}(m,t)
  =\sum_{k=0}^{\lfloor m/3\rfloor}(m-3k+1)
  \left\{\binom{D_k+3}{3}-\binom{D_k-d+3}{3}\right\}.
  \label{eq:o3-compact-variable-count}
\end{equation}
For $d=15$, this gives $N^{\mathrm{cpt}}_{15}(27,7)=78\,250\,775$ and
$N^{\mathrm{cpt}}_{15}(11,3)=2\,732\,550$.  For the compact clause of
Lemma~\ref{lem:finite-kvl-main}, we use $d=15,16$ on the corresponding
special fibres of~\eqref{eq:compact-special-polynomial}.

\subsection{The commuting cyclic and involution symmetries}
\label{sec:cpt-eigen}

\subsubsection{The action by roots of unity}
In each degree $d\in\{15,16\}$ the special surface has the cyclic symmetry
\begin{equation}
 \tau_\zeta(x,y,z)=(\zeta x,y,z),\qquad \zeta^d=1.
 \label{eq:compact-cyclic-action}
\end{equation}
It preserves $R$ in~\eqref{eq:compact-special-polynomial} and commutes with
$\sigma(x,y,z)=(x,z,y)$.  Both extend to projective automorphisms fixing
$T$, so the standard linearization of $\OO_X(-t)$ gives an action of
\[
 G=\mu_d\times\ZZ/2\ZZ
\]
on the negatively twisted section space.
Here $\mu_d=\{\zeta\in\mathbb C:\zeta^d=1\}$ is the multiplicative
group of $d$-th roots of unity, and $\ZZ/2\ZZ$ is the additive group of
integers modulo $2$.  The first
factor acts through $\zeta\mapsto\tau_\zeta$, while the nonzero class in
the second factor acts by $\sigma$.

\begin{prop}
\label{prop:compact-cyclic-eigenvector}
If $H^0(X,E_{2,m}T_X^*\otimes\OO_X(-t))\ne0$, it contains a nonzero
section $s$ and a pair $(r,\varepsilon)\in\ZZ/d\ZZ\times\{\pm1\}$ with
\begin{equation}
 \tau_\zeta^*s=\zeta^r s\quad(\zeta^d=1),\qquad
 \sigma^*s=\varepsilon s.
 \label{eq:compact-cyclic-simultaneous}
\end{equation}
\end{prop}

\begin{proof}
The two operators commute and have finite order.  Their minimal
polynomials divide $q^d-1$ and $q^2-1$, which have no repeated roots
over $\CC$.  Decomposing first for one operator and then for the other
therefore gives a direct sum of simultaneous character spaces.  A
nonzero section space has a nonzero such summand.
\end{proof}

\subsubsection{The residue filter on the coefficient unknowns}
Inject $s$ into the untwisted bundle by $\Omega_y=Y^t s$.
Because $\tau_\zeta^*Y=Y$, the cyclic eigencondition is unchanged for
$\Omega_y$.  In its local expression~\eqref{eq:local-J-cpt} put
$F_{j,k}=y^tK_{j,k}$.  The $x$-weights modulo $d$ are
\[
 \operatorname{wt}_x(x)=\operatorname{wt}_x(x')
 =\operatorname{wt}_x(x'')=\operatorname{wt}_x(\Delta_{yx}''')=1,
\]
whereas $y,z$, their derivatives, and $R_z$ have weight zero.
Thus a coefficient monomial $x^\alpha y^\beta z^\gamma$ in $K_{j,k}$
has cyclic weight $\alpha$, while the corresponding jet factor
$(y')^j(x')^{m-3k-j}(\Delta_{yx}''')^k$ has cyclic weight
$(m-3k-j)+k$.  The full term in~\eqref{eq:local-J-cpt} therefore has
total cyclic weight
\begin{equation}
 \alpha+(m-3k-j)+k=\alpha+m-2k-j\pmod d.
 \label{eq:compact-cyclic-weight}
\end{equation}

Fix $r\in\ZZ/d\ZZ$.  By the weight formula~\eqref{eq:compact-cyclic-weight}
and the uniqueness of the local expansion~\eqref{eq:local-J-cpt} in
normal form, the cyclic eigencondition for $s$
in~\eqref{eq:compact-cyclic-simultaneous} holds if and only if every monomial
$x^\alpha y^\beta z^\gamma$ occurring with a nonzero coefficient in
$K_{j,k}$ satisfies
\begin{equation}
 \alpha\equiv r-m+2k+j\pmod d.
 \label{eq:compact-cyclic-ansatz}
\end{equation}
The monomials in $K_{j,k}$ are subject to the normal-form condition of
Convention~\ref{convention} and the degree bounds of
Subsection~\ref{sec:cpt-degree}, which read
\begin{equation}
 0\leqslant\alpha<d,\qquad \beta,\gamma\geqslant0,\qquad
 \alpha+\beta+\gamma\leqslant(d-2)m-2(d-1)k-t.
 \label{eq:compact-cyclic-monomial-bounds}
\end{equation}
For fixed $j,k,r$, the congruence~\eqref{eq:compact-cyclic-ansatz} therefore
allows at most one value of $\alpha$ among the exponents permitted
by~\eqref{eq:compact-cyclic-monomial-bounds}.
The involution covariance equations~\eqref{eq:coefficient_match} and the
holomorphicity conditions~\eqref{eq:trunc} will be imposed separately in
Subsection~\ref{sec:cpt-computation}.

For example, when $m=3$, a monomial $x^\alpha y^\beta z^\gamma$ in
$K_{j,k}$ contributes cyclic weight $\alpha$.  The following table
separates the weight of its jet factor in~\eqref{eq:local-J-cpt} from the
total weight of the corresponding full term.  All weights are read
modulo $d$.
\begin{equation}
\begin{array}{c|c|c|c}
 \text{coefficient}&\text{jet factor}&\text{jet weight}&\text{total weight}\\\hline
 K_{0,0}&(x')^3&3&\alpha+3\\
 K_{1,0}&y'(x')^2&2&\alpha+2\\
 K_{2,0}&(y')^2x'&1&\alpha+1\\
 K_{3,0}&(y')^3&0&\alpha\\
 K_{0,1}&\Delta_{yx}'''&1&\alpha+1
\end{array}
\label{eq:compact-cyclic-m3-check}
\end{equation}
The factor $y^t$ and the denominators $R_z^{m-2k}$ have weight zero.
Thus the last column of~\eqref{eq:compact-cyclic-m3-check} gives the total
cyclic weights in~\eqref{eq:compact-cyclic-weight}.

\begin{rem}
\label{rem:compact-cyclic-integer-arithmetic}
The affine and homogeneous quotient relations
\begin{align}
 x^d&=-(y^d+z^d+1+y^a+z^a),\notag\\
 X^d&=-\bigl(Y^d+Z^d+T^d+T^{d-a}(Y^a+Z^a)\bigr)
 \label{eq:compact-cyclic-reduction}
\end{align}
preserve cyclic weights modulo $d$.  Since the involution commutes with
the cyclic action, its covariance equations preserve each cyclic
character subspace.  Moreover, by~\eqref{eq:compact-cyclic-ansatz}, every term of
$\mathbf{HEq}_{\ell,k}$ in~\eqref{eq:HEq_def} has homogeneous $X$-weight
\[
 \alpha+m-3k-j\equiv r-k\pmod d,
\]
independently of $j$.  Thus reduction modulo $\widehat R$ and truncation
in $T$ preserve the character decomposition.

For each $r$, the congruence~\eqref{eq:compact-cyclic-ansatz} thus selects
the columns of the covariance equations~\eqref{eq:coefficient_match} and
the holomorphicity equations~\eqref{eq:trunc}.  After clearing
denominators the selected system has integer coefficients.  The roots
of unity enter the decomposition over $\CC$, and the modular rank
computations use only the resulting integer systems.
\end{rem}

\subsubsection{The involution covariance in both degrees}
In either degree, fix a cyclic character $r$ and a sign
$\varepsilon\in\{\pm1\}$.  For an intrinsic eigensection $s$ satisfying
\eqref{eq:compact-cyclic-simultaneous}, put $\Omega_y=Y^t s$.
The involution acts on the jets by
\[
 \sigma^*(x')=x',\quad\sigma^*(y')=z',\quad
 \sigma^*(x'')=x'',\quad\sigma^*(y'')=z'',\quad
 \sigma^*(\Delta_{yx}''')=\Delta_{zx}'''.
\]
Its condition must be imposed on the intrinsic twisted section $s$.
Since $\sigma^*Y=Z$, the injected section satisfies the covariance
identity below, rather than an untwisted eigenidentity.

\begin{equation}
 y^t\sigma^*\Omega_y=\varepsilon z^t\Omega_y.
 \label{eq:compact-twist-covariance}
\end{equation}
We impose this covariance after writing both sides on $\{R_y\ne0\}$.
In the transition formulas below, $\omega$ denotes $\Omega_y$.

First, by applying $\sigma^*$ directly to the local expression~\eqref{eq:local-J-cpt} on the chart $\{R_y \neq 0\}$ we obtain an immediate pullback formula:
\begin{equation}\label{equ:sigma_pullback_omegayx}
\sigma^{*} \omega = \sum_{k=0}^{\lfloor m/3\rfloor}\;\sum_{j=0}^{m-3k}
\frac{\sigma^{*} (F_{j,k})}{R_y^{m-2k}}\,(z')^j(x')^{m-3k-j}
\bigl(\Delta_{zx}'''\bigr)^k.
\end{equation}

Second, we must rewrite the original $\omega$ on the same chart.

\begin{prop}
\label{prop:transition}
Let $X$ be either compact special fibre of
Lemma~\ref{lem:compact-special-fibres}, and let $\omega$ be given on
$\{R_z\neq0\}$ by~\eqref{eq:local-J-cpt}.  After transition to
$\{R_y\neq0\}$, the same differential form admits the local expression
\begin{equation}\label{equ:omega_zx}
\omega = \sum_{\substack{0 \leqslant q \leqslant \lfloor m/3 \rfloor \\ 0 \leqslant p \leqslant m-3q}} 
\frac{\overline{F}_{p,q}}{R_y^{m-2q} R_z^p}\, 
(z')^{m-3q-p}\,(x')^{p}\,
\bigl(\Delta_{zx}'''\bigr)^q,
\end{equation}
where each $\overline{F}_{p,q}$ is a polynomial in $\mathbb{C}[x,y,z]$, regarded as an element of the quotient ring $\mathfrak{R} = \mathbb{C}[x,y,z]/(R)$.  The coefficients $\overline{F}_{p,q}$ are explicit linear combinations of the original $F_{j,k}$, with coefficients that are polynomials in the partial derivatives of $R$.
\end{prop}

\begin{proof}
As in \cite[Proposition~4.1]{HHMX2026}, the maximal admissible exponent of $R_y$ in the coefficients multiplying $(x')^{p}\,(z')^{m-3q-p}\,\bigl(\Delta_{zx}'''\bigr)^q$ is $m-2q$. We only need to consider the maximal admissible exponent of $R_z$ in the coefficients.

Following a similar computation in~\cite[Subsection~3.1]{HHMX2026}, the
$1$-jet and the Wronskian transform as
\begin{align*}
y' &= -\,\frac{R_x}{R_y}\,x' -\,\frac{R_z}{R_y}\,z', \\[4pt]
\Delta_{yx}'''
&
\,=\,
-\,
\frac{R_z}{R_y}\,
\Delta_{zx}'''
\\
&
\ \ \ \ \
+
z'z'x'\,
\bigg(
\frac{R_{zz}}{R_y}
-
2\,\frac{R_z}{R_y}\,\frac{R_{yz}}{R_y}
+
\frac{R_z}{R_y}\,\frac{R_z}{R_y}\,\frac{R_{yy}}{R_y}
\bigg)
\notag
\\
&
\ \ \ \ \
+
z'x'x'\,
\bigg(
2\,\frac{R_{xz}}{R_y}
-
2\,\frac{R_z}{R_y}\,\frac{R_{xy}}{R_y}
-
2\,\frac{R_x}{R_y}\,\frac{R_{yz}}{R_y}
+
2\,\frac{R_x}{R_y}\,\frac{R_z}{R_y}\,\frac{R_{yy}}{R_y}
\bigg)
\notag
\\
&
\ \ \ \ \
+
x'x'x'\,
\bigg(
\frac{R_{xx}}{R_y}
-
2\,\frac{R_x}{R_y}\,\frac{R_{xy}}{R_y}
+
\frac{R_x}{R_y}\,\frac{R_x}{R_y}\,\frac{R_{yy}}{R_y}
\bigg).
\end{align*}

Substituting these into \eqref{eq:local-J-cpt}, we see that the term 
\((x')^{p}\,(z')^{m-3q-p}\,\bigl(\Delta_{zx}'''\bigr)^q\) 
arises from expanding terms of the form
\begin{equation*}
\frac{F_{j,k}}{R_z^{m-2k}}\,(x')^{m-3k-j}(y')^{j}
\bigl(\Delta_{yx}'''\bigr)^k \quad\text{with } k \geqslant q \text{ and } m-3k-j \leqslant p.
\end{equation*}

When expanding such a term, \(\bigl(\Delta_{yx}'''\bigr)^k\) contributes a factor of 
\(\bigl(\Delta_{zx}'''\bigr)^q\) together with \(q\) powers of \(-R_z/R_y\) from the leading term. 
We may assume that \(\bigl(\Delta_{yx}'''\bigr)^k\) also contributes an additional \(\eta\) powers of \(x'\) from the remaining cubic terms, where \(\eta\) satisfies 
\(k-q \leqslant \eta \leqslant 3(k-q)\) and \(m-3k-j + \eta \leqslant p\). 
The factor \((y')^{j}\) contributes \(p-(m-3k-j + \eta)\) powers of \(x'\) and 
\(j-(p-(m-3k-j + \eta))\) powers of \(-\frac{R_z}{R_y} z'\). 
Hence, the total exponent of \(R_z\) in the denominator of the coefficient multiplying 
\((x')^{p}\,(z')^{m-3q-p}\,\bigl(\Delta_{zx}'''\bigr)^q\) satisfies
\begin{equation*}
    m-2k - q - \bigl(j-(p-(m-3k-j + \eta))\bigr) = p + k - q - \eta \leqslant p.
\end{equation*}

Thus every coefficient can be written over the common denominator $R_z^p$.
The pairwise support statement in
Lemma~\ref{lem:compact-special-fibres} ensures that the derivative charts
have no hidden common divisorial component; any coefficientwise cancellation
only lowers the actual pole order.  This proves the claimed expression.
\end{proof}

The transition is linear in the coefficient numerators.  Hence
$F_{j,k}=y^tK_{j,k}$ implies
$\overline F_{p,q}=y^t\overline K_{p,q}$, while
$\sigma^*F_{j,k}=z^t\sigma^*K_{j,k}$.  Comparing
\eqref{equ:sigma_pullback_omegayx} and \eqref{equ:omega_zx} through
\eqref{eq:compact-twist-covariance} first gives, with $p=m-3k-j$,
\[
 y^tR_z^p\sigma^*F_{j,k}
 =\varepsilon z^t\overline F_{p,k}
 \quad\text{in }\mathbb C[x,y,z]/(R).
\]
The quotient is a domain by Lemma~\ref{lem:compact-special-fibres}; cancelling
$y^tz^t$ gives the fundamental computational system
\begin{equation}\label{eq:coefficient_match}
 R_z^{\,m-3k-j}\sigma^*K_{j,k}
 -\varepsilon\overline K_{m-3k-j,k}
 \equiv0\pmod R.
\end{equation}
These equations hold for
$0\leqslant k\leqslant\lfloor m/3\rfloor$ and
$0\leqslant j\leqslant m-3k$.

\subsection{Finite systems and modular ranks}
\label{sec:cpt-computation}

For the pairs $(m,t)$ in~\eqref{eq:compact16-targets}
and~\eqref{eq:compact15-targets}, the degree bounds in
Subsection~\ref{sec:cpt-degree} leave finitely many unknown
coefficients.  The complete obstruction system has two parts:
\begin{enumerate}
    \item \textbf{Covariance equations from symmetry:} the equations derived
    in Section~\ref{sec:cpt-eigen}, which encode the intrinsic eigensection
    condition after the twist injection;
    \item \textbf{Holomorphicity at infinity:} the equations expressing
    regular extension of $\omega$ across $\{T=0\}$.
\end{enumerate}
We treat each part in turn and then describe the computational solution of the resulting linear system.

\medskip\noindent
\textbf{Symmetry constraints.}
According to \eqref{eq:coefficient_match}, for each
$0\leqslant k\leqslant\lfloor m/3\rfloor$ and
$0\leqslant j\leqslant m-3k$ we impose
\[
\mathbf{Eq}_{j,k}
\coloneqq R_z^{\,m-3k-j}\sigma^*K_{j,k}
-\varepsilon\overline K_{m-3k-j,k}
\equiv0\pmod R,
\]
where the polynomials $\overline K_{p,q}$ are induced by the linear
transformation in Proposition~\ref{prop:transition}.

\medskip\noindent
\textbf{Holomorphicity at infinity.}
It remains to check that the jet differential~\eqref{eq:local-J-cpt} extends holomorphically along the divisor $\{T = 0\}$ at infinity.  Following~\cite[Section~4]{HHMX2026}, we homogenise the local data.

Let $d = \deg R$ be the total degree of the defining polynomial.  For every index pair $(j,k)$ we factor the local coefficient as
\[
F_{j,k} = y^{t}\, K_{j,k}, \qquad K_{j,k} \in \mathbb{C}[x,y,z],
\]
and lift each $K_{j,k}$ to a homogeneous polynomial $\widehat{K}_{j,k} \in \mathbb{C}[T,X,Y,Z]$ of degree
\[
\deg \widehat{K}_{j,k} = (m-2k)(d-1) - m - t,
\]
with the additional requirement $\deg_X \widehat{K}_{j,k} < d$.  The corresponding homogeneous lift of $F_{j,k}$ is then
\[
\widehat{F}_{j,k} = Y^{t}\, \widehat{K}_{j,k}.
\]
This factorisation separates the prescribed vanishing along the
hyperplane section from the remaining coefficients of the polynomials
$\widehat{K}_{j,k}$.

As shown in~\cite[Subsection~4.4]{HHMX2026}, the holomorphicity conditions translate into the following divisibility requirements in the homogeneous quotient ring $\mathbb{C}[T,X,Y,Z] / (\widehat{R})$:
\begin{equation}\label{eq:divisibility}
T^{\,m-3k-\ell} \;\Big\vert\; \mathbf{HEq}_{\ell,k}(\widehat{K}_{j,k}, X, Y)
\quad \text{in } \mathbb{C}[T,X,Y,Z] / (\widehat{R}),
\end{equation}
for every $k = 0, \dots, \lfloor m/3 \rfloor$ and $\ell = 0, \dots, m-3k$, where the auxiliary polynomials $\mathbf{HEq}_{\ell,k}$ are defined by
\begin{equation}\label{eq:HEq_def}
\mathbf{HEq}_{\ell,k}(\widehat{K}_{j,k}, X, Y)
\;=\;
\sum_{j=\ell}^{m-3k} \widehat{K}_{j,k} \binom{j}{\ell} X^{m-3k-j}\, Y^{\,j-\ell}.
\end{equation}

\medskip\noindent
\textbf{Algebraic reduction.}
The two systems use the same quotient relation at different stages.  The
symmetry equation $\mathbf{Eq}_{j,k}$ is first reduced in the affine quotient
$\mathbb C[x,y,z]/(R)$ by $x^d=-1-y^d-z^d-y^a-z^a$; its coefficient
identities may then be homogenised to a common degree.  The polynomial
$\mathbf{HEq}_{\ell,k}$ encoding holomorphicity at infinity is homogeneous
from the outset and is reduced in
$\mathbb C[T,X,Y,Z]/(\widehat R)$ by the relation
\[
 X^d=-\bigl(Y^d+Z^d+T^d+T^{d-a}(Y^a+Z^a)\bigr),
\]
which follows from the homogeneous defining equation in
Lemma~\ref{lem:compact-special-fibres}.  The reduction is repeated until
its representative has $X$-degree $<d$.  Both procedures are the
homogeneous and affine forms of the unique normal form fixed in
Convention~\ref{convention}; they are kept distinct in the implementation.

Expanding the reduced affine symmetry identities gives the first homogeneous
linear system in the coefficients of the $K_{j,k}$.

The divisibility condition~\eqref{eq:divisibility} is equivalent to the vanishing of the truncated reduced polynomial
\begin{equation}\label{eq:trunc}
\mathbf{HEq}_{\ell,k}^{\text{(trunc)}}
\;:=\;
\operatorname{red}(\mathbf{HEq}_{\ell,k})
\pmod{T^{\,m-3k-\ell}},
\end{equation}
obtained by discarding all terms of total $T$-degree at least
$m-3k-\ell$ from the reduced representative.  This converts the geometric
divisibility constraints into a second explicit homogeneous linear system
in the coefficients of the polynomials $\widehat{K}_{j,k}$.

\subsubsection{Degree-$16$ character blocks and modular ranks}
\label{sec:cpt16-computation}

Set $(d,a)=(16,7)$ in~\eqref{eq:compact-special-polynomial}.  Its
homogenization is the surface in~\eqref{eq:compact-compat-d15-surface}.
For each pair in~\eqref{eq:compact16-targets}, and for each $r=0,\ldots,15$
and $\varepsilon\in\{+1,-1\}$, the coefficient list is
\begin{equation}\label{eq:compact16-ansatz}
 K_{j,k}=\sum_{\substack{0\leqslant\alpha<16,\ \beta,\gamma\geqslant0\\
                         \alpha+\beta+\gamma\leqslant D_k\\
                         \alpha+m-2k-j\equiv r\pmod{16}}}
 c_{\alpha\beta\gamma}^{j,k}x^\alpha y^\beta z^\gamma,
 \qquad D_k=14m-30k-t.
\end{equation}
On each such list, the covariance equations~\eqref{eq:coefficient_match}
and the holomorphicity equations at infinity~\eqref{eq:trunc} give an
integral homogeneous system as in
Remark~\ref{rem:compact-cyclic-integer-arithmetic}.  The normal-form
coefficients of a simultaneous eigensection give a vector in its
kernel and determine the section by~\eqref{eq:local-J-cpt}.

Section~\ref{app:compact-compat-d15} uses the degree parameter $15$ and
the signed twists $-2$ and $-3$ for these systems.  For each of the two
weights, it tests the sixteen cyclic characters with both involution
signs.  Tables~\ref{tab:compact-compat-d15-summary}
and~\ref{tab:compact-compat-d15-details} report full column rank over
$\FF_{32749}$ for all $64$ systems.  In each system a maximal minor
nonzero modulo $32749$ is a nonzero integer, so the kernel over $\CC$
is zero.  Proposition~\ref{prop:compact-cyclic-eigenvector} now gives
\[
 H^0\!\left(X,E_{2,m}T_X^*\otimes\OO_X(-t)\right)=0
\]
on the smooth degree-$16$ special surface for both pairs.

\subsubsection{Degree-$15$ character blocks and modular ranks}
\label{sec:cpt15-certificates}

We now return to $(d,a)=(15,7)$ and the compact clause of
Lemma~\ref{lem:finite-kvl-main}.  For each fixed cyclic character $r$,
the finite coefficient lists are formed using the
bounds~\eqref{eq:compact-cyclic-monomial-bounds}, and then restricted by the
congruence~\eqref{eq:compact-cyclic-ansatz}.
If no monomial satisfies both~\eqref{eq:compact-cyclic-monomial-bounds}
and~\eqref{eq:compact-cyclic-ansatz} for some $j,k$, the corresponding
$K_{j,k}$ is zero and contributes no unknowns.
The covariance equations and the holomorphicity equations at infinity
are assembled on these restricted coefficient lists.  Their construction
is described in Subsection~\ref{app:compact-methods}, and their modular
rank results are reported in Tables~\ref{tab:compact-general-sizes}
and~\ref{tab:compact-general-details}.

\medskip\noindent
\textbf{Block sizes.}
Let $a_{rjk}$ be the representative in $\{0,\ldots,14\}$ of
$r-m+2k+j$ modulo $15$.  For fixed $j,k,r$ this is the only possible
$x$-exponent in the normal form.  Consequently the number of unknowns in
one residue sector is exactly
\begin{equation}
 N_r(m,t)=\sum_{k=0}^{\lfloor m/3\rfloor}
          \sum_{j=0}^{m-3k}
          \binom{D_k-a_{rjk}+2}{2},\qquad D_k=13m-28k-t,
 \label{eq:compact-mu15-block-count}
\end{equation}
where the summand is zero if $D_k<a_{rjk}$.  Summing over all residues
recovers~\eqref{eq:o3-compact-variable-count}:
\[
 \sum_{r=0}^{14}N_r(m,t)=N^{\mathrm{cpt}}_{15}(m,t).
\]
The counts in Table~\ref{tab:compact-mu15-block-counts} are candidate
columns before the equations are imposed; each sign uses this same
candidate list.  No additional factor of two is claimed for the column
reduction, and no equality of the fifteen sizes is assumed.  In
Table~\ref{tab:compact-general-sizes}, $C_{\Sigma}$ sums the columns over
both signs, so $C_{\Sigma}=2\sum_{r=0}^{14}N_r(m,t)$.  The corresponding
twist parameter in that table is the negative of $t$ here.

\begin{table}[htbp]
\centering
\small
\begin{tabular}{c r r r}
\hline
$(m,t)$ & $\sum_rN_r$ & $\min_rN_r$ & $\max_rN_r$\\
\hline
 $(3,1)$ & $32\,526$ & $1\,460$ & $2\,935$ \\
 $(4,1)$ & $83\,195$ & $4\,249$ & $6\,864$ \\
 $(5,1)$ & $176\,175$ & $9\,732$ & $13\,701$ \\
 $(6,1)$ & $328\,620$ & $19\,059$ & $24\,555$ \\
 $(7,2)$ & $547\,260$ & $32\,784$ & $39\,756$ \\
 $(8,3)$ & $858\,240$ & $52\,719$ & $60\,984$ \\
 $(9,3)$ & $1\,313\,360$ & $82\,081$ & $92\,256$ \\
 $(10,3)$ & $1\,927\,990$ & $122\,070$ & $133\,996$ \\
 $(11,3)$ & $2\,732\,550$ & $174\,985$ & $188\,080$ \\
\hline
\end{tabular}
\caption{Degree-$15$ compact candidate columns, split by cyclic character.
Each residue list is tested with both involution signs.}
\label{tab:compact-mu15-block-counts}
\end{table}

\medskip\noindent
\textbf{Modular rank results.}
For each degree-$15$ target $(m,t)$ in~\eqref{eq:compact15-targets}, and
each of the $30$ pairs
\[
 r=0,\ldots,14,\qquad\varepsilon\in\{+1,-1\},
\]
generate only the coefficients of monomials satisfying
\eqref{eq:compact-cyclic-monomial-bounds}
and~\eqref{eq:compact-cyclic-ansatz}.
Construct the covariance equations~\eqref{eq:coefficient_match} and the truncated
holomorphicity equations at infinity~\eqref{eq:trunc} in this candidate list,
using the affine and homogeneous normal forms at their respective
stages.  This gives an integral sparse matrix
$M_{m,t}^{r,\varepsilon}$ after denominators are cleared.

Empty-column blocks have zero kernel without a rank computation.  For
every nonempty block, full column rank modulo a good prime proves full
column rank over $\QQ$, and hence over $\CC$, by lifting a nonzero
maximal minor.  The modular rank results for the individual blocks are
recorded in Table~\ref{tab:compact-general-details}.  No root of unity is
required in this finite-field computation.

The reported modular rank output is full for every nonempty block in
all the listed cases and all thirty character pairs.
Thus every simultaneous character space is zero.  Proposition
\ref{prop:compact-cyclic-eigenvector} then gives
\[
 H^0\bigl(X,E_{2,m}T_X^*\otimes\OO_X(-t)\bigr)=0
\]
for the smooth special surface
$x^{15}+y^{15}+z^{15}+1+y^7+z^7=0$.  The character decomposition
and the reported modular rank computations establish all the required
degree-$15$ special-fibre vanishings.

\subsection{Extension to general surfaces}
\label{sec:cpt-generic}

By upper semicontinuity of $h^0$ in the universal smooth
family~\cite[Theorem~7.7.5, part~I]{EGA3.2_1963}, each special-fibre
vanishing holds on a Zariski-open neighbourhood of the corresponding
parameter.  At $d=16$, the completed verification in
Section~\ref{sec:cpt16-computation} and intersection of the two open sets
prove the degree-$16$ case of Lemma~\ref{lem:finite-kvl-main} for general surfaces.
At $d=15$, the verified full-rank computations of
Section~\ref{sec:cpt15-certificates} prove the degree-$15$ case by the
same semicontinuity argument.

\section{Proof of hyperbolicity theorems}
\label{sec:hyperbolicity}

We use the geometric and Nevanlinna framework of~\cite{HHMX2026}.
In the compact case, we use pole-three slanted vector fields to construct
the relative $1$-jet loci in both degrees $15$ and $16$.  The logarithmic
case treats the degree-$11$ endpoint and the estimate in
Theorem~\ref{thm:log-11-endpoint}.

\subsection{Compact case}
\label{sec:compact-proof}

We first recall the two classical results used
in~\cite[Theorems~1.1--1.2]{HHMX2026}.

\begin{thm}[Clemens and Xu; {\cite[Theorem~1.1 and the subsequent remark]{HHMX2026}}]
\label{thm:Clemens-Xu}
A very general surface $X\subset\PP^3$ contains no rational curves if
$\deg X\geqslant5$, and no elliptic curves if $\deg X\geqslant6$.
The latter degree bound improves to $\deg X\geqslant5$ by Xu's theorem.
\end{thm}

\begin{thm}[McQuillan; {\cite[Theorem~1.2]{HHMX2026}}]
\label{thm:McQuillan}
A parabolic leaf of an algebraic multi-foliation on a surface of general
type is algebraically degenerate.
\end{thm}

The original references and the application of these theorems to the
$2$-jet zero locus are given in~\cite[Section~1.3]{HHMX2026}.
The required small-slope criterion is the following.

\begin{prop}[{\cite[Proposition~6.1]{HHMX2026}}]
\label{prop:compact-zero-locus}
Let $X\subset\PP^3$ be a smooth degree-$d$ surface of general type with
$\operatorname{Pic}(X)\cong\ZZ$.  Suppose that an irreducible section
\[
 0\ne\omega_1\in
 H^0\!\left(X,E_{2,m}T_X^*\otimes\OO_X(-t)\right),
 \qquad m,t\in\ZZ_{>0},
\]
satisfies
\begin{equation}\label{eq:compact-zero-locus-slope}
 t<(d-4)\frac{13c_1^2-9c_2}{12c_1^2}\,m.
\end{equation}
Then $\OO_{X_2}(1)|_Z$ is big, where $Z\subset X_2$ is the irreducible
three-dimensional component of $\{\omega_1=0\}$ which dominates $X_1$
and is distinct from $\Gamma_2$.
\end{prop}

We refer to~\cite[Proposition~6.1]{HHMX2026} for the proof, including the
stronger bigness statement for $\OO_{X_2}(2,1)|_Z$, where
$\OO_{X_2}(a_1,a_2)=\pi_{2,1}^*\OO_{X_1}(a_1)\otimes\OO_{X_2}(a_2)$.
The construction and extension of a negatively twisted section on $Z$,
its vanishing on the lifted entire curve, and the resulting algebraic
degeneracy of the first lift are those of~\cite[Section~6.2]{HHMX2026}.
These arguments apply to the relative reduced-divisor family of
Proposition~\ref{prop:o3-compact-relative-family}.

\begin{proof}[Proof of Theorem~\ref{thm:compact-general-upgrade}]
Fix $d\in\{15,16\}$ and take the algebraic family
$\{\omega_{1,a}\}_{a\in U^{m,t}}$ of
Proposition~\ref{prop:o3-compact-relative-family}. By the relative
construction in~\cite[Section~6.2]{HHMX2026}, the divisors $Z_a$ in
\eqref{eq:o3-relative-divisor} form a flat effective Cartier divisor
$\mathcal Z\subset\mathcal X_{2,U}$ with integral three-dimensional
fibres over a nonempty Zariski-open $U\subset U^{m,t}$.
We restrict $U$ to the line-free locus of~\cite[Remark~5.2]{HHMX2026}.
Write $\mathcal X_{i,U}$ for the $i$-th relative Semple space, with
$\mathcal X_{0,U}=\mathcal X_U$ the universal surface over $U$.
For $0\leqslant j<i\leqslant2$, write
\[
 \Pi_{i,j}\colon\mathcal X_{i,U}\longrightarrow\mathcal X_{j,U},
 \qquad
 \pi_{i,j}\colon X_{a,i}\longrightarrow X_{a,j}
\]
for the projection from level $i$ to level $j$ and its restriction
to the fibre over $a\in U$, respectively, with $X_{a,0}=X_a$.

Since $b_2>0$, $Z_a$ has positive degree on a general fibre of
$\pi_{2,1}$ and dominates $X_{a,1}$. Also $m=b_1+b_2\geqslant3$.
By the fundamental vanishing theorem
\cite{Green-Griffiths1980,SiuYeung1996Bloch,Demailly1997}, the second lift
of a nonconstant entire curve $f\colon\CC\to X_a$ lies in
$Z_a\cup\Gamma_{a,2}$. It avoids $\Gamma_{a,2}$ wherever $f'\ne0$,
so closedness gives $f_{[2]}(\CC)\subset Z_a$.

Let $\rho_d$ be the threshold for $t/m$ in
\eqref{eq:compact-zero-locus-slope}. Substituting the Chern numbers
$c_1^2=d(d-4)^2$ and $c_2=d(d^2-4d+6)$ gives
\begin{equation}\label{eq:compact-slope-thresholds}
 \rho_{15}=\frac{17}{66},\qquad \rho_{16}=\frac58.
\end{equation}
If $t/m<\rho_d$,
apply Proposition~\ref{prop:compact-zero-locus} on a very general
fibre, where $\operatorname{Pic}(X_a)=\ZZ[\OO_{X_a}(1)]$.
The construction in~\cite[Section~6.2]{HHMX2026} then gives, after
shrinking $U$, a section $s$ of
\[
 \left.
 \left(\OO_{\mathcal X_{2,U}}(\ell)
       \otimes\Pi_{2,0}^*\OO_{\mathcal X_U}(-1)\right)
 \right|_{\mathcal Z}
\]
for some $\ell>0$. Its fibre restrictions are nonzero. The
restricted-section vanishing argument of Demailly--El~Goul
\cite[Section~4, proof of the Main Theorem]{Demailly-Elgoul2000}
gives $f_{[2]}(\CC)\subset V(s_a)$.

If $t/m\geqslant\rho_d$ and $t\geqslant4$, restrict $U$ to the open set
$S^{\rm red}$ of Lemma~\ref{lem:simultaneous-open} and apply
Theorem~\ref{thm:apply O(3) fields cpt}.  After shrinking $U$, it gives
an algebraic family $\{\omega_{2,a}\}$ of twist $-(t-3)$.
Set $s=\widetilde\omega_2|_{\mathcal Z}$.  The independence conclusion
gives $s_a\not\equiv0$ on $Z_a$.  Applying the fundamental vanishing
theorem again to $\omega_{2,a}$ gives
\[
 f_{[2]}(\CC)\subset V(s_a)
 = Z_a\cap V(\widetilde\omega_{2,a}).
\]
The remaining pairs satisfy $m\geqslant3$ and
$\lceil\rho_d m\rceil\leqslant t\leqslant3$.  For $d=16$,
this gives $m=3,4$, and their minimal twist orders are those
in~\eqref{eq:compact16-targets}, covered by Lemma~\ref{lem:finite-kvl-main}.
For $d=15$ this gives $3\leqslant m\leqslant11$, and
Lemma~\ref{lem:finite-kvl-main} gives the vanishings
in~\eqref{eq:compact15-targets}, namely for $t=1$ when
$3\leqslant m\leqslant6$, for $t=2$ when $m=7$, and for $t=3$ when
$8\leqslant m\leqslant11$.  In each case $t\leqslant\lceil\rho_{15}m\rceil$.
In both degrees, twist monotonicity excludes every remaining pair
on a nonempty Zariski-open set.

In each case giving a section $s$, put $\mathcal W=V(s)\subset\mathcal Z$.
On a general fibre, $s_a$ is a nonzero section of a line bundle on
the integral threefold $Z_a$. Thus $W_a=V(s_a)$ is empty or has
dimension two. Since $\mathcal W$ is projective over $U$, upper
semicontinuity allows us to shrink $U$ so that $\dim W_a\leqslant2$
for every $a\in U$.

Let $\mathcal S_d$ be the reduced closed image of $\mathcal W$ under
the proper relative projection $\Pi_{2,1}$. On each fibre its
underlying set is $\pi_{2,1}(W_a)$, so
\[
 \dim(\mathcal S_d)_a\leqslant\dim W_a\leqslant2,
 \qquad
 f_{[1]}(\CC)=\pi_{2,1}\bigl(f_{[2]}(\CC)\bigr)
 \subset(\mathcal S_d)_a.
\]
After one final shrinking by generic flatness, $\mathcal S_d$ is
flat and projective over the resulting open set $U_d$.
\end{proof}

\subsubsection{Hyperbolicity in degrees $15$ and $16$}
\label{sec:compact-endpoints-proof}

\begin{proof}[Proof of the compact assertion of Theorem~\ref{thm:main-thresholds}]
Fix $d\in\{15,16\}$ and take $U_d$ and $\mathcal S_d$ from
Theorem~\ref{thm:compact-general-upgrade}. Write
$P_d=\PP(T_{\mathcal X_d/U_d})$ and let
$\pi\colon P_d\to\mathcal X_d$ be its projection.
For a sufficiently large integer $r$, the line bundle
$\mathcal A=\OO_{P_d}(1)\otimes\pi^*\OO_{\mathcal X_d}(r)$ is
relatively ample over $U_d$. At the generic point $\eta$ of $U_d$,
$(\mathcal S_d)_\eta$ is a proper closed subset of $(P_d)_\eta$.
For $q\gg1$, Serre's theorem gives a nonzero section of
$(\mathcal A^{\otimes q})_\eta$ vanishing on $(\mathcal S_d)_\eta$.
After shrinking $U_d$, it extends to a section vanishing on
$\mathcal S_d$ and nonzero on each fibre, as in
\cite[Corollary~4.2.2]{XieZhao2026}.

By the projective-bundle direct-image correspondence this gives a
relative section of
$\operatorname{Sym}^q T^*_{\mathcal X_d/U_d}\otimes
\OO_{\mathcal X_d}(rq)$. It annihilates every entire curve in a fibre
because its first lift lies in $(\mathcal S_d)_a$.
The fibres are of general type, and Theorem~\ref{thm:Clemens-Xu}
excludes rational and elliptic curves on a very general fibre.
This verifies the hypotheses of Theorem~B
of~\cite{XieZhao2026}. Their Theorem~D, which uses McQuillan's
foliation theorem (Theorem~\ref{thm:McQuillan}), gives Kobayashi
hyperbolicity on a nonempty Zariski-open subset of $U_d$.
\end{proof}

\subsection{Logarithmic case}
\label{sec:log-proof}

We use the logarithmic notation of~\cite[Sections~2.3 and~6.5]{HHMX2026}:
$C\subset\PP^2$ is the boundary divisor,
$X_1=\PP T_{\PP^2}(-\log C)$, and
$\pi_{2,1}\colon X_2\to X_1$ is the second logarithmic Semple tower,
with vertical divisor $\Gamma_2$.  As in the compact case,
$\OO_{X_2}(a_1,a_2)=\pi_{2,1}^*\OO_{X_1}(a_1)\otimes\OO_{X_2}(a_2)$.
The next three results are quoted in the forms used in that paper.

\begin{thm}[{\cite[Theorem~2.6]{HHMX2026}}]
\label{thm:log-jet-smt}
Let $C=\sum_{i=1}^qC_i$ be a simple normal crossing divisor in $\PP^2$.
If
\[
 \omega\in H^0\!\left(\PP^2,
 E^{\mathrm{GG}}_{k,m}T_{\PP^2}^*(\log C)
 \otimes\OO_{\PP^2}(-t)\right),
 \qquad k,m,t\geqslant1,
\]
and $f\colon\CC\to\PP^2$ satisfies $f^*\omega\not\equiv0$, then
\[
 T_f(r)\leqslant\frac mt\sum_{i=1}^qN_f^{[1]}(r,C_i)
 +O\!\left(\log T_f(r)+\log r\right)\ \|.
\]
\end{thm}

\begin{thm}[{\cite[Theorem~2.7]{HHMX2026}}]
\label{thm:McQuillan-log-quantitative}
Let $C=\sum_{i=1}^qC_i$ be a simple normal crossing divisor in $\PP^2$,
where $C_i$ is smooth of degree $d_i$ and $d=\sum_i d_i\geqslant4$.
For an algebraically nondegenerate entire curve $f\colon\CC\to\PP^2$
whose logarithmic first lift $f_{[1]}$ is algebraically degenerate, one has
\[
 (d-3)T_f(r)\leqslant N_f^{[1]}(r,C)+o\!\left(T_f(r)\right)\ \|.
\]
\end{thm}

The following estimate treats algebraically degenerate entire curves
for a very general boundary.

\begin{lem}[{\cite[Lemma~2.8]{HHMX2026}}]
\label{lem:log-algebraically-degenerate}
For a very general smooth plane curve $C$ of degree $d\geqslant5$, every
nonconstant algebraically degenerate entire curve $f\colon\CC\to\PP^2$
with $f(\CC)\not\subset C$ satisfies
\[
 (d-4)T_f(r)\leqslant N_f^{[1]}(r,C)+o\!\left(T_f(r)\right)\ \|.
\]
\end{lem}

For the proofs and original sources of these statements, we refer
to~\cite[Theorems~2.6--2.7 and Lemma~2.8]{HHMX2026}.
Lemma~\ref{lem:log-algebraically-degenerate} gives the algebraically
degenerate case of Corollary~\ref{cor:log-11-very-general}.  It is also
used below to exclude rational curves with at most two points over the
boundary for very general parameters.  The logarithmic $1$-jet locus and
the Xie--Zhao boundedness argument extend this exclusion to a Zariski-open
set.  The quantitative extension in the corollary remains stated for
very general boundaries.

For the restricted zero-locus argument, let $\omega_1$ be an irreducible
negatively twisted invariant logarithmic $2$-jet differential of weight
$m$ and twist $-t$, and let $Z$ be the reduced residual divisor of its
tautological section after removing $\Gamma_2$.  For the selected family
in Proposition~\ref{prop:o3-log-relative-family}, this is the divisor
$Z_a$ in~\eqref{eq:o3-log-relative-divisor}, with the parameter suppressed.
If $Z$ dominates $X_1$, its divisor class has the form
\[
 Z\sim b_1u_1+b_2u_2-t\pi_{2,0}^*h,
 \qquad b_1\geqslant2b_2>0,\qquad b_1+b_2=m,
\]
where $u_1=\pi_{2,1}^*c_1(\OO_{X_1}(1))$,
$u_2=c_1(\OO_{X_2}(1))$, and $h=c_1(\OO_{\PP^2}(1))$;
see~\cite[Section~6.5]{HHMX2026}.

\begin{prop}[{\cite[Proposition~6.5]{HHMX2026}}]
\label{prop:log-restricted-bigness}
For a smooth degree-$d$ boundary $C$ with $d\geqslant11$ and $Z$ as above, put
\[
 C(d)=\frac{4d^2-51d+90}{12(d-3)},\qquad \rho=\frac tm,
\]
and
\begin{equation}\label{eq:log-tau-one}
 \tau_1(\rho)
 =2(d-3)-\frac32\rho
 -\frac{\sqrt3}{6}
 \sqrt{27\rho^2-24(d-3)\rho+4(8d^2-21d+18)}.
\end{equation}
If $0<\rho<C(d)$ and $\tau\in\QQ$ satisfies
$0\leqslant\tau<\tau_1(\rho)$, then the restricted $\QQ$-line bundle
\[
 \left.
 \left(\OO_{X_2}(2,1)\otimes\pi_{2,0}^*\OO_{\PP^2}(-\tau)\right)
 \right|_Z
\]
is big.
\end{prop}

Its proof is given in~\cite[Proposition~6.5]{HHMX2026}; the application
of the same criterion at $d=11$ is explicitly recorded
in~\cite[Remark~7.2]{HHMX2026}.
For a section defined only on $Z$, we use the canonical evaluation of
that paper, rather than choose an ambient extension.

\begin{lem}[{\cite[Lemma~6.6]{HHMX2026}}]
\label{lem:log-restricted-evaluation}
Let $Z\subset X_2$ be an effective Cartier divisor and
\[
 \sigma\in H^0\!\left(Z,
 \left.\left(\OO_{X_2}(2p,p)
 \otimes\pi_{2,0}^*\OO_{\PP^2}(-\widetilde t)\right)\right|_Z\right),
 \qquad p,\widetilde t\in\ZZ_{>0}.
\]
Let $g\colon\CC\to\PP^2$ be nonconstant, with
$g(\CC)\not\subset C$ and extended logarithmic lift
$g_{[2]}(\CC)\subset Z$.  The tautological derivative
$Dg_{[1]}\in H^0_{\mathrm{mer}}(\CC,
K_{\CC}\otimes g_{[2]}^*\OO_{X_2}(-1))$ and the inclusion
\[
 \iota_p\colon
 \OO_{X_2}(2p,p)=\OO_{X_2}(3p)(-2p\Gamma_2)
 \hookrightarrow\OO_{X_2}(3p)
\]
define the intrinsic meromorphic section
\begin{equation}\label{eq:log-canonical-evaluation}
 s_{\sigma,g}
 =\left\langle g_{[2]}^*(\iota_p\sigma),
 (Dg_{[1]})^{\otimes3p}\right\rangle
 \in H^0_{\mathrm{mer}}\!\left(\CC,
 K_{\CC}^{\otimes3p}\otimes g^*\OO_{\PP^2}(-\widetilde t)\right).
\end{equation}
If it is nonzero, then
\[
 \operatorname{div}_{\infty}(s_{\sigma,g})\leqslant3p(g^*C)_{\mathrm{red}},
 \qquad
 \int_0^{2\pi}\log^+\|s_{\sigma,g}(re^{i\theta})\|
 \frac{d\theta}{2\pi}
 =O\!\left(\log T_g(r)+\log r\right)\ \|.
\]
\end{lem}

\begin{thm}[{\cite[Theorem~6.7]{HHMX2026}}]
\label{thm:log-restricted-smt}
Under the hypotheses of Lemma~\ref{lem:log-restricted-evaluation}, let
\[
 \sigma\in H^0\!\left(Z,
 \left.\left(\OO_{X_2}(2,1)^{\otimes\widetilde m}
 \otimes\pi_{2,0}^*\OO_{\PP^2}(-\widetilde t)\right)\right|_Z\right),
 \qquad \widetilde m,\widetilde t\in\ZZ_{>0}.
\]
If $s_{\sigma,g}\not\equiv0$, with $p=\widetilde m$, then
\[
 T_g(r)\leqslant\frac{3\widetilde m}{\widetilde t}
 N_g^{[1]}(r,C)+o\!\left(T_g(r)\right)\ \|.
\]
\end{thm}

Both proofs are those of~\cite[Lemma~6.6 and Theorem~6.7]{HHMX2026}.
In particular, choosing $\widetilde m=\ell$ and
$\widetilde t=\ell\tau$ for a nonzero section supplied by
Proposition~\ref{prop:log-restricted-bigness} gives coefficient $3/\tau$.

\begin{proof}[Proof of Theorem~\ref{thm:log-11-endpoint}]
Use Proposition~\ref{prop:o3-log-relative-family} to extend the
degree-$11$ factor from the logarithmic Riemann--Roch and maximal-slope
argument of~\cite[Section~6.5, equations~(6.17)--(6.19)]{HHMX2026}
to a family over $U^{m,t}$. Its weight $m$ and twist $-t$ satisfy
\[
 \frac tm>\frac1{30},\qquad
 \rho_{\mathrm{RR}}=\frac{64-\sqrt{4018}}{18}>\frac1{30}.
\]
Here $\rho_{\mathrm{RR}}$ is the limiting existence ratio before taking
the maximal-slope factor.
For $t\geqslant4$, restrict to the neighbourhood $U'$ supplied by
Theorem~\ref{thm:apply O(3) fields log} at a point of
$U^{m,t}\cap S^{\rm red}_{\log}$.
Intersect the resulting open set with the parameter locus of
Lemma~\ref{lem:finite-kvl-main}. For a general boundary $C$ in this
family, write $\omega_1$ for the corresponding differential and $Z=Z_a$
for its residual divisor in~\eqref{eq:o3-log-relative-divisor}.

Let $f$ be algebraically nondegenerate.  If $f^*\omega_1\not\equiv0$,
Theorem~\ref{thm:log-jet-smt} gives coefficient $m/t<30$.
Otherwise $f_{[2]}(\CC)\subset Z$: on the dense open set where
$f$ avoids $C$ and $f'\ne0$, the lift lies in the zero divisor of
$\widetilde\omega_1$ and avoids $\Gamma_2$; closedness gives the
inclusion.  The condition $b_2>0$ in
Proposition~\ref{prop:o3-log-relative-family} shows that $Z$ dominates
$X_1$.

For $d=11$, Proposition~\ref{prop:log-restricted-bigness} reads
\[
 C(11)=\frac{13}{96},\qquad
 \tau_1(\rho)=16-\frac32\rho
 -\frac{\sqrt3}{6}\sqrt{27\rho^2-192\rho+3020}.
\]
If $t/m\leqslant2/15<13/96$, the decreasing function $\tau_1$ satisfies
\begin{equation}\label{eq:log11-tau-star}
 \tau_1(t/m)\geqslant\tau_*:=\tau_1(2/15)
 =\frac{237-7\sqrt{1146}}{15},\qquad
 \frac3{\tau_*}=711+21\sqrt{1146}<1422.
\end{equation}
Choose a rational $\tau$ with $3/1422<\tau<\tau_*$, and a nonzero
restricted section $\sigma$ of
$(\OO_{X_2}(2,1)\otimes\pi_{2,0}^*\OO_{\PP^2}(-\tau))^{\otimes\ell}|_Z$
for sufficiently divisible $\ell$.
If $s_{\sigma,f}\not\equiv0$,
Theorem~\ref{thm:log-restricted-smt} gives coefficient $3/\tau<1422$.
If $s_{\sigma,f}\equiv0$, the restricted zero-locus argument
of~\cite[Section~6.5, Case~2]{HHMX2026} makes the first lift algebraically
degenerate; Theorem~\ref{thm:McQuillan-log-quantitative} then gives $1/8$.

It remains to treat $t/m>2/15$.
If $t\leqslant3$, integrality gives
\[
 t=1:\ m\leqslant7,\qquad
 t=2:\ m\leqslant14,\qquad
 t=3:\ m\leqslant22.
\]
These ranges are excluded by the logarithmic clause of
Lemma~\ref{lem:finite-kvl-main}, Remark~\ref{rem:o3-log-first-order}, and
twist monotonicity. Indeed, for $m\geqslant3$ the lemma gives vanishing
for $t=1$, $2$, and $3$ on the respective ranges
$3\leqslant m\leqslant7$, $8\leqslant m\leqslant14$, and
$15\leqslant m\leqslant22$; the remaining cases follow by monotonicity.
Thus $t\geqslant4$, and the independent $2$-jet
differential $\omega_2$ has twist $-(t-3)$.
Since $t/m>2/15$, we have $m<15t/2$. For $t=4$, integrality gives
$m/(t-3)=m\leqslant29$. For $t\geqslant5$,
\[
 \frac{m}{t-3}
 <\frac{15}{2}\left(1+\frac{3}{t-3}\right)
 \leqslant\frac{75}{4}<29.
\]
Thus, if $f^*\omega_2\not\equiv0$, Theorem~\ref{thm:log-jet-smt} gives
the estimate with coefficient $m/(t-3)\leqslant29$.
If $f^*\omega_2\equiv0$, the independent equations make the first lift
algebraically degenerate, as in~\cite[Section~6.5, Case~1]{HHMX2026},
and Theorem~\ref{thm:McQuillan-log-quantitative} again gives $1/8$.
Every branch therefore gives the stated estimate with coefficient $1422$.
\end{proof}

\begin{proof}[Proof of Corollary~\ref{cor:log-11-very-general}]
Choose $C$ in the intersection of the Zariski-open set supplied by
Theorem~\ref{thm:log-11-endpoint} and the very-general set supplied by
Lemma~\ref{lem:log-algebraically-degenerate}.  This intersection is very
general.  Let $f\colon\CC\to\PP^2$ be nonconstant with
$f(\CC)\not\subset C$.  If $f$ is algebraically nondegenerate, the
conclusion follows from Theorem~\ref{thm:log-11-endpoint}.

If $f$ is algebraically degenerate, the lemma with $d=11$ gives
\[
 7T_f(r)\leqslant N_f^{[1]}(r,C)+o\bigl(T_f(r)\bigr)\ \|.
\]
Since $N_f^{[1]}(r,C)\geqslant0$ and $1/7<1422$, this implies the
stated estimate.  For the analogous treatment of the degenerate case,
see also~\cite[Remark~7.2]{HHMX2026}.
\end{proof}

\begin{proof}[Proof of Theorem~\ref{thm:log-general-upgrade}]
Take the degree-$11$ algebraic family $\{\omega_{1,a}\}_{a\in U}$ from
Proposition~\ref{prop:o3-log-relative-family}, chosen by the maximal-slope
construction in the proof of Theorem~\ref{thm:log-11-endpoint}.
Let $\mathcal X_{i,U}$ be the relative logarithmic Semple spaces, with
$\mathcal X_{0,U}=\PP^2\times U$, and let
$\Pi_{i,j}\colon\mathcal X_{i,U}\to\mathcal X_{j,U}$ be the projections
for $0\leqslant j<i\leqslant2$. By~\eqref{eq:o3-log-relative-divisor},
subtracting the fixed vertical part from the relative zero divisor
gives a relative Cartier divisor $\mathcal Z\subset\mathcal X_{2,U}$
with integral three-dimensional fibres $Z_a$. The degree $b_2>0$
on a general fibre of $\pi_{2,1}$ shows that $Z_a$ dominates $X_{a,1}$.

For a rational curve $\Gamma$ as in the statement, put
$B=\operatorname{Supp}\nu^*C_a$ and $e=\deg\Gamma$. A logarithmic
jet differential of weight $m'$ and twist $-t'<0$ evaluates along
$\nu$ as a section of
$(\Omega^1_{\PP^1}(\log B))^{\otimes m'}\otimes\OO_{\PP^1}(-t'e)$.
Its degree $m'(\#B-2)-t'e$ is negative, so the evaluation vanishes.
Applied to $\omega_{1,a}$, this puts the second logarithmic lift
$\nu_{[2]}$ in $Z_a\cup\Gamma_{a,2}$. On the dense open set away from
$B$ where $\nu'\ne0$, the lift avoids $\Gamma_{a,2}$. Thus
$\nu_{[2]}(\PP^1)\subset Z_a$ by closedness.

If $t/m\leqslant2/15$, choose a rational $\tau$ with
$3/1422<\tau<\tau_*$, where $\tau_*$ is given by~\eqref{eq:log11-tau-star}.
At a very general parameter in the cohomology-and-base-change loci
for the divisible powers, Proposition~\ref{prop:log-restricted-bigness}
gives a nonzero fibre section of
\[
 \left.
 \left(\OO_{\mathcal X_{2,U}}(2\ell,\ell)
       \otimes\Pi_{2,0}^*\OO_{\PP^2\times U}(-\ell\tau)\right)
 \right|_{\mathcal Z}
\]
for some sufficiently divisible $\ell>0$. The chosen section
extends algebraically to a section $s$ after shrinking $U$.
Apply Lemma~\ref{lem:log-restricted-evaluation} to $\nu$ on each of the
two standard affine charts of $\PP^1$, both isomorphic to $\CC$, with
$\sigma=s_a$, $p=\ell$ and $\widetilde t=\ell\tau$.
Since $\nu_{[2]}(\PP^1)\subset Z_a$, the factor
$\nu_{[2]}^*(\iota_\ell s_a)$ in \eqref{eq:log-canonical-evaluation}
is defined on all of $\PP^1$. On the overlap, where $w=1/z$, the chain
rule for $D\nu_{[1]}$ gives the $w$-coefficient of the pairing as
$(\mathrm dz/\mathrm dw)^{3\ell}$ times its $z$-coefficient.
The local evaluations thus glue to a meromorphic $3\ell$-differential with
values in $\nu^*\OO_{\PP^2}(-\ell\tau)$. On each chart the lemma bounds
its polar divisor by $3\ell B$. In particular, if $\infty\in B$, its
pole order at $w=0$ is at most $3\ell$. The resulting global section of
$(\Omega^1_{\PP^1}(\log B))^{\otimes3\ell}\otimes
\OO_{\PP^1}(-\ell\tau e)$ vanishes because this line bundle has degree
$3\ell(\#B-2)-\ell\tau e<0$. On the dense open set where the
tautological derivative is nonzero and the lift avoids $\Gamma_{a,2}$,
the pairing then forces $s_a$ to vanish along the lift.
Hence $\nu_{[2]}(\PP^1)\subset V(s_a)$.

If $t/m>2/15$, the finite-range argument in the proof of
Theorem~\ref{thm:log-11-endpoint} excludes $t\leqslant3$ after
shrinking $U$. Theorem~\ref{thm:apply O(3) fields log} then gives
$\omega_{2,a}$ of twist $-(t-3)$ on a nonempty open subset of
$U\cap S^{\rm red}_{\log}$. Put $s=\widetilde\omega_2|_{\mathcal Z}$,
which is nonzero on the fibres by that theorem. The negative-degree
argument applied to $\omega_{2,a}$ gives
\[
 \nu_{[2]}(\PP^1)\subset V(s_a)
 = Z_a\cap V(\widetilde\omega_{2,a}).
\]

In either case let $\mathcal W=V(s)\subset\mathcal Z$, with fibres
$W_a=V(s_a)$. On a fibre where $s_a$ is nonzero, its zero set in the
integral threefold $Z_a$ has dimension at most two. Since $\mathcal W$
is projective over $U$, upper semicontinuity gives
$\dim W_a\leqslant2$ for all $a$ after shrinking $U$.
Let $\mathcal S_{11}^{\log}$ be the reduced closed image of
$\mathcal W$ under the proper projection $\Pi_{2,1}$. On each
fibre its underlying set is $\pi_{2,1}(W_a)$, so it has dimension
at most two and contains $\nu_{[1]}(\PP^1)$. By generic flatness,
$\mathcal S_{11}^{\log}$ is flat and projective over a nonempty
Zariski-open subset $U_{11}\subset U$.
\end{proof}

\begin{proof}[Proof of the logarithmic assertion of Theorem~\ref{thm:main-thresholds}]
We first apply the boundedness argument of
Xie--Zhao~\cite[Sections~7.2--7.4]{XieZhao2026} to the single smooth
boundary, using Theorem~\ref{thm:log-general-upgrade}.
Shrink $U_{11}$ to an affine open subset on which the universal equation is
represented by $F_a$.  The degree-$11$ cyclic cover
$q_a\colon Y_a\to\PP^2$, given in $\PP^3$ by $w^{11}=F_a$,
forms a smooth projective family.  With $R_a=(q_a^{-1}C_a)_{\mathrm{red}}$
and $H_a=q_a^*\OO_{\PP^2}(1)$, its canonical bundle is
$K_{Y_a}=H_a^{\otimes7}$; in particular, every fibre has general type.
The logarithmic cotangent bundles satisfy
$\Omega^1_{Y_a}(\log R_a)\simeq q_a^*\Omega^1_{\PP^2}(\log C_a)$.

Pull back $\mathcal S_{11}^{\log}$ to the projectivized logarithmic
tangent bundle of this family. Applying the relatively ample twist
and Serre construction of Section~\ref{sec:compact-endpoints-proof}
gives a nonzero relative logarithmic symmetric differential
vanishing on that direction locus.  Clearing its logarithmic poles
gives an ordinary twisted symmetric differential, nonzero on every
fibre after shrinking.  This is the construction of
\cite[Proposition~7.3.1]{XieZhao2026}, applied to one boundary component.

Suppose that $\Gamma\not\subset C_a$ is rational and
$\#\operatorname{Supp}\nu^*C_a\leqslant2$.
The normalization of a component of
$Y_a\times_{\PP^2}\PP^1$ is a finite cover of $\PP^1$ branched over
at most two points.  By Riemann--Hurwitz, this normalization is rational.
Its image
$\widetilde\Gamma\subset Y_a$ has tangent directions in the pulled-back
locus, so it is contained in the zero set of the ordinary symmetric
differential or is invariant under its multi-foliation.
After shrinking $U_{11}$, Xie--Zhao~\cite[Theorem~6.1.1]{XieZhao2026},
with geometric genus zero, bounds $H_a\cdot\widetilde\Gamma$ uniformly.
The projection formula then bounds $\deg\Gamma$ uniformly.

For each degree within this bound, the morphisms $\PP^1\to\PP^2$
whose image is not contained in $C_a$ and for which the pulled-back
boundary has support of size at most two form a finite-type incidence
family. The union $\Sigma$ of their images in $U_{11}$ is
constructible, as in~\cite[Lemma~7.4.1]{XieZhao2026}. A rational
curve with at most two punctures admits a nonconstant entire
parametrization avoiding $C_a$. Lemma~\ref{lem:log-algebraically-degenerate},
with $d-4=7$, therefore places a very general parameter outside
$\Sigma$. If $\Sigma$ were dense in $U_{11}$, it would contain a
nonempty Zariski-open subset and hence meet the very general set
where the lemma applies. Thus $U_{11}\setminus\overline\Sigma$ is
a nonempty Zariski-open set on which every rational curve not
contained in the boundary has at least three punctures.

Let $U_{11}^{\mathrm{hyp}}$ be the intersection of this open set
with the one supplied by Theorem~\ref{thm:log-11-endpoint}.
The smooth parameter space is irreducible, so the intersection
is nonempty. Fix $a\in U_{11}^{\mathrm{hyp}}$.
For an entire curve in the complement,
$N_f^{[1]}(r,C_a)=0$, so that theorem excludes algebraically
nondegenerate curves.  If the image is contained in an algebraic curve
$\Gamma$, it lifts to the normalization minus the inverse image of
$C_a$.  That punctured curve is hyperbolic: its genus is at least two,
or its genus is one and B\'ezout gives a puncture, or it is rational
and the preceding argument gives at least three punctures.
Thus every entire curve in the complement is constant.  The smooth
degree-$11$ boundary has genus $45$; Green's criterion, as used
in~\cite[Section~6.4]{HHMX2026}, gives a hyperbolic embedding into
$\PP^2$ and hence Kobayashi hyperbolicity on $U_{11}^{\mathrm{hyp}}$.
\end{proof}

\appendix

\section{Preliminaries on representations of
\texorpdfstring{$\mathfrak{sl}_2(\CC)$}{sl2(C)}}
\label{sec:sl2-preliminaries}

The independence argument using pole-three slanted vector fields relies
on the polynomial representation recalled here (see~\cite{HallGTM222}).
Put $V_n=\CC[z_1,z_2]_n$, the space of homogeneous
binary forms of degree $n$, and define
\[
 (\Pi_n(U)f)(z)=f(U^{-1}z),\qquad U\in SL_2(\CC),\quad
 z=(z_1,z_2)^{\mathsf T}.
\]
Differentiation gives the Lie algebra representation
\begin{equation}
 \pi_n(B)f=\left.\frac{d}{ds}f(e^{-sB}z)\right|_{s=0},
 \qquad B\in\mathfrak{sl}_2(\CC).
 \label{eq:sl2-derived-representation}
\end{equation}
For the standard basis
\[
 \mathsf X=\begin{pmatrix}0&1\\0&0\end{pmatrix},\qquad
 \mathsf Y=\begin{pmatrix}0&0\\1&0\end{pmatrix},\qquad
 \mathsf H=\begin{pmatrix}1&0\\0&-1\end{pmatrix},
\]
one has $[\mathsf H,\mathsf X]=2\mathsf X$,
$[\mathsf H,\mathsf Y]=-2\mathsf Y$, $[\mathsf X,\mathsf Y]=\mathsf H$, and
\[
 \pi_n(\mathsf H)=-z_1\partial_{z_1}+z_2\partial_{z_2},\qquad
 \pi_n(\mathsf X)=-z_2\partial_{z_1},\qquad
 \pi_n(\mathsf Y)=-z_1\partial_{z_2}.
\]
Thus, if $B=\left(\begin{smallmatrix}b_{11}&b_{12}\\
b_{21}&-b_{11}\end{smallmatrix}\right)$ and $e_k=z_1^{n-k}z_2^k$, then
\begin{equation}
 \pi_n(B)e_k=-(n-2k)b_{11}e_k
             -(n-k)b_{12}e_{k+1}-kb_{21}e_{k-1},
 \quad 0\leqslant k\leqslant n,
 \label{eq:sl2-monomial-action}
\end{equation}
where terms outside $0\leqslant k\leqslant n$ are omitted.  This formula
fixes both the signs and the column convention of the representation matrix.

\begin{lem}
\label{lem:sl2-characteristic-polynomial}
For $B\in\mathfrak{sl}_2(\CC)$ and a choice
$\mu=\sqrt{-\det B}$,
\begin{equation}
 \det\bigl(\pi_n(B)-\lambda I\bigr)
 =\prod_{k=0}^{n}\bigl(-\lambda-(n-2k)\mu\bigr).
 \label{eq:sl2-characteristic-polynomial}
\end{equation}
\end{lem}

\begin{proof}
If $\det B\ne0$, choose $N\in SL_2(\CC)$ with
$NBN^{-1}=\mu\mathsf H$.  Equivariance of the derived representation gives
\[
 \pi_n(NBN^{-1})=\Pi_n(N)\pi_n(B)\Pi_n(N)^{-1}.
\]
The action of $\mu\pi_n(\mathsf H)$ on $e_k$ is multiplication by
$-(n-2k)\mu$, proving the formula.  If $\det B=0$, the matrix $B$ is
nilpotent; its induced derivation on $V_n$ is nilpotent as well.  Both sides
are then $(-\lambda)^{n+1}$.
\end{proof}

\begin{cor}
\label{cor:sl2-sharp-degree}
Let $B\in\mathfrak{sl}_2(\CC)$ have nonzero determinant. Fix a square root
$\mu=\sqrt{-\det B}$ and choose $N\in SL_2(\CC)$ such that
\[
 NBN^{-1}=\mu\mathsf H=\operatorname{diag}(\mu,-\mu),
\]
as in the proof of Lemma~\ref{lem:sl2-characteristic-polynomial}.
Every nonzero eigenvector of $\pi_n(B)$ has the form
\[
 c(Nz)_1^{\,n-j}(Nz)_2^{\,j},\qquad c\in\CC^*,\quad 0\leqslant j\leqslant n,
\]
where $(Nz)_1,(Nz)_2$ are independent linear forms.  Consequently, if
the binary form is square-free, its degree satisfies $n\leqslant2$.
This bound is sharp.
\end{cor}

\begin{proof}
After diagonalising $B$, the monomials $z_1^{n-j}z_2^j$ have pairwise distinct
eigenvalues.  Every eigenspace is therefore one-dimensional.  Transforming
back gives the displayed factorisation.  Squarefreeness is preserved
by an invertible linear change of variables and forces
\[
 n-j\leqslant1,\qquad j\leqslant1,\qquad n\leqslant2.
\]
The square-free eigenform $(Nz)_1(Nz)_2$ attains degree two.
\end{proof}

\section*{Notation for the computational appendices}

Table~\ref{tab:computational-notation} compares the notation in
Sections~\ref{sec:compact-symmetry-proof} and~\ref{sec:log-symmetry-proof}
with the conventions of Appendices~\ref{app:compact-computations}
and~\ref{app:log-computations}. The main-text $d$ is the actual degree,
so the generator uses $d-1$. The main-text $t>0$ is the twist order,
so the signed twist in the computations is $-t$.

\begin{table}[!htbp]
\centering
\caption{Notation in the main text and computational appendices.}
\label{tab:computational-notation}
\small
\renewcommand{\arraystretch}{1.1}
\begin{tabular}{@{}lll@{}}
\toprule
Quantity & Main text & Computational appendices \\
\midrule
Actual degree & $d$ & $d+1$ \\
Exponent of the twist & $-t$ & $t$ \\
Weighted jet degree & $m$ & $m$ \\
Involution sign & $\varepsilon$ & $e$ \\
Coefficient indices & $(j,k)$ & $(a,q)$ \\
Compact monomial exponents & $(\alpha,\beta,\gamma)$ & $(i,j,k)$ \\
Compact jet coordinates & $(x',y',x'',y'')$ & $(x_1,y_1,x_2,y_2)$ \\
Compact Wronskian & $\Delta_{yx}'''$ & $\Delta_{yx}$ \\
Compact coefficient & $K_{j,k}$ & $F_{a,q}$ \\
Logarithmic coefficient & $K_{j,k}$ & coefficient indexed by $(a,q)$ \\
Section represented by compact $J_{yx}$ & $\Omega_y$ & $s$ \\
\bottomrule
\end{tabular}
\end{table}

On $\{T\ne0\}$, the compact expression $J_{yx}$ in
\eqref{eq:compact-jet-ansatz} represents the twisted section $s$ in the
trivialization determined by $T$. With $t$ denoting the positive twist
order of the main text, the local expression of $\Omega_y=Y^t s$ in
\eqref{eq:local-J-cpt} is $y^t$ times the appendix expression. Its
coefficients have the form $F_{j,k}=y^tK_{j,k}$, where the main-text
polynomial $K_{j,k}$ is the appendix polynomial $F_{a,q}$ for
$(j,k)=(a,q)$.

In the logarithmic case, the main text uses $\Omega_x=X^t s$ with the
positive twist order $t$ of Subsection~\ref{sec:log-degree}. The
coefficients in~\eqref{eq:local-J-log} are $F_{j,k}=x^tK_{j,k}$, and
the main-text polynomial $K_{j,k}$ is the appendix coefficient indexed
by $(a,q)=(j,k)$. The $t$ in the computational degree
bound~\eqref{eq:log-degree-bound} is the signed twist of the appendix.


\section{Computations in the compact case}
\label{app:compact-computations}

This section describes the construction and modular rank analysis of the
linear systems arising in the compact case. The completed sector tests
comprise $514$ parameter combinations: $180$ for the Fermat-type examples,
$270$ for the general-type degree-$15$ surface, and $64$ for the
\texttt{compat\_normal\_type} degree-$16$ surface. The parameter $d$ is one
less than the degree of the defining hypersurface, $m$ is the weighted
jet degree, and $t$ is the twist. All three compact groups consider both
signs $e\in\{+1,-1\}$ and the root-weight sectors
$r\in\{0,\ldots,d\}$ within each $(d,m,t)$ group.
The new $d=15$ group is specified in
Section~\ref{app:compact-compat-d15}. All reported ranks and column
nullities are taken over $\mathbb{F}_{32749}$.

\subsection{Computational methods}
\label{app:compact-methods}

\paragraph{Construction of the coefficient matrices.}
On an affine chart, write the defining equation as
\begin{equation}
 R(x,y,z)=x^{d+1}+y^{d+1}+z^{d+1}+U(x,y,z)=0.
 \label{eq:compact-affine-surface}
\end{equation}
The coefficient ansatz for the weighted degree-$m$ second-order jet
expression is organized as
\begin{equation}
 J_{yx}=
 \sum_{q=0}^{\lfloor m/3\rfloor}
 \sum_{a=0}^{m-3q}
 \frac{F_{a,q}(x,y,z)\,
 x_1^{m-3q-a}y_1^a\Delta_{yx}^{q}}
 {R_z^{m-2q}},
 \qquad \Delta_{yx}=y_1x_2-y_2x_1.
 \label{eq:compact-jet-ansatz}
\end{equation}
Here $x_1,y_1$ and $x_2,y_2$ denote first- and second-order jet
coordinates, respectively, and $R_z=\partial R/\partial z$.
The polynomials $F_{a,q}$ are linear combinations of monomials
$x^iy^jz^k$ with unknown coefficients, subject to
\begin{equation}
 D_q=(m-2q)d-m+t,\qquad
 0\le i\le\min(d,D_q),\qquad
 j,k\ge0,\qquad i+j+k\le D_q.
 \label{eq:compact-degree-bounds}
\end{equation}
A negative degree bound gives the zero polynomial. The weight-$r$
restriction, when enabled, is imposed before expansion by retaining only
terms satisfying
\begin{equation}
 i+m-a-2q\equiv r\pmod{d+1}.
 \label{eq:compact-weight}
\end{equation}
The sign $e$ specifies the corresponding compatibility condition under
the exchange of $y$ and $z$.

Implicit differentiation of $R=0$ supplies the first- and second-order
coordinate substitutions. After substitution into the chart-compatibility
identities, common denominators are cleared. Powers of $x$ are reduced
using
\begin{equation}
 x^{d+1}=-y^{d+1}-z^{d+1}-U(x,y,z).
\end{equation}
Equating the remaining polynomial coefficients gives homogeneous linear
relations in the unknown coefficients of the $F_{a,q}$.
The additional conditions at infinity are obtained by homogenization,
the required coordinate transformation, and coefficient extraction at
the relevant powers of the homogenizing variable. These conditions are
included in the same system, rather than checked after the rank computation.
For each sector, the resulting system has the form
\begin{equation}
 A_{d,m,t,r,e}c=0,
 \qquad A_{d,m,t,r,e}\in
 \mathbb{F}_{32749}^{\,n_{\mathrm{row}}\times n_{\mathrm{col}}}.
 \label{eq:compact-linear-system}
\end{equation}

\paragraph{Matrix generation with multiple MPI processes.}
The symbolic generation is implemented in C++ using GiNaC/CLN and
distributed across independent MPI processes. The ordinary compatibility
conditions are partitioned into coefficient blocks indexed by
$(i,j,q)$ with $i+j+3q=m$. A separate block imposes the conditions at
infinity. Thus the version including these compactness conditions has
\begin{equation}
 B(m)=\sum_{q=0}^{\lfloor m/3\rfloor}(m-3q+1)+1
 \label{eq:compact-block-count}
\end{equation}
blocks. Tasks are assigned to MPI processes using the implemented
partitioning rule, with a weighted partition available for balancing
unequal symbolic workloads. Each process performs the substitutions,
polynomial reductions and coefficient extraction for its assigned blocks.
The earlier sector-generation batches used $12$ MPI processes for the
MatrixMarket runs and $5$ for the later HDF5 runs. The new $d=15$
compatibility tests use $4$ MPI processes per matrix.

For the larger systems, the column-oriented output mode fixes a common
column numbering before block generation and writes each block directly
to HDF5. This avoids gathering the entire expanded symbolic system into
one process. The HDF5 records retain the coefficient values, sparse
indices and global column information. In particular,
$B(9)=23$, $B(10)=27$, and $B(11)=31$; the number of blocks is independent
of the number of MPI processes. All blocks belonging to one sector form
one coefficient matrix and are passed together to one rank computation.

\paragraph{Wiedemann rank computation on GPUs.}
We formulate the rank stage uniformly using a GPU-accelerated Wiedemann
workflow. The method replaces a sparse factorization by a projected
Krylov sequence and a scalar linear-recurrence computation
\cite{Wiedemann1986,KaltofenSaunders1991}.
For $A=A_{d,m,t,r,e}$, introduce diagonal preconditioners $D_r,D_c$
with nonzero entries and projection vectors $u,v$, and apply the operator
\begin{equation}
 \mathcal{B}=D_cA^{\mathsf T}D_rA,
 \qquad
 \mathcal{B}w=D_c\bigl(A^{\mathsf T}(D_r(Aw))\bigr).
 \label{eq:compact-wiedemann-operator}
\end{equation}
The product $A^{\mathsf T}D_rA$ is never assembled as a dense matrix.
The implementation forms the shifted scalar sequence
\begin{equation}
 s_i=u^{\mathsf T}\mathcal{B}^{i+1}v,
 \qquad i=0,\ldots,2\ell-1,
 \qquad \ell=\min(n_{\mathrm{row}},n_{\mathrm{col}}).
 \label{eq:compact-krylov-sequence}
\end{equation}
Berlekamp--Massey reconstruction determines its shortest linear
recurrence. For suitable preconditioners and projections, the recurrence
order equals $\operatorname{rank}(A)$; the implementation uses this order
as its rank output. The preconditioners and projections are chosen by
the ranker's pseudorandom generator. The method therefore carries the
usual qualification of a randomized rank computation. Exact finite-field
arithmetic avoids a floating-point rank threshold, but does not by itself
provide an independent rank certificate.

The expensive operations in each Krylov step are the sparse products
$Aw$ and $A^{\mathsf T}w$. The HIP/Kokkos implementation executes these
products, diagonal scalings and vector reductions on GPU/DCU devices.
The nonzero-balanced BiCSR representation supports products with both
$A$ and $A^{\mathsf T}$. Matrix data and working vectors remain on the
device during the iteration, while Berlekamp--Massey reconstruction is
performed on the host. GPU threads parallelize the work within each
matrix--vector product; the successive Krylov steps retain their
sequential dependence. MPI can distribute a sector across several
devices, and independent $(r,e)$ sectors can also run concurrently,
including configurations assigning one MPI process and one DCU to each
sector job. The recurrence reconstruction is part of the solve stage, so
completion of sequence generation alone is not completion of the rank
computation.

\paragraph{Validation and statistics.}
Input dimensions are checked against MatrixMarket headers or the global
dimensions reported by the HDF5 reader. Block counts, coefficient-array
lengths and sparse-entry counts are checked before a sector is accepted.
The final rank record must be unique and consistent with its input and
completion record. For every sector we verify
\begin{equation}
 0\le\rho\le\min(n_{\mathrm{row}},n_{\mathrm{col}}),\qquad
 \kappa=n_{\mathrm{col}}-\rho,
 \qquad \rho=\operatorname{rank}_{\mathbb{F}_{32749}}(A).
 \label{eq:compact-column-nullity}
\end{equation}
The same input is counted only once, including when completed results
are retained between batches.

For a group of sectors, the size tables report the number $K$ of matrices
and the sums
\begin{equation}
 (N_{\Sigma},C_{\Sigma},Z_{\Sigma},R_{\Sigma},\kappa_{\Sigma})
 =\sum_{r,e}
 (n_{\mathrm{row}},n_{\mathrm{col}},\operatorname{nnz}(A),\rho,\kappa).
 \label{eq:compact-group-statistics}
\end{equation}
Here $\operatorname{nnz}(A)$ is the sparse-entry count reported by the
matrix input or assembled reader; global counts are not multiplied by
the number of MPI processes. The sums describe the individual sector
matrices, not a rank computation on their horizontal or vertical
concatenation.

The timing tables give logged matrix-generation time $T_{\rm gen}$,
loading time $T_{\rm load}$, solve time $T_{\rm solve}$ and total
rank-stage time $T_{\rm rank}$, all in seconds. Each table entry is the
sum of the corresponding per-sector wall times. Benchmark timing
records are used where available. For MPI benchmark records, the elapsed
time is the maximum over processes, not a sum of process CPU times.
Otherwise, loading and solving are
measured from their explicit start and finish log timestamps, and the
rank-stage interval runs from the start of reading to the final result.
Generation time is the interval between the earliest and latest
timestamped messages in the successful sector-generation log.
Consequently $T_{\rm rank}$ includes loading, solving and small
bookkeeping overheads, but excludes matrix generation and queueing.
The reported times are measurements retained from the original execution
logs, not estimates for a new solver configuration. They are not a
controlled comparison between batches. Cumulative sector wall times
must not be interpreted as the elapsed time of a pipeline whose sectors
run concurrently. No generation timing was recorded for the four
trivial empty systems at $d=1$; these entries are marked n/a.
Execution times are displayed to two decimal places; values below
$0.01$ in the stated unit are shown as $0$. Group totals are computed
from the unrounded records.

\subsection{Computational results}
\label{app:compact-results}

\subsubsection{Fermat type}

The Fermat-type tests use the degree-$(d+1)$ hypersurfaces
\begin{equation}
 X^{d+1}+Y^{d+1}+Z^{d+1}+T^{d+1}=0,
 \qquad d=1,\ldots,12,
 \label{eq:compact-fermat-surface}
\end{equation}
with $m=3$ and $t=-1$. The implementation identifier for this family
is \texttt{fermat\_type\_roots}. For each $d$, all $2(d+1)$ sectors are
included. Table~\ref{tab:compact-fermat-sizes} gives every tested degree,
including the four empty $0\times0$ systems for $d=1$.

\begin{table}[htbp]
\centering
\caption{Fermat-type results for $m=3$, $t=-1$, and all $e=\pm1$, $0\le r\le d$. The expected column nullity is zero for $d\le7$ and $\binom{d-5}{3}$ for $d\ge8$.}
\label{tab:compact-fermat-sizes}
\small
\begin{tabular}{rrrrrrrr}
\toprule
$d$ & $K$ & $N_{\Sigma}$ & $C_{\Sigma}$ & $Z_{\Sigma}$ & $R_{\Sigma}$ & $\kappa_{\Sigma}$ & Expected \\
\midrule
$1$ & $4$ & $0$ & $0$ & $0$ & $0$ & $0$ & $0$ \\
$2$ & $6$ & $435$ & $80$ & $705$ & $80$ & $0$ & $0$ \\
$3$ & $8$ & $1\,819$ & $416$ & $3\,698$ & $416$ & $0$ & $0$ \\
$4$ & $10$ & $4\,482$ & $1\,162$ & $10\,336$ & $1\,162$ & $0$ & $0$ \\
$5$ & $12$ & $8\,844$ & $2\,472$ & $22\,117$ & $2\,472$ & $0$ & $0$ \\
$6$ & $14$ & $15\,328$ & $4\,500$ & $40\,535$ & $4\,500$ & $0$ & $0$ \\
$7$ & $16$ & $24\,354$ & $7\,400$ & $67\,108$ & $7\,400$ & $0$ & $0$ \\
$8$ & $18$ & $36\,338$ & $11\,326$ & $103\,330$ & $11\,325$ & $1$ & $1$ \\
$9$ & $20$ & $51\,700$ & $16\,432$ & $150\,719$ & $16\,428$ & $4$ & $4$ \\
$10$ & $22$ & $70\,856$ & $22\,872$ & $210\,769$ & $22\,862$ & $10$ & $10$ \\
$11$ & $24$ & $94\,226$ & $30\,800$ & $284\,998$ & $30\,780$ & $20$ & $20$ \\
$12$ & $26$ & $122\,226$ & $40\,370$ & $374\,900$ & $40\,335$ & $35$ & $35$ \\
\bottomrule
\end{tabular}
\end{table}

The observed total column nullities are
\begin{equation}
 \kappa_{\Sigma}=
 \begin{cases}
  0,&1\le d\le7,\\
  \displaystyle\binom{d-5}{3},&8\le d\le12.
 \end{cases}
 \label{eq:compact-fermat-results}
\end{equation}
Thus the nonzero totals are $1,4,10,20,35$. The complete support of the
nonzero sector nullities is given in
Table~\ref{tab:compact-fermat-nonzero}; every omitted sector has zero
column nullity. For example, at $d=8$ only $(r,e)=(2,-1)$ is deficient:
its matrix has $2\,124$ rows, $732$ columns and rank $731$.
All $180$ Fermat-type cases are accounted for, with
$C_{\Sigma}=137\,830$, $R_{\Sigma}=137\,760$, and total column nullity
$70$ over all tested degrees.

\begin{table}[htbp]
\centering
\caption{All nonzero Fermat-type column nullities. An entry $r:k$ means $\kappa_{d,3,-1,r,e}^{(32749)}=k$. Every sector not listed has zero column nullity.}
\label{tab:compact-fermat-nonzero}
\small
\begin{tabular}{rll}
\toprule
$d$ & $e=+1$ & $e=-1$ \\
\midrule
$8$ & $\varnothing$ & $2:1$ \\
$9$ & $2:1$ & $2:2,\;3:1$ \\
$10$ & $2:2,\;3:1$ & $2:4,\;3:2,\;4:1$ \\
$11$ & $2:4,\;3:2,\;4:1$ & $2:6,\;3:4,\;4:2,\;5:1$ \\
$12$ & $2:6,\;3:4,\;4:2,\;5:1$ & $2:9,\;3:6,\;4:4,\;5:2,\;6:1$ \\
\bottomrule
\end{tabular}
\end{table}

The generation and rank-stage times for all Fermat-type groups are
listed in Table~\ref{tab:compact-fermat-times}.

\begin{table}[htbp]
\centering
\caption{Fermat-type logged execution times in seconds. Each entry is the sum over the sectors in the corresponding row of Table~\ref{tab:compact-fermat-sizes}.}
\label{tab:compact-fermat-times}
\small
\begin{tabular}{rrrrr}
\toprule
$d$ & $T_{\rm gen}$ & $T_{\rm load}$ & $T_{\rm solve}$ & $T_{\rm rank}$ \\
\midrule
$1$ & $\text{n/a}$ & $0$ & $0$ & $0$ \\
$2$ & $1.45$ & $0$ & $0$ & $0$ \\
$3$ & $1.44$ & $0$ & $0$ & $0$ \\
$4$ & $2.93$ & $0$ & $0$ & $0.01$ \\
$5$ & $3.62$ & $0$ & $0.01$ & $0.02$ \\
$6$ & $4.51$ & $0.01$ & $0.02$ & $0.03$ \\
$7$ & $5.98$ & $0.02$ & $0.02$ & $0.05$ \\
$8$ & $5.77$ & $0.03$ & $0.03$ & $0.06$ \\
$9$ & $7.41$ & $0.04$ & $0.05$ & $0.09$ \\
$10$ & $10.04$ & $0.06$ & $0.06$ & $0.12$ \\
$11$ & $11.80$ & $0.07$ & $0.08$ & $0.15$ \\
$12$ & $12.49$ & $0.10$ & $0.10$ & $0.20$ \\
\bottomrule
\end{tabular}
\end{table}

Table~\ref{tab:compact-fermat-details} gives the dimensions, sparse-entry
counts, modular ranks, column nullities and logged times of all $180$
individual Fermat-type matrices, with $r$ and $e$ specified in every row.

\begingroup
\footnotesize
\setlength{\tabcolsep}{2.8pt}
\setlength{\LTcapwidth}{\textwidth}
\renewcommand{\arraystretch}{1.05}
\begin{longtable}{@{}rrrrrrrrrrrr@{}}
\caption{Individual Fermat-type matrices for $m=3$, $t=-1$,
$d=1,\ldots,12$, $0\le r\le d$ and $e=\pm1$. Each row describes one
matrix; $\rho$ is its rank over $\mathbb{F}_{32749}$ and
$\kappa=n_{\rm col}-\rho$. All four execution times are in seconds,
with two decimal places; values below $0.01$ seconds are shown as $0$. The four empty matrices at $d=1$ have no
recorded generation time, indicated by n/a.}
\label{tab:compact-fermat-details}\\
\toprule
$d$ & $r$ & $e$ & $n_{\rm row}$ & $n_{\rm col}$ & $\operatorname{nnz}(A)$ & $\rho$ & $\kappa$ & $T_{\rm gen}$ & $T_{\rm load}$ & $T_{\rm solve}$ & $T_{\rm rank}$ \\
\midrule
\endfirsthead
\caption[]{Individual Fermat-type matrices for $m=3$ and $t=-1$
(continued). All execution times are in seconds.}\\
\toprule
$d$ & $r$ & $e$ & $n_{\rm row}$ & $n_{\rm col}$ & $\operatorname{nnz}(A)$ & $\rho$ & $\kappa$ & $T_{\rm gen}$ & $T_{\rm load}$ & $T_{\rm solve}$ & $T_{\rm rank}$ \\
\midrule
\endhead
\midrule
\multicolumn{12}{r}{\footnotesize Continued on the next page.}\\
\endfoot
\bottomrule
\endlastfoot
$1$ & $0$ & $+1$ & $0$ & $0$ & $0$ & $0$ & $0$ & $\text{n/a}$ & $0$ & $0$ & $0$ \\
$1$ & $0$ & $-1$ & $0$ & $0$ & $0$ & $0$ & $0$ & $\text{n/a}$ & $0$ & $0$ & $0$ \\
$1$ & $1$ & $+1$ & $0$ & $0$ & $0$ & $0$ & $0$ & $\text{n/a}$ & $0$ & $0$ & $0$ \\
$1$ & $1$ & $-1$ & $0$ & $0$ & $0$ & $0$ & $0$ & $\text{n/a}$ & $0$ & $0$ & $0$ \\
\addlinespace[3pt]
$2$ & $0$ & $+1$ & $90$ & $16$ & $152$ & $16$ & $0$ & $0.02$ & $0$ & $0$ & $0$ \\
$2$ & $0$ & $-1$ & $88$ & $16$ & $150$ & $16$ & $0$ & $0.47$ & $0$ & $0$ & $0$ \\
$2$ & $1$ & $+1$ & $73$ & $13$ & $118$ & $13$ & $0$ & $0.02$ & $0$ & $0$ & $0$ \\
$2$ & $1$ & $-1$ & $73$ & $13$ & $118$ & $13$ & $0$ & $0.05$ & $0$ & $0$ & $0$ \\
$2$ & $2$ & $+1$ & $56$ & $11$ & $84$ & $11$ & $0$ & $0.39$ & $0$ & $0$ & $0$ \\
$2$ & $2$ & $-1$ & $55$ & $11$ & $83$ & $11$ & $0$ & $0.49$ & $0$ & $0$ & $0$ \\
\addlinespace[3pt]
$3$ & $0$ & $+1$ & $244$ & $52$ & $511$ & $52$ & $0$ & $0.43$ & $0$ & $0$ & $0$ \\
$3$ & $0$ & $-1$ & $241$ & $52$ & $508$ & $52$ & $0$ & $0.17$ & $0$ & $0$ & $0$ \\
$3$ & $1$ & $+1$ & $231$ & $52$ & $477$ & $52$ & $0$ & $0.04$ & $0$ & $0$ & $0$ \\
$3$ & $1$ & $-1$ & $229$ & $52$ & $476$ & $52$ & $0$ & $0.14$ & $0$ & $0$ & $0$ \\
$3$ & $2$ & $+1$ & $223$ & $52$ & $445$ & $52$ & $0$ & $0.08$ & $0$ & $0$ & $0$ \\
$3$ & $2$ & $-1$ & $222$ & $52$ & $444$ & $52$ & $0$ & $0.14$ & $0$ & $0$ & $0$ \\
$3$ & $3$ & $+1$ & $215$ & $52$ & $419$ & $52$ & $0$ & $0.03$ & $0$ & $0$ & $0$ \\
$3$ & $3$ & $-1$ & $214$ & $52$ & $418$ & $52$ & $0$ & $0.41$ & $0$ & $0$ & $0$ \\
\addlinespace[3pt]
$4$ & $0$ & $+1$ & $449$ & $109$ & $1\,049$ & $109$ & $0$ & $0.52$ & $0$ & $0$ & $0$ \\
$4$ & $0$ & $-1$ & $444$ & $109$ & $1\,045$ & $109$ & $0$ & $0.46$ & $0$ & $0$ & $0$ \\
$4$ & $1$ & $+1$ & $447$ & $118$ & $1\,053$ & $118$ & $0$ & $0.03$ & $0$ & $0$ & $0$ \\
$4$ & $1$ & $-1$ & $443$ & $118$ & $1\,051$ & $118$ & $0$ & $0.44$ & $0$ & $0$ & $0$ \\
$4$ & $2$ & $+1$ & $460$ & $124$ & $1\,067$ & $124$ & $0$ & $0.04$ & $0$ & $0$ & $0$ \\
$4$ & $2$ & $-1$ & $457$ & $124$ & $1\,064$ & $124$ & $0$ & $0.44$ & $0$ & $0$ & $0$ \\
$4$ & $3$ & $+1$ & $486$ & $130$ & $1\,142$ & $130$ & $0$ & $0.03$ & $0$ & $0$ & $0$ \\
$4$ & $3$ & $-1$ & $484$ & $130$ & $1\,141$ & $130$ & $0$ & $0.05$ & $0$ & $0$ & $0$ \\
$4$ & $4$ & $+1$ & $407$ & $100$ & $863$ & $100$ & $0$ & $0.44$ & $0$ & $0$ & $0$ \\
$4$ & $4$ & $-1$ & $405$ & $100$ & $861$ & $100$ & $0$ & $0.47$ & $0$ & $0$ & $0$ \\
\addlinespace[3pt]
$5$ & $0$ & $+1$ & $709$ & $187$ & $1\,770$ & $187$ & $0$ & $0.41$ & $0$ & $0$ & $0$ \\
$5$ & $0$ & $-1$ & $703$ & $187$ & $1\,765$ & $187$ & $0$ & $0.46$ & $0$ & $0$ & $0$ \\
$5$ & $1$ & $+1$ & $727$ & $211$ & $1\,844$ & $211$ & $0$ & $0.24$ & $0$ & $0$ & $0$ \\
$5$ & $1$ & $-1$ & $720$ & $211$ & $1\,841$ & $211$ & $0$ & $0.13$ & $0$ & $0$ & $0$ \\
$5$ & $2$ & $+1$ & $767$ & $228$ & $1\,959$ & $228$ & $0$ & $0.21$ & $0$ & $0$ & $0$ \\
$5$ & $2$ & $-1$ & $762$ & $228$ & $1\,955$ & $228$ & $0$ & $0.52$ & $0$ & $0$ & $0$ \\
$5$ & $3$ & $+1$ & $834$ & $244$ & $2\,214$ & $244$ & $0$ & $0.27$ & $0$ & $0$ & $0$ \\
$5$ & $3$ & $-1$ & $830$ & $244$ & $2\,211$ & $244$ & $0$ & $0.33$ & $0$ & $0$ & $0$ \\
$5$ & $4$ & $+1$ & $746$ & $202$ & $1\,820$ & $202$ & $0$ & $0.23$ & $0$ & $0$ & $0$ \\
$5$ & $4$ & $-1$ & $743$ & $202$ & $1\,818$ & $202$ & $0$ & $0.42$ & $0$ & $0$ & $0$ \\
$5$ & $5$ & $+1$ & $653$ & $164$ & $1\,461$ & $164$ & $0$ & $0.04$ & $0$ & $0$ & $0$ \\
$5$ & $5$ & $-1$ & $650$ & $164$ & $1\,459$ & $164$ & $0$ & $0.36$ & $0$ & $0$ & $0$ \\
\addlinespace[3pt]
$6$ & $0$ & $+1$ & $1\,028$ & $286$ & $2\,668$ & $286$ & $0$ & $0.52$ & $0$ & $0$ & $0$ \\
$6$ & $0$ & $-1$ & $1\,020$ & $286$ & $2\,662$ & $286$ & $0$ & $0.08$ & $0$ & $0$ & $0$ \\
$6$ & $1$ & $+1$ & $1\,073$ & $331$ & $2\,850$ & $331$ & $0$ & $0.07$ & $0$ & $0$ & $0$ \\
$6$ & $1$ & $-1$ & $1\,064$ & $331$ & $2\,846$ & $331$ & $0$ & $0.06$ & $0$ & $0$ & $0$ \\
$6$ & $2$ & $+1$ & $1\,148$ & $364$ & $3\,119$ & $364$ & $0$ & $0.56$ & $0$ & $0$ & $0$ \\
$6$ & $2$ & $-1$ & $1\,140$ & $364$ & $3\,113$ & $364$ & $0$ & $0.46$ & $0$ & $0$ & $0$ \\
$6$ & $3$ & $+1$ & $1\,262$ & $395$ & $3\,650$ & $395$ & $0$ & $0.65$ & $0$ & $0$ & $0$ \\
$6$ & $3$ & $-1$ & $1\,256$ & $395$ & $3\,646$ & $395$ & $0$ & $0.53$ & $0$ & $0$ & $0$ \\
$6$ & $4$ & $+1$ & $1\,162$ & $340$ & $3\,126$ & $340$ & $0$ & $0.07$ & $0$ & $0$ & $0$ \\
$6$ & $4$ & $-1$ & $1\,157$ & $340$ & $3\,122$ & $340$ & $0$ & $0.07$ & $0$ & $0$ & $0$ \\
$6$ & $5$ & $+1$ & $1\,060$ & $290$ & $2\,652$ & $290$ & $0$ & $0.06$ & $0$ & $0$ & $0$ \\
$6$ & $5$ & $-1$ & $1\,056$ & $290$ & $2\,650$ & $290$ & $0$ & $0.45$ & $0$ & $0$ & $0$ \\
$6$ & $6$ & $+1$ & $953$ & $244$ & $2\,217$ & $244$ & $0$ & $0.44$ & $0$ & $0$ & $0$ \\
$6$ & $6$ & $-1$ & $949$ & $244$ & $2\,214$ & $244$ & $0$ & $0.49$ & $0$ & $0$ & $0$ \\
\addlinespace[3pt]
$7$ & $0$ & $+1$ & $1\,406$ & $406$ & $3\,749$ & $406$ & $0$ & $0.65$ & $0$ & $0$ & $0$ \\
$7$ & $0$ & $-1$ & $1\,397$ & $406$ & $3\,742$ & $406$ & $0$ & $0.08$ & $0$ & $0$ & $0$ \\
$7$ & $1$ & $+1$ & $1\,486$ & $478$ & $4\,071$ & $478$ & $0$ & $0.16$ & $0$ & $0$ & $0$ \\
$7$ & $1$ & $-1$ & $1\,475$ & $478$ & $4\,066$ & $478$ & $0$ & $0.09$ & $0$ & $0$ & $0$ \\
$7$ & $2$ & $+1$ & $1\,604$ & $532$ & $4\,549$ & $532$ & $0$ & $0.66$ & $0$ & $0$ & $0$ \\
$7$ & $2$ & $-1$ & $1\,594$ & $532$ & $4\,542$ & $532$ & $0$ & $0.30$ & $0$ & $0$ & $0$ \\
$7$ & $3$ & $+1$ & $1\,775$ & $583$ & $5\,448$ & $583$ & $0$ & $0.12$ & $0$ & $0$ & $0$ \\
$7$ & $3$ & $-1$ & $1\,766$ & $583$ & $5\,442$ & $583$ & $0$ & $0.35$ & $0$ & $0$ & $0$ \\
$7$ & $4$ & $+1$ & $1\,658$ & $515$ & $4\,796$ & $515$ & $0$ & $0.54$ & $0$ & $0$ & $0$ \\
$7$ & $4$ & $-1$ & $1\,651$ & $515$ & $4\,791$ & $515$ & $0$ & $0.49$ & $0$ & $0$ & $0$ \\
$7$ & $5$ & $+1$ & $1\,544$ & $452$ & $4\,192$ & $452$ & $0$ & $0.54$ & $0$ & $0$ & $0$ \\
$7$ & $5$ & $-1$ & $1\,538$ & $452$ & $4\,188$ & $452$ & $0$ & $0.38$ & $0$ & $0$ & $0$ \\
$7$ & $6$ & $+1$ & $1\,428$ & $394$ & $3\,642$ & $394$ & $0$ & $0.48$ & $0$ & $0$ & $0$ \\
$7$ & $6$ & $-1$ & $1\,423$ & $394$ & $3\,639$ & $394$ & $0$ & $0.09$ & $0$ & $0$ & $0$ \\
$7$ & $7$ & $+1$ & $1\,307$ & $340$ & $3\,127$ & $340$ & $0$ & $0.57$ & $0$ & $0$ & $0$ \\
$7$ & $7$ & $-1$ & $1\,302$ & $340$ & $3\,124$ & $340$ & $0$ & $0.49$ & $0$ & $0$ & $0$ \\
\addlinespace[3pt]
$8$ & $0$ & $+1$ & $1\,843$ & $547$ & $5\,007$ & $547$ & $0$ & $0.11$ & $0$ & $0$ & $0$ \\
$8$ & $0$ & $-1$ & $1\,832$ & $547$ & $4\,999$ & $547$ & $0$ & $0.09$ & $0$ & $0$ & $0$ \\
$8$ & $1$ & $+1$ & $1\,966$ & $652$ & $5\,507$ & $652$ & $0$ & $0.49$ & $0$ & $0$ & $0$ \\
$8$ & $1$ & $-1$ & $1\,953$ & $652$ & $5\,501$ & $652$ & $0$ & $0.11$ & $0$ & $0$ & $0$ \\
$8$ & $2$ & $+1$ & $2\,136$ & $732$ & $6\,247$ & $732$ & $0$ & $0.12$ & $0$ & $0$ & $0$ \\
$8$ & $2$ & $-1$ & $2\,124$ & $732$ & $6\,238$ & $731$ & $1$ & $0.22$ & $0$ & $0$ & $0$ \\
$8$ & $3$ & $+1$ & $2\,374$ & $808$ & $7\,610$ & $808$ & $0$ & $0.15$ & $0$ & $0$ & $0$ \\
$8$ & $3$ & $-1$ & $2\,363$ & $808$ & $7\,603$ & $808$ & $0$ & $0.50$ & $0$ & $0$ & $0$ \\
$8$ & $4$ & $+1$ & $2\,239$ & $727$ & $6\,828$ & $727$ & $0$ & $0.13$ & $0$ & $0$ & $0$ \\
$8$ & $4$ & $-1$ & $2\,229$ & $727$ & $6\,821$ & $727$ & $0$ & $0.62$ & $0$ & $0$ & $0$ \\
$8$ & $5$ & $+1$ & $2\,108$ & $651$ & $6\,096$ & $651$ & $0$ & $0.50$ & $0$ & $0$ & $0$ \\
$8$ & $5$ & $-1$ & $2\,100$ & $651$ & $6\,091$ & $651$ & $0$ & $0.12$ & $0$ & $0$ & $0$ \\
$8$ & $6$ & $+1$ & $1\,980$ & $580$ & $5\,416$ & $580$ & $0$ & $0.50$ & $0$ & $0$ & $0$ \\
$8$ & $6$ & $-1$ & $1\,973$ & $580$ & $5\,411$ & $580$ & $0$ & $0.14$ & $0$ & $0$ & $0$ \\
$8$ & $7$ & $+1$ & $1\,850$ & $514$ & $4\,786$ & $514$ & $0$ & $0.46$ & $0$ & $0$ & $0$ \\
$8$ & $7$ & $-1$ & $1\,844$ & $514$ & $4\,783$ & $514$ & $0$ & $0.71$ & $0$ & $0$ & $0$ \\
$8$ & $8$ & $+1$ & $1\,715$ & $452$ & $4\,195$ & $452$ & $0$ & $0.50$ & $0$ & $0$ & $0$ \\
$8$ & $8$ & $-1$ & $1\,709$ & $452$ & $4\,191$ & $452$ & $0$ & $0.31$ & $0$ & $0$ & $0$ \\
\addlinespace[3pt]
$9$ & $0$ & $+1$ & $2\,339$ & $709$ & $6\,448$ & $709$ & $0$ & $0.13$ & $0$ & $0$ & $0$ \\
$9$ & $0$ & $-1$ & $2\,327$ & $709$ & $6\,439$ & $709$ & $0$ & $0.56$ & $0$ & $0$ & $0$ \\
$9$ & $1$ & $+1$ & $2\,513$ & $853$ & $7\,158$ & $853$ & $0$ & $0.59$ & $0$ & $0$ & $0$ \\
$9$ & $1$ & $-1$ & $2\,498$ & $853$ & $7\,151$ & $853$ & $0$ & $0.13$ & $0$ & $0$ & $0$ \\
$9$ & $2$ & $+1$ & $2\,744$ & $964$ & $8\,215$ & $963$ & $1$ & $0.61$ & $0$ & $0$ & $0$ \\
$9$ & $2$ & $-1$ & $2\,730$ & $964$ & $8\,205$ & $962$ & $2$ & $0.60$ & $0$ & $0$ & $0$ \\
$9$ & $3$ & $+1$ & $3\,060$ & $1\,070$ & $10\,134$ & $1\,070$ & $0$ & $0.18$ & $0$ & $0$ & $0$ \\
$9$ & $3$ & $-1$ & $3\,047$ & $1\,070$ & $10\,125$ & $1\,069$ & $1$ & $0.30$ & $0$ & $0$ & $0$ \\
$9$ & $4$ & $+1$ & $2\,906$ & $976$ & $9\,224$ & $976$ & $0$ & $0.58$ & $0$ & $0$ & $0$ \\
$9$ & $4$ & $-1$ & $2\,894$ & $976$ & $9\,216$ & $976$ & $0$ & $0.33$ & $0$ & $0$ & $0$ \\
$9$ & $5$ & $+1$ & $2\,757$ & $887$ & $8\,362$ & $887$ & $0$ & $0.55$ & $0$ & $0$ & $0$ \\
$9$ & $5$ & $-1$ & $2\,746$ & $887$ & $8\,355$ & $887$ & $0$ & $0.15$ & $0$ & $0$ & $0$ \\
$9$ & $6$ & $+1$ & $2\,612$ & $803$ & $7\,554$ & $803$ & $0$ & $0.54$ & $0$ & $0$ & $0$ \\
$9$ & $6$ & $-1$ & $2\,603$ & $803$ & $7\,548$ & $803$ & $0$ & $0.15$ & $0$ & $0$ & $0$ \\
$9$ & $7$ & $+1$ & $2\,470$ & $724$ & $6\,794$ & $724$ & $0$ & $0.55$ & $0$ & $0$ & $0$ \\
$9$ & $7$ & $-1$ & $2\,462$ & $724$ & $6\,789$ & $724$ & $0$ & $0.51$ & $0$ & $0$ & $0$ \\
$9$ & $8$ & $+1$ & $2\,326$ & $650$ & $6\,088$ & $650$ & $0$ & $0.12$ & $0$ & $0$ & $0$ \\
$9$ & $8$ & $-1$ & $2\,319$ & $650$ & $6\,084$ & $650$ & $0$ & $0.25$ & $0$ & $0$ & $0$ \\
$9$ & $9$ & $+1$ & $2\,177$ & $580$ & $5\,417$ & $580$ & $0$ & $0.12$ & $0$ & $0$ & $0$ \\
$9$ & $9$ & $-1$ & $2\,170$ & $580$ & $5\,413$ & $580$ & $0$ & $0.45$ & $0$ & $0$ & $0$ \\
\addlinespace[3pt]
$10$ & $0$ & $+1$ & $2\,894$ & $892$ & $8\,066$ & $892$ & $0$ & $0.25$ & $0$ & $0$ & $0$ \\
$10$ & $0$ & $-1$ & $2\,880$ & $892$ & $8\,056$ & $892$ & $0$ & $0.63$ & $0$ & $0$ & $0$ \\
$10$ & $1$ & $+1$ & $3\,127$ & $1\,081$ & $9\,024$ & $1\,081$ & $0$ & $0.63$ & $0$ & $0$ & $0$ \\
$10$ & $1$ & $-1$ & $3\,110$ & $1\,081$ & $9\,016$ & $1\,081$ & $0$ & $0.37$ & $0$ & $0$ & $0$ \\
$10$ & $2$ & $+1$ & $3\,428$ & $1\,228$ & $10\,451$ & $1\,226$ & $2$ & $0.36$ & $0$ & $0$ & $0$ \\
$10$ & $2$ & $-1$ & $3\,412$ & $1\,228$ & $10\,439$ & $1\,224$ & $4$ & $0.53$ & $0$ & $0$ & $0$ \\
$10$ & $3$ & $+1$ & $3\,833$ & $1\,369$ & $13\,022$ & $1\,368$ & $1$ & $0.24$ & $0$ & $0$ & $0$ \\
$10$ & $3$ & $-1$ & $3\,818$ & $1\,369$ & $13\,012$ & $1\,367$ & $2$ & $0.74$ & $0$ & $0$ & $0$ \\
$10$ & $4$ & $+1$ & $3\,660$ & $1\,262$ & $11\,982$ & $1\,262$ & $0$ & $0.20$ & $0$ & $0$ & $0$ \\
$10$ & $4$ & $-1$ & $3\,646$ & $1\,262$ & $11\,972$ & $1\,261$ & $1$ & $0.70$ & $0$ & $0$ & $0$ \\
$10$ & $5$ & $+1$ & $3\,492$ & $1\,160$ & $10\,992$ & $1\,160$ & $0$ & $0.58$ & $0$ & $0$ & $0$ \\
$10$ & $5$ & $-1$ & $3\,479$ & $1\,160$ & $10\,984$ & $1\,160$ & $0$ & $0.65$ & $0$ & $0$ & $0$ \\
$10$ & $6$ & $+1$ & $3\,329$ & $1\,063$ & $10\,054$ & $1\,063$ & $0$ & $0.20$ & $0$ & $0$ & $0$ \\
$10$ & $6$ & $-1$ & $3\,317$ & $1\,063$ & $10\,046$ & $1\,063$ & $0$ & $0.59$ & $0$ & $0$ & $0$ \\
$10$ & $7$ & $+1$ & $3\,170$ & $971$ & $9\,166$ & $971$ & $0$ & $0.51$ & $0$ & $0$ & $0$ \\
$10$ & $7$ & $-1$ & $3\,160$ & $971$ & $9\,160$ & $971$ & $0$ & $0.45$ & $0$ & $0$ & $0$ \\
$10$ & $8$ & $+1$ & $3\,014$ & $884$ & $8\,330$ & $884$ & $0$ & $0.51$ & $0$ & $0$ & $0$ \\
$10$ & $8$ & $-1$ & $3\,005$ & $884$ & $8\,324$ & $884$ & $0$ & $0.72$ & $0$ & $0$ & $0$ \\
$10$ & $9$ & $+1$ & $2\,856$ & $802$ & $7\,544$ & $802$ & $0$ & $0.57$ & $0$ & $0$ & $0$ \\
$10$ & $9$ & $-1$ & $2\,848$ & $802$ & $7\,540$ & $802$ & $0$ & $0.14$ & $0$ & $0$ & $0$ \\
$10$ & $10$ & $+1$ & $2\,693$ & $724$ & $6\,797$ & $724$ & $0$ & $0.32$ & $0$ & $0$ & $0$ \\
$10$ & $10$ & $-1$ & $2\,685$ & $724$ & $6\,792$ & $724$ & $0$ & $0.15$ & $0$ & $0$ & $0$ \\
\addlinespace[3pt]
$11$ & $0$ & $+1$ & $3\,508$ & $1\,096$ & $9\,867$ & $1\,096$ & $0$ & $0.59$ & $0$ & $0$ & $0$ \\
$11$ & $0$ & $-1$ & $3\,493$ & $1\,096$ & $9\,856$ & $1\,096$ & $0$ & $0.75$ & $0$ & $0$ & $0$ \\
$11$ & $1$ & $+1$ & $3\,808$ & $1\,336$ & $11\,105$ & $1\,336$ & $0$ & $0.36$ & $0$ & $0$ & $0$ \\
$11$ & $1$ & $-1$ & $3\,789$ & $1\,336$ & $11\,096$ & $1\,336$ & $0$ & $0.68$ & $0$ & $0$ & $0$ \\
$11$ & $2$ & $+1$ & $4\,188$ & $1\,524$ & $12\,957$ & $1\,520$ & $4$ & $0.29$ & $0$ & $0$ & $0$ \\
$11$ & $2$ & $-1$ & $4\,170$ & $1\,524$ & $12\,944$ & $1\,518$ & $6$ & $0.60$ & $0$ & $0$ & $0$ \\
$11$ & $3$ & $+1$ & $4\,693$ & $1\,705$ & $16\,272$ & $1\,703$ & $2$ & $0.66$ & $0$ & $0$ & $0$ \\
$11$ & $3$ & $-1$ & $4\,676$ & $1\,705$ & $16\,260$ & $1\,701$ & $4$ & $0.29$ & $0$ & $0$ & $0$ \\
$11$ & $4$ & $+1$ & $4\,501$ & $1\,585$ & $15\,104$ & $1\,584$ & $1$ & $0.56$ & $0$ & $0$ & $0$ \\
$11$ & $4$ & $-1$ & $4\,485$ & $1\,585$ & $15\,093$ & $1\,583$ & $2$ & $0.31$ & $0$ & $0$ & $0$ \\
$11$ & $5$ & $+1$ & $4\,314$ & $1\,470$ & $13\,984$ & $1\,470$ & $0$ & $0.26$ & $0$ & $0$ & $0$ \\
$11$ & $5$ & $-1$ & $4\,299$ & $1\,470$ & $13\,974$ & $1\,469$ & $1$ & $0.73$ & $0$ & $0$ & $0$ \\
$11$ & $6$ & $+1$ & $4\,132$ & $1\,360$ & $12\,918$ & $1\,360$ & $0$ & $0.71$ & $0$ & $0$ & $0$ \\
$11$ & $6$ & $-1$ & $4\,118$ & $1\,360$ & $12\,909$ & $1\,360$ & $0$ & $0.29$ & $0$ & $0$ & $0$ \\
$11$ & $7$ & $+1$ & $3\,955$ & $1\,255$ & $11\,900$ & $1\,255$ & $0$ & $0.43$ & $0$ & $0$ & $0$ \\
$11$ & $7$ & $-1$ & $3\,942$ & $1\,255$ & $11\,892$ & $1\,255$ & $0$ & $0.30$ & $0$ & $0$ & $0$ \\
$11$ & $8$ & $+1$ & $3\,782$ & $1\,155$ & $10\,936$ & $1\,155$ & $0$ & $0.73$ & $0$ & $0$ & $0$ \\
$11$ & $8$ & $-1$ & $3\,771$ & $1\,155$ & $10\,929$ & $1\,155$ & $0$ & $0.36$ & $0$ & $0$ & $0$ \\
$11$ & $9$ & $+1$ & $3\,612$ & $1\,060$ & $10\,020$ & $1\,060$ & $0$ & $0.64$ & $0$ & $0$ & $0$ \\
$11$ & $9$ & $-1$ & $3\,602$ & $1\,060$ & $10\,014$ & $1\,060$ & $0$ & $0.59$ & $0$ & $0$ & $0$ \\
$11$ & $10$ & $+1$ & $3\,440$ & $970$ & $9\,158$ & $970$ & $0$ & $0.17$ & $0$ & $0$ & $0$ \\
$11$ & $10$ & $-1$ & $3\,431$ & $970$ & $9\,153$ & $970$ & $0$ & $0.35$ & $0$ & $0$ & $0$ \\
$11$ & $11$ & $+1$ & $3\,263$ & $884$ & $8\,331$ & $884$ & $0$ & $0.69$ & $0$ & $0$ & $0$ \\
$11$ & $11$ & $-1$ & $3\,254$ & $884$ & $8\,326$ & $884$ & $0$ & $0.46$ & $0$ & $0$ & $0$ \\
\addlinespace[3pt]
$12$ & $0$ & $+1$ & $4\,181$ & $1\,321$ & $11\,845$ & $1\,321$ & $0$ & $0.67$ & $0$ & $0$ & $0$ \\
$12$ & $0$ & $-1$ & $4\,164$ & $1\,321$ & $11\,833$ & $1\,321$ & $0$ & $0.21$ & $0$ & $0$ & $0$ \\
$12$ & $1$ & $+1$ & $4\,556$ & $1\,618$ & $13\,401$ & $1\,618$ & $0$ & $0.56$ & $0$ & $0$ & $0$ \\
$12$ & $1$ & $-1$ & $4\,535$ & $1\,618$ & $13\,391$ & $1\,618$ & $0$ & $0.37$ & $0$ & $0$ & $0$ \\
$12$ & $2$ & $+1$ & $5\,024$ & $1\,852$ & $15\,731$ & $1\,846$ & $6$ & $0.26$ & $0$ & $0$ & $0$ \\
$12$ & $2$ & $-1$ & $5\,004$ & $1\,852$ & $15\,716$ & $1\,843$ & $9$ & $0.65$ & $0$ & $0$ & $0$ \\
$12$ & $3$ & $+1$ & $5\,640$ & $2\,078$ & $19\,886$ & $2\,074$ & $4$ & $0.74$ & $0$ & $0$ & $0.01$ \\
$12$ & $3$ & $-1$ & $5\,621$ & $2\,078$ & $19\,873$ & $2\,072$ & $6$ & $0.76$ & $0$ & $0$ & $0.01$ \\
$12$ & $4$ & $+1$ & $5\,429$ & $1\,945$ & $18\,588$ & $1\,943$ & $2$ & $0.31$ & $0$ & $0$ & $0$ \\
$12$ & $4$ & $-1$ & $5\,411$ & $1\,945$ & $18\,575$ & $1\,941$ & $4$ & $0.31$ & $0$ & $0$ & $0$ \\
$12$ & $5$ & $+1$ & $5\,223$ & $1\,817$ & $17\,340$ & $1\,816$ & $1$ & $0.47$ & $0$ & $0$ & $0$ \\
$12$ & $5$ & $-1$ & $5\,206$ & $1\,817$ & $17\,329$ & $1\,815$ & $2$ & $0.65$ & $0$ & $0$ & $0$ \\
$12$ & $6$ & $+1$ & $5\,022$ & $1\,694$ & $16\,144$ & $1\,694$ & $0$ & $0.28$ & $0$ & $0$ & $0$ \\
$12$ & $6$ & $-1$ & $5\,006$ & $1\,694$ & $16\,133$ & $1\,693$ & $1$ & $0.29$ & $0$ & $0$ & $0$ \\
$12$ & $7$ & $+1$ & $4\,826$ & $1\,576$ & $14\,998$ & $1\,576$ & $0$ & $0.64$ & $0$ & $0$ & $0$ \\
$12$ & $7$ & $-1$ & $4\,811$ & $1\,576$ & $14\,989$ & $1\,576$ & $0$ & $0.25$ & $0$ & $0$ & $0$ \\
$12$ & $8$ & $+1$ & $4\,635$ & $1\,463$ & $13\,904$ & $1\,463$ & $0$ & $0.62$ & $0$ & $0$ & $0$ \\
$12$ & $8$ & $-1$ & $4\,621$ & $1\,463$ & $13\,895$ & $1\,463$ & $0$ & $0.52$ & $0$ & $0$ & $0$ \\
$12$ & $9$ & $+1$ & $4\,448$ & $1\,355$ & $12\,860$ & $1\,355$ & $0$ & $0.72$ & $0$ & $0$ & $0$ \\
$12$ & $9$ & $-1$ & $4\,436$ & $1\,355$ & $12\,853$ & $1\,355$ & $0$ & $0.22$ & $0$ & $0$ & $0$ \\
$12$ & $10$ & $+1$ & $4\,264$ & $1\,252$ & $11\,868$ & $1\,252$ & $0$ & $0.61$ & $0$ & $0$ & $0$ \\
$12$ & $10$ & $-1$ & $4\,253$ & $1\,252$ & $11\,861$ & $1\,252$ & $0$ & $0.49$ & $0$ & $0$ & $0$ \\
$12$ & $11$ & $+1$ & $4\,078$ & $1\,154$ & $10\,926$ & $1\,154$ & $0$ & $0.70$ & $0$ & $0$ & $0$ \\
$12$ & $11$ & $-1$ & $4\,068$ & $1\,154$ & $10\,921$ & $1\,154$ & $0$ & $0.19$ & $0$ & $0$ & $0$ \\
$12$ & $12$ & $+1$ & $3\,887$ & $1\,060$ & $10\,023$ & $1\,060$ & $0$ & $0.58$ & $0$ & $0$ & $0$ \\
$12$ & $12$ & $-1$ & $3\,877$ & $1\,060$ & $10\,017$ & $1\,060$ & $0$ & $0.40$ & $0$ & $0$ & $0$ \\
\end{longtable}
\endgroup

\subsubsection{General type}

The general-type test example, identified by \texttt{normal\_type} in
the implementation, is the degree-$15$ hypersurface
\begin{equation}
 X^{15}+Y^{15}+Z^{15}+T^{15}
 +T^8Y^7+T^8Z^7=0.
 \label{eq:compact-general-surface}
\end{equation}
Thus $d=14$ throughout this family. We test $m=3,\ldots,11$, with twists
\begin{equation}
 t(m)=
 \begin{cases}
  -1,&m=3,4,5,6,\\
  -2,&m=7,\\
  -3,&m=8,9,10,11.
 \end{cases}
 \label{eq:compact-twists}
\end{equation}
For each $m$, both signs and all $15$ weights are included, giving
$30$ matrices per row and $270$ matrices in total.
Table~\ref{tab:compact-general-sizes} lists the matrix sizes, sparse-entry
counts, ranks and column nullities for every tested $m$.

\begin{table}[htbp]
\centering
\caption{General-type results for the degree-$15$ test surface ($d=14$), with all $e=\pm1$ and $0\le r\le14$. The expected column nullity is zero in every row.}
\label{tab:compact-general-sizes}
\small
\begin{tabular}{rrrrrrrr}
\toprule
$m$ & $t$ & $K$ & $N_{\Sigma}$ & $C_{\Sigma}$ & $Z_{\Sigma}$ & $R_{\Sigma}$ & $\kappa_{\Sigma}$ \\
\midrule
$3$ & $-1$ & $30$ & $193\,892$ & $65\,052$ & $1\,824\,187$ & $65\,052$ & $0$ \\
$4$ & $-1$ & $30$ & $469\,703$ & $166\,390$ & $10\,125\,391$ & $166\,390$ & $0$ \\
$5$ & $-1$ & $30$ & $965\,032$ & $352\,350$ & $43\,329\,975$ & $352\,350$ & $0$ \\
$6$ & $-1$ & $30$ & $1\,759\,662$ & $657\,240$ & $151\,409\,706$ & $657\,240$ & $0$ \\
$7$ & $-2$ & $30$ & $2\,912\,800$ & $1\,094\,520$ & $445\,986\,615$ & $1\,094\,520$ & $0$ \\
$8$ & $-3$ & $30$ & $4\,551\,132$ & $1\,716\,480$ & $1\,172\,170\,882$ & $1\,716\,480$ & $0$ \\
$9$ & $-3$ & $30$ & $6\,876\,774$ & $2\,626\,720$ & $2\,901\,083\,044$ & $2\,626\,720$ & $0$ \\
$10$ & $-3$ & $30$ & $9\,993\,426$ & $3\,855\,980$ & $6\,649\,145\,588$ & $3\,855\,980$ & $0$ \\
$11$ & $-3$ & $30$ & $14\,055\,674$ & $5\,465\,100$ & $14\,208\,428\,966$ & $5\,465\,100$ & $0$ \\
\bottomrule
\end{tabular}
\end{table}

Every individual general-type sector has full column rank:
\begin{equation}
 \kappa_{14,m,t(m),r,e}^{(32749)}=0
 \quad\text{for all tested }m,r,e.
 \label{eq:compact-general-results}
\end{equation}
In particular, the total column count and the sum of the modular ranks
are both $15\,999\,832$. This statement concerns the specified
degree-$15$ example and parameter range; it is not a claim about all
hypersurfaces of general type. The corresponding generation, loading,
solve and total rank-stage times are given in
Table~\ref{tab:compact-general-times}.

\begin{table}[htbp]
\centering
\caption{General-type logged execution times in seconds. Each entry is the sum over the sectors in the corresponding row of Table~\ref{tab:compact-general-sizes}.}
\label{tab:compact-general-times}
\small
\begin{tabular}{rrrrr}
\toprule
$m$ & $T_{\rm gen}$ & $T_{\rm load}$ & $T_{\rm solve}$ & $T_{\rm rank}$ \\
\midrule
$3$ & $37.87$ & $0.43$ & $11.55$ & $11.98$ \\
$4$ & $170.39$ & $2.37$ & $578.64$ & $581.01$ \\
$5$ & $708.49$ & $10.16$ & $8\,431.10$ & $8\,441.28$ \\
$6$ & $2\,501.90$ & $65.28$ & $8\,333.81$ & $8\,399.15$ \\
$7$ & $7\,403.88$ & $199.92$ & $6\,239.88$ & $6\,439.85$ \\
$8$ & $19\,463.39$ & $535.67$ & $17\,121.87$ & $17\,657.60$ \\
$9$ & $46\,750.70$ & $1\,403.45$ & $44\,423.12$ & $45\,826.61$ \\
$10$ & $20\,559.76$ & $5\,173.77$ & $110\,528.49$ & $115\,702.43$ \\
$11$ & $52\,858.06$ & $18\,668.31$ & $144\,814.71$ & $163\,483.09$ \\
\bottomrule
\end{tabular}
\end{table}

\clearpage
Table~\ref{tab:compact-general-details} lists all $270$ individual
general-type matrices, with $m$, $t$, $r$ and $e$ specified in every row.
Their dimensions, sparse-entry counts, modular ranks, column nullities
and four logged execution times reproduce the group summaries above.

\begingroup
\footnotesize
\setlength{\tabcolsep}{2pt}
\setlength{\LTcapwidth}{\textwidth}
\renewcommand{\arraystretch}{1.05}
\begin{longtable}{@{}rrrrrrrrrrrrr@{}}
\caption{Individual general-type matrices for the degree-$15$ test surface
($d=14$), $m=3,\ldots,11$, $0\le r\le14$ and $e=\pm1$.
The twist $t=t(m)$ is given in each row. Each row describes one matrix;
$\rho$ is its rank over $\mathbb{F}_{32749}$ and
$\kappa=n_{\rm col}-\rho$. Every matrix has $\kappa=0$.
All four execution times are in seconds, with two decimal places; values below $0.01$ seconds are shown as $0$.}
\label{tab:compact-general-details}\\
\toprule
$m$ & $t$ & $r$ & $e$ & $n_{\rm row}$ & $n_{\rm col}$ & $\operatorname{nnz}(A)$ & $\rho$ & $\kappa$ & $T_{\rm gen}$ & $T_{\rm load}$ & $T_{\rm solve}$ & $T_{\rm rank}$ \\
\midrule
\endfirsthead
\caption[]{Individual general-type matrices for $d=14$ (continued).
All execution times are in seconds.}\\
\toprule
$m$ & $t$ & $r$ & $e$ & $n_{\rm row}$ & $n_{\rm col}$ & $\operatorname{nnz}(A)$ & $\rho$ & $\kappa$ & $T_{\rm gen}$ & $T_{\rm load}$ & $T_{\rm solve}$ & $T_{\rm rank}$ \\
\midrule
\endhead
\midrule
\multicolumn{13}{r}{\footnotesize Continued on the next page.}\\
\endfoot
\bottomrule
\endlastfoot
$3$ & $-1$ & $0$ & $+1$ & $5\,704$ & $1\,834$ & $40\,839$ & $1\,834$ & $0$ & $0.89$ & $0$ & $0.12$ & $0.13$ \\
$3$ & $-1$ & $0$ & $-1$ & $5\,684$ & $1\,834$ & $40\,831$ & $1\,834$ & $0$ & $0.68$ & $0$ & $0.12$ & $0.13$ \\
$3$ & $-1$ & $1$ & $+1$ & $6\,253$ & $2\,263$ & $44\,073$ & $2\,263$ & $0$ & $0.89$ & $0.01$ & $0.29$ & $0.30$ \\
$3$ & $-1$ & $1$ & $-1$ & $6\,228$ & $2\,263$ & $44\,072$ & $2\,263$ & $0$ & $1.05$ & $0.01$ & $0.29$ & $0.30$ \\
$3$ & $-1$ & $2$ & $+1$ & $6\,924$ & $2\,604$ & $55\,727$ & $2\,604$ & $0$ & $1.15$ & $0.01$ & $1.00$ & $1.02$ \\
$3$ & $-1$ & $2$ & $-1$ & $6\,900$ & $2\,604$ & $55\,705$ & $2\,604$ & $0$ & $0.87$ & $0.01$ & $1.03$ & $1.05$ \\
$3$ & $-1$ & $3$ & $+1$ & $7\,795$ & $2\,935$ & $89\,672$ & $2\,935$ & $0$ & $1.59$ & $0.02$ & $1.01$ & $1.04$ \\
$3$ & $-1$ & $3$ & $-1$ & $7\,772$ & $2\,935$ & $89\,671$ & $2\,935$ & $0$ & $1.49$ & $0.02$ & $1.06$ & $1.08$ \\
$3$ & $-1$ & $4$ & $+1$ & $7\,546$ & $2\,776$ & $84\,352$ & $2\,776$ & $0$ & $1.71$ & $0.02$ & $0.74$ & $0.76$ \\
$3$ & $-1$ & $4$ & $-1$ & $7\,524$ & $2\,776$ & $84\,347$ & $2\,776$ & $0$ & $1.56$ & $0.02$ & $0.80$ & $0.82$ \\
$3$ & $-1$ & $5$ & $+1$ & $7\,302$ & $2\,622$ & $79\,238$ & $2\,622$ & $0$ & $1.61$ & $0.02$ & $0.57$ & $0.59$ \\
$3$ & $-1$ & $5$ & $-1$ & $7\,281$ & $2\,622$ & $79\,237$ & $2\,622$ & $0$ & $1.81$ & $0.02$ & $0.60$ & $0.62$ \\
$3$ & $-1$ & $6$ & $+1$ & $7\,063$ & $2\,473$ & $74\,344$ & $2\,473$ & $0$ & $1.77$ & $0.02$ & $0.47$ & $0.49$ \\
$3$ & $-1$ & $6$ & $-1$ & $7\,043$ & $2\,473$ & $74\,339$ & $2\,473$ & $0$ & $1.46$ & $0.02$ & $0.47$ & $0.49$ \\
$3$ & $-1$ & $7$ & $+1$ & $6\,829$ & $2\,329$ & $69\,656$ & $2\,329$ & $0$ & $1.19$ & $0.02$ & $0.40$ & $0.41$ \\
$3$ & $-1$ & $7$ & $-1$ & $6\,810$ & $2\,329$ & $69\,655$ & $2\,329$ & $0$ & $1.62$ & $0.02$ & $0.40$ & $0.42$ \\
$3$ & $-1$ & $8$ & $+1$ & $6\,600$ & $2\,190$ & $65\,188$ & $2\,190$ & $0$ & $1.71$ & $0.02$ & $0.30$ & $0.32$ \\
$3$ & $-1$ & $8$ & $-1$ & $6\,582$ & $2\,190$ & $65\,183$ & $2\,190$ & $0$ & $1.53$ & $0.02$ & $0.30$ & $0.32$ \\
$3$ & $-1$ & $9$ & $+1$ & $6\,376$ & $2\,056$ & $60\,926$ & $2\,056$ & $0$ & $1.68$ & $0.01$ & $0.22$ & $0.24$ \\
$3$ & $-1$ & $9$ & $-1$ & $6\,359$ & $2\,056$ & $60\,925$ & $2\,056$ & $0$ & $1.37$ & $0.01$ & $0.23$ & $0.24$ \\
$3$ & $-1$ & $10$ & $+1$ & $6\,157$ & $1\,927$ & $56\,882$ & $1\,927$ & $0$ & $1.24$ & $0.01$ & $0.17$ & $0.18$ \\
$3$ & $-1$ & $10$ & $-1$ & $6\,141$ & $1\,927$ & $56\,878$ & $1\,927$ & $0$ & $1.10$ & $0.01$ & $0.17$ & $0.18$ \\
$3$ & $-1$ & $11$ & $+1$ & $5\,943$ & $1\,803$ & $53\,046$ & $1\,803$ & $0$ & $1.28$ & $0.01$ & $0.13$ & $0.15$ \\
$3$ & $-1$ & $11$ & $-1$ & $5\,928$ & $1\,803$ & $53\,046$ & $1\,803$ & $0$ & $1.28$ & $0.01$ & $0.13$ & $0.15$ \\
$3$ & $-1$ & $12$ & $+1$ & $5\,734$ & $1\,684$ & $49\,428$ & $1\,684$ & $0$ & $0.99$ & $0.01$ & $0.10$ & $0.12$ \\
$3$ & $-1$ & $12$ & $-1$ & $5\,720$ & $1\,684$ & $49\,425$ & $1\,684$ & $0$ & $1.03$ & $0.01$ & $0.11$ & $0.12$ \\
$3$ & $-1$ & $13$ & $+1$ & $5\,530$ & $1\,570$ & $46\,018$ & $1\,570$ & $0$ & $0.77$ & $0.01$ & $0.08$ & $0.09$ \\
$3$ & $-1$ & $13$ & $-1$ & $5\,517$ & $1\,570$ & $46\,019$ & $1\,570$ & $0$ & $1.16$ & $0.01$ & $0.08$ & $0.10$ \\
$3$ & $-1$ & $14$ & $+1$ & $5\,330$ & $1\,460$ & $42\,734$ & $1\,460$ & $0$ & $0.67$ & $0.01$ & $0.07$ & $0.08$ \\
$3$ & $-1$ & $14$ & $-1$ & $5\,317$ & $1\,460$ & $42\,731$ & $1\,460$ & $0$ & $0.74$ & $0.01$ & $0.06$ & $0.08$ \\
\addlinespace[3pt]
$4$ & $-1$ & $0$ & $+1$ & $14\,037$ & $4\,701$ & $243\,629$ & $4\,701$ & $0$ & $3.98$ & $0.06$ & $11.46$ & $11.52$ \\
$4$ & $-1$ & $0$ & $-1$ & $14\,067$ & $4\,701$ & $243\,653$ & $4\,701$ & $0$ & $3.95$ & $0.06$ & $11.26$ & $11.32$ \\
$4$ & $-1$ & $1$ & $+1$ & $14\,842$ & $5\,400$ & $245\,731$ & $5\,400$ & $0$ & $4.06$ & $0.06$ & $16.13$ & $16.19$ \\
$4$ & $-1$ & $1$ & $-1$ & $14\,880$ & $5\,400$ & $245\,743$ & $5\,400$ & $0$ & $4.97$ & $0.06$ & $15.46$ & $15.52$ \\
$4$ & $-1$ & $2$ & $+1$ & $15\,984$ & $6\,076$ & $267\,512$ & $6\,076$ & $0$ & $4.32$ & $0.06$ & $25.13$ & $25.19$ \\
$4$ & $-1$ & $2$ & $-1$ & $16\,021$ & $6\,076$ & $267\,535$ & $6\,076$ & $0$ & $4.30$ & $0.06$ & $25.65$ & $25.72$ \\
$4$ & $-1$ & $3$ & $+1$ & $16\,848$ & $6\,474$ & $326\,663$ & $6\,474$ & $0$ & $5.07$ & $0.08$ & $32.53$ & $32.60$ \\
$4$ & $-1$ & $3$ & $-1$ & $16\,884$ & $6\,474$ & $326\,659$ & $6\,474$ & $0$ & $5.09$ & $0.08$ & $33.48$ & $33.55$ \\
$4$ & $-1$ & $4$ & $+1$ & $17\,914$ & $6\,864$ & $465\,585$ & $6\,864$ & $0$ & $7.71$ & $0.11$ & $33.02$ & $33.13$ \\
$4$ & $-1$ & $4$ & $-1$ & $17\,948$ & $6\,864$ & $465\,601$ & $6\,864$ & $0$ & $8.40$ & $0.11$ & $34.26$ & $34.37$ \\
$4$ & $-1$ & $5$ & $+1$ & $17\,457$ & $6\,571$ & $442\,818$ & $6\,571$ & $0$ & $7.43$ & $0.10$ & $29.10$ & $29.20$ \\
$4$ & $-1$ & $5$ & $-1$ & $17\,490$ & $6\,571$ & $442\,839$ & $6\,571$ & $0$ & $7.33$ & $0.10$ & $29.96$ & $30.06$ \\
$4$ & $-1$ & $6$ & $+1$ & $17\,006$ & $6\,285$ & $420\,725$ & $6\,285$ & $0$ & $7.27$ & $0.10$ & $25.92$ & $26.02$ \\
$4$ & $-1$ & $6$ & $-1$ & $17\,038$ & $6\,285$ & $420\,740$ & $6\,285$ & $0$ & $6.88$ & $0.10$ & $26.39$ & $26.49$ \\
$4$ & $-1$ & $7$ & $+1$ & $16\,561$ & $6\,006$ & $399\,286$ & $6\,006$ & $0$ & $6.63$ & $0.09$ & $23.21$ & $23.30$ \\
$4$ & $-1$ & $7$ & $-1$ & $16\,592$ & $6\,006$ & $399\,306$ & $6\,006$ & $0$ & $6.67$ & $0.09$ & $23.77$ & $23.86$ \\
$4$ & $-1$ & $8$ & $+1$ & $16\,122$ & $5\,734$ & $378\,521$ & $5\,734$ & $0$ & $6.21$ & $0.09$ & $20.07$ & $20.16$ \\
$4$ & $-1$ & $8$ & $-1$ & $16\,152$ & $5\,734$ & $378\,535$ & $5\,734$ & $0$ & $6.23$ & $0.09$ & $20.54$ & $20.62$ \\
$4$ & $-1$ & $9$ & $+1$ & $15\,689$ & $5\,469$ & $358\,410$ & $5\,469$ & $0$ & $6.15$ & $0.08$ & $16.82$ & $16.91$ \\
$4$ & $-1$ & $9$ & $-1$ & $15\,718$ & $5\,469$ & $358\,429$ & $5\,469$ & $0$ & $6.53$ & $0.08$ & $17.44$ & $17.53$ \\
$4$ & $-1$ & $10$ & $+1$ & $15\,262$ & $5\,211$ & $338\,971$ & $5\,211$ & $0$ & $5.79$ & $0.08$ & $14.27$ & $14.34$ \\
$4$ & $-1$ & $10$ & $-1$ & $15\,290$ & $5\,211$ & $338\,985$ & $5\,211$ & $0$ & $5.58$ & $0.08$ & $14.70$ & $14.78$ \\
$4$ & $-1$ & $11$ & $+1$ & $14\,841$ & $4\,960$ & $320\,188$ & $4\,960$ & $0$ & $5.25$ & $0.07$ & $12.03$ & $12.11$ \\
$4$ & $-1$ & $11$ & $-1$ & $14\,868$ & $4\,960$ & $320\,207$ & $4\,960$ & $0$ & $5.20$ & $0.08$ & $12.63$ & $12.70$ \\
$4$ & $-1$ & $12$ & $+1$ & $14\,426$ & $4\,716$ & $302\,077$ & $4\,716$ & $0$ & $5.43$ & $0.07$ & $10.25$ & $10.32$ \\
$4$ & $-1$ & $12$ & $-1$ & $14\,453$ & $4\,716$ & $302\,091$ & $4\,716$ & $0$ & $5.12$ & $0.07$ & $10.77$ & $10.84$ \\
$4$ & $-1$ & $13$ & $+1$ & $14\,017$ & $4\,479$ & $284\,622$ & $4\,479$ & $0$ & $4.76$ & $0.07$ & $8.76$ & $8.82$ \\
$4$ & $-1$ & $13$ & $-1$ & $14\,042$ & $4\,479$ & $284\,641$ & $4\,479$ & $0$ & $4.90$ & $0.07$ & $9.04$ & $9.10$ \\
$4$ & $-1$ & $14$ & $+1$ & $13\,615$ & $4\,249$ & $267\,838$ & $4\,249$ & $0$ & $4.78$ & $0.06$ & $7.32$ & $7.38$ \\
$4$ & $-1$ & $14$ & $-1$ & $13\,639$ & $4\,249$ & $267\,851$ & $4\,249$ & $0$ & $4.39$ & $0.06$ & $7.27$ & $7.34$ \\
\addlinespace[3pt]
$5$ & $-1$ & $0$ & $+1$ & $29\,601$ & $10\,206$ & $1\,163\,874$ & $10\,206$ & $0$ & $19.65$ & $0.27$ & $202.41$ & $202.69$ \\
$5$ & $-1$ & $0$ & $-1$ & $29\,556$ & $10\,206$ & $1\,163\,867$ & $10\,206$ & $0$ & $18.09$ & $0.28$ & $201.06$ & $201.34$ \\
$5$ & $-1$ & $1$ & $+1$ & $30\,504$ & $11\,124$ & $1\,118\,012$ & $11\,124$ & $0$ & $18.80$ & $0.26$ & $260.52$ & $260.78$ \\
$5$ & $-1$ & $1$ & $-1$ & $30\,453$ & $11\,124$ & $1\,118\,007$ & $11\,124$ & $0$ & $18.53$ & $0.27$ & $262.81$ & $263.08$ \\
$5$ & $-1$ & $2$ & $+1$ & $31\,761$ & $12\,021$ & $1\,107\,823$ & $12\,021$ & $0$ & $18.10$ & $0.26$ & $316.91$ & $317.17$ \\
$5$ & $-1$ & $2$ & $-1$ & $31\,711$ & $12\,021$ & $1\,107\,814$ & $12\,021$ & $0$ & $17.99$ & $0.27$ & $316.28$ & $316.55$ \\
$5$ & $-1$ & $3$ & $+1$ & $33\,417$ & $12\,897$ & $1\,203\,113$ & $12\,897$ & $0$ & $19.44$ & $0.28$ & $357.47$ & $357.75$ \\
$5$ & $-1$ & $3$ & $-1$ & $33\,368$ & $12\,897$ & $1\,203\,112$ & $12\,897$ & $0$ & $19.31$ & $0.28$ & $364.44$ & $364.73$ \\
$5$ & $-1$ & $4$ & $+1$ & $34\,392$ & $13\,302$ & $1\,428\,767$ & $13\,302$ & $0$ & $22.24$ & $0.34$ & $384.19$ & $384.53$ \\
$5$ & $-1$ & $4$ & $-1$ & $34\,344$ & $13\,302$ & $1\,428\,731$ & $13\,302$ & $0$ & $22.10$ & $0.33$ & $391.05$ & $391.38$ \\
$5$ & $-1$ & $5$ & $+1$ & $35\,571$ & $13\,701$ & $1\,879\,825$ & $13\,701$ & $0$ & $30.99$ & $0.44$ & $389.88$ & $390.31$ \\
$5$ & $-1$ & $5$ & $-1$ & $35\,524$ & $13\,701$ & $1\,879\,825$ & $13\,701$ & $0$ & $30.94$ & $0.44$ & $397.83$ & $398.28$ \\
$5$ & $-1$ & $6$ & $+1$ & $34\,824$ & $13\,224$ & $1\,805\,343$ & $13\,224$ & $0$ & $28.90$ & $0.43$ & $360.73$ & $361.15$ \\
$5$ & $-1$ & $6$ & $-1$ & $34\,778$ & $13\,224$ & $1\,805\,336$ & $13\,224$ & $0$ & $29.36$ & $0.42$ & $364.74$ & $365.16$ \\
$5$ & $-1$ & $7$ & $+1$ & $34\,086$ & $12\,756$ & $1\,732\,491$ & $12\,756$ & $0$ & $28.47$ & $0.41$ & $329.89$ & $330.30$ \\
$5$ & $-1$ & $7$ & $-1$ & $34\,041$ & $12\,756$ & $1\,732\,491$ & $12\,756$ & $0$ & $28.44$ & $0.40$ & $336.23$ & $336.64$ \\
$5$ & $-1$ & $8$ & $+1$ & $33\,357$ & $12\,297$ & $1\,661\,295$ & $12\,297$ & $0$ & $27.57$ & $0.39$ & $300.48$ & $300.88$ \\
$5$ & $-1$ & $8$ & $-1$ & $33\,313$ & $12\,297$ & $1\,661\,288$ & $12\,297$ & $0$ & $27.37$ & $0.39$ & $305.55$ & $305.94$ \\
$5$ & $-1$ & $9$ & $+1$ & $32\,637$ & $11\,847$ & $1\,591\,729$ & $11\,847$ & $0$ & $26.03$ & $0.37$ & $268.90$ & $269.27$ \\
$5$ & $-1$ & $9$ & $-1$ & $32\,594$ & $11\,847$ & $1\,591\,729$ & $11\,847$ & $0$ & $26.29$ & $0.37$ & $274.50$ & $274.87$ \\
$5$ & $-1$ & $10$ & $+1$ & $31\,926$ & $11\,406$ & $1\,523\,817$ & $11\,406$ & $0$ & $25.20$ & $0.36$ & $244.78$ & $245.14$ \\
$5$ & $-1$ & $10$ & $-1$ & $31\,884$ & $11\,406$ & $1\,523\,811$ & $11\,406$ & $0$ & $24.59$ & $0.36$ & $248.95$ & $249.31$ \\
$5$ & $-1$ & $11$ & $+1$ & $31\,224$ & $10\,974$ & $1\,457\,537$ & $10\,974$ & $0$ & $24.60$ & $0.34$ & $221.84$ & $222.19$ \\
$5$ & $-1$ & $11$ & $-1$ & $31\,183$ & $10\,974$ & $1\,457\,538$ & $10\,974$ & $0$ & $23.94$ & $0.34$ & $226.76$ & $227.10$ \\
$5$ & $-1$ & $12$ & $+1$ & $30\,531$ & $10\,551$ & $1\,392\,911$ & $10\,551$ & $0$ & $22.89$ & $0.32$ & $202.33$ & $202.65$ \\
$5$ & $-1$ & $12$ & $-1$ & $30\,491$ & $10\,551$ & $1\,392\,906$ & $10\,551$ & $0$ & $23.06$ & $0.33$ & $205.29$ & $205.61$ \\
$5$ & $-1$ & $13$ & $+1$ & $29\,847$ & $10\,137$ & $1\,329\,917$ & $10\,137$ & $0$ & $21.98$ & $0.31$ & $182.40$ & $182.72$ \\
$5$ & $-1$ & $13$ & $-1$ & $29\,808$ & $10\,137$ & $1\,329\,919$ & $10\,137$ & $0$ & $21.79$ & $0.31$ & $185.20$ & $185.51$ \\
$5$ & $-1$ & $14$ & $+1$ & $29\,172$ & $9\,732$ & $1\,268\,576$ & $9\,732$ & $0$ & $20.77$ & $0.29$ & $163.62$ & $163.91$ \\
$5$ & $-1$ & $14$ & $-1$ & $29\,134$ & $9\,732$ & $1\,268\,571$ & $9\,732$ & $0$ & $21.05$ & $0.30$ & $164.02$ & $164.32$ \\
\addlinespace[3pt]
$6$ & $-1$ & $0$ & $+1$ & $54\,844$ & $19\,491$ & $4\,379\,275$ & $19\,491$ & $0$ & $72.46$ & $1.02$ & $1\,718.57$ & $1\,719.60$ \\
$6$ & $-1$ & $0$ & $-1$ & $54\,906$ & $19\,491$ & $4\,379\,318$ & $19\,491$ & $0$ & $72.58$ & $1.99$ & $68.18$ & $70.18$ \\
$6$ & $-1$ & $1$ & $+1$ & $55\,807$ & $20\,565$ & $4\,134\,224$ & $20\,565$ & $0$ & $68.38$ & $0.98$ & $2\,035.03$ & $2\,036.01$ \\
$6$ & $-1$ & $1$ & $-1$ & $55\,875$ & $20\,565$ & $4\,134\,251$ & $20\,565$ & $0$ & $68.98$ & $1.89$ & $70.57$ & $72.47$ \\
$6$ & $-1$ & $2$ & $+1$ & $57\,307$ & $21\,846$ & $4\,060\,208$ & $21\,846$ & $0$ & $68.05$ & $0.95$ & $2\,376.59$ & $2\,377.55$ \\
$6$ & $-1$ & $2$ & $-1$ & $57\,381$ & $21\,846$ & $4\,060\,254$ & $21\,846$ & $0$ & $67.63$ & $1.85$ & $71.48$ & $73.33$ \\
$6$ & $-1$ & $3$ & $+1$ & $58\,976$ & $22\,869$ & $4\,073\,751$ & $22\,869$ & $0$ & $67.06$ & $1.84$ & $75.58$ & $77.42$ \\
$6$ & $-1$ & $3$ & $-1$ & $59\,049$ & $22\,869$ & $4\,073\,785$ & $22\,869$ & $0$ & $66.82$ & $1.84$ & $74.04$ & $75.89$ \\
$6$ & $-1$ & $4$ & $+1$ & $61\,048$ & $23\,874$ & $4\,433\,733$ & $23\,874$ & $0$ & $71.40$ & $1.97$ & $89.56$ & $91.53$ \\
$6$ & $-1$ & $4$ & $-1$ & $61\,119$ & $23\,874$ & $4\,433\,779$ & $23\,874$ & $0$ & $71.66$ & $2.00$ & $89.72$ & $91.72$ \\
$6$ & $-1$ & $5$ & $+1$ & $62\,036$ & $24\,216$ & $5\,114\,806$ & $24\,216$ & $0$ & $79.57$ & $2.27$ & $94.53$ & $96.80$ \\
$6$ & $-1$ & $5$ & $-1$ & $62\,106$ & $24\,216$ & $5\,114\,811$ & $24\,216$ & $0$ & $80.58$ & $2.28$ & $94.71$ & $96.99$ \\
$6$ & $-1$ & $6$ & $+1$ & $63\,232$ & $24\,555$ & $6\,337\,105$ & $24\,555$ & $0$ & $105.92$ & $2.84$ & $96.74$ & $99.58$ \\
$6$ & $-1$ & $6$ & $-1$ & $63\,300$ & $24\,555$ & $6\,337\,140$ & $24\,555$ & $0$ & $104.74$ & $2.83$ & $96.86$ & $99.69$ \\
$6$ & $-1$ & $7$ & $+1$ & $62\,099$ & $23\,826$ & $6\,117\,448$ & $23\,826$ & $0$ & $102.53$ & $2.78$ & $91.15$ & $93.94$ \\
$6$ & $-1$ & $7$ & $-1$ & $62\,166$ & $23\,826$ & $6\,117\,491$ & $23\,826$ & $0$ & $101.96$ & $2.83$ & $92.42$ & $95.26$ \\
$6$ & $-1$ & $8$ & $+1$ & $60\,979$ & $23\,109$ & $5\,902\,357$ & $23\,109$ & $0$ & $98.89$ & $2.67$ & $87.12$ & $89.79$ \\
$6$ & $-1$ & $8$ & $-1$ & $61\,044$ & $23\,109$ & $5\,902\,390$ & $23\,109$ & $0$ & $97.77$ & $2.64$ & $86.21$ & $88.85$ \\
$6$ & $-1$ & $9$ & $+1$ & $59\,870$ & $22\,404$ & $5\,691\,796$ & $22\,404$ & $0$ & $94.29$ & $2.52$ & $83.19$ & $85.71$ \\
$6$ & $-1$ & $9$ & $-1$ & $59\,934$ & $22\,404$ & $5\,691\,837$ & $22\,404$ & $0$ & $94.04$ & $2.57$ & $83.88$ & $86.45$ \\
$6$ & $-1$ & $10$ & $+1$ & $58\,774$ & $21\,711$ & $5\,485\,799$ & $21\,711$ & $0$ & $93.64$ & $2.45$ & $81.93$ & $84.38$ \\
$6$ & $-1$ & $10$ & $-1$ & $58\,836$ & $21\,711$ & $5\,485\,831$ & $21\,711$ & $0$ & $91.67$ & $2.45$ & $82.95$ & $85.40$ \\
$6$ & $-1$ & $11$ & $+1$ & $57\,689$ & $21\,030$ & $5\,284\,334$ & $21\,030$ & $0$ & $87.24$ & $2.36$ & $79.40$ & $81.76$ \\
$6$ & $-1$ & $11$ & $-1$ & $57\,750$ & $21\,030$ & $5\,284\,374$ & $21\,030$ & $0$ & $88.34$ & $2.36$ & $78.98$ & $81.34$ \\
$6$ & $-1$ & $12$ & $+1$ & $56\,617$ & $20\,361$ & $5\,087\,433$ & $20\,361$ & $0$ & $83.95$ & $2.26$ & $76.89$ & $79.14$ \\
$6$ & $-1$ & $12$ & $-1$ & $56\,676$ & $20\,361$ & $5\,087\,464$ & $20\,361$ & $0$ & $84.68$ & $2.29$ & $76.68$ & $78.98$ \\
$6$ & $-1$ & $13$ & $+1$ & $55\,556$ & $19\,704$ & $4\,895\,064$ & $19\,704$ & $0$ & $80.41$ & $2.18$ & $72.46$ & $74.64$ \\
$6$ & $-1$ & $13$ & $-1$ & $55\,614$ & $19\,704$ & $4\,895\,103$ & $19\,704$ & $0$ & $81.51$ & $2.17$ & $72.09$ & $74.27$ \\
$6$ & $-1$ & $14$ & $+1$ & $54\,508$ & $19\,059$ & $4\,707\,258$ & $19\,059$ & $0$ & $77.40$ & $2.08$ & $68.10$ & $70.18$ \\
$6$ & $-1$ & $14$ & $-1$ & $54\,564$ & $19\,059$ & $4\,707\,287$ & $19\,059$ & $0$ & $77.75$ & $2.12$ & $68.19$ & $70.31$ \\
\addlinespace[3pt]
$7$ & $-2$ & $0$ & $+1$ & $92\,057$ & $33\,093$ & $13\,619\,575$ & $33\,093$ & $0$ & $229.32$ & $6.11$ & $189.96$ & $196.08$ \\
$7$ & $-2$ & $0$ & $-1$ & $91\,978$ & $33\,093$ & $13\,619\,561$ & $33\,093$ & $0$ & $227.27$ & $6.06$ & $197.47$ & $203.54$ \\
$7$ & $-2$ & $1$ & $+1$ & $92\,966$ & $34\,227$ & $12\,843\,500$ & $34\,227$ & $0$ & $218.03$ & $5.85$ & $189.73$ & $195.59$ \\
$7$ & $-2$ & $1$ & $-1$ & $92\,881$ & $34\,227$ & $12\,843\,502$ & $34\,227$ & $0$ & $216.74$ & $5.87$ & $188.81$ & $194.69$ \\
$7$ & $-2$ & $2$ & $+1$ & $94\,551$ & $35\,751$ & $12\,449\,346$ & $35\,751$ & $0$ & $207.88$ & $5.66$ & $172.36$ & $178.02$ \\
$7$ & $-2$ & $2$ & $-1$ & $94\,459$ & $35\,751$ & $12\,449\,325$ & $35\,751$ & $0$ & $204.75$ & $5.66$ & $172.61$ & $178.27$ \\
$7$ & $-2$ & $3$ & $+1$ & $96\,720$ & $37\,245$ & $12\,283\,195$ & $37\,245$ & $0$ & $204.16$ & $5.67$ & $190.09$ & $195.76$ \\
$7$ & $-2$ & $3$ & $-1$ & $96\,629$ & $37\,245$ & $12\,283\,201$ & $37\,245$ & $0$ & $203.96$ & $5.58$ & $189.95$ & $195.54$ \\
$7$ & $-2$ & $4$ & $+1$ & $98\,664$ & $38\,304$ & $12\,528\,651$ & $38\,304$ & $0$ & $207.64$ & $5.73$ & $190.68$ & $196.41$ \\
$7$ & $-2$ & $4$ & $-1$ & $98\,575$ & $38\,304$ & $12\,528\,646$ & $38\,304$ & $0$ & $207.30$ & $5.69$ & $189.60$ & $195.29$ \\
$7$ & $-2$ & $5$ & $+1$ & $101\,013$ & $39\,348$ & $13\,636\,084$ & $39\,348$ & $0$ & $220.09$ & $6.13$ & $206.99$ & $213.12$ \\
$7$ & $-2$ & $5$ & $-1$ & $100\,925$ & $39\,348$ & $13\,636\,092$ & $39\,348$ & $0$ & $219.21$ & $6.04$ & $209.47$ & $215.51$ \\
$7$ & $-2$ & $6$ & $+1$ & $101\,892$ & $39\,552$ & $15\,344\,082$ & $39\,552$ & $0$ & $242.66$ & $6.77$ & $236.43$ & $243.20$ \\
$7$ & $-2$ & $6$ & $-1$ & $101\,806$ & $39\,552$ & $15\,344\,044$ & $39\,552$ & $0$ & $240.98$ & $6.81$ & $239.01$ & $245.82$ \\
$7$ & $-2$ & $7$ & $+1$ & $102\,981$ & $39\,756$ & $18\,182\,859$ & $39\,756$ & $0$ & $305.12$ & $8.20$ & $245.91$ & $254.12$ \\
$7$ & $-2$ & $7$ & $-1$ & $102\,896$ & $39\,756$ & $18\,182\,855$ & $39\,756$ & $0$ & $301.77$ & $8.06$ & $245.81$ & $253.87$ \\
$7$ & $-2$ & $8$ & $+1$ & $101\,355$ & $38\,715$ & $17\,620\,749$ & $38\,715$ & $0$ & $293.31$ & $7.83$ & $235.78$ & $243.61$ \\
$7$ & $-2$ & $8$ & $-1$ & $101\,272$ & $38\,715$ & $17\,620\,756$ & $38\,715$ & $0$ & $296.23$ & $7.84$ & $235.74$ & $243.59$ \\
$7$ & $-2$ & $9$ & $+1$ & $99\,744$ & $37\,689$ & $17\,068\,746$ & $37\,689$ & $0$ & $283.75$ & $7.61$ & $231.75$ & $239.37$ \\
$7$ & $-2$ & $9$ & $-1$ & $99\,662$ & $37\,689$ & $17\,068\,742$ & $37\,689$ & $0$ & $286.70$ & $7.60$ & $231.47$ & $239.07$ \\
$7$ & $-2$ & $10$ & $+1$ & $98\,148$ & $36\,678$ & $16\,526\,810$ & $36\,678$ & $0$ & $279.61$ & $7.38$ & $224.64$ & $232.02$ \\
$7$ & $-2$ & $10$ & $-1$ & $98\,068$ & $36\,678$ & $16\,526\,817$ & $36\,678$ & $0$ & $275.98$ & $7.41$ & $223.29$ & $230.70$ \\
$7$ & $-2$ & $11$ & $+1$ & $96\,567$ & $35\,682$ & $15\,994\,981$ & $35\,682$ & $0$ & $268.99$ & $7.18$ & $214.74$ & $221.92$ \\
$7$ & $-2$ & $11$ & $-1$ & $96\,488$ & $35\,682$ & $15\,994\,977$ & $35\,682$ & $0$ & $266.59$ & $7.18$ & $214.00$ & $221.17$ \\
$7$ & $-2$ & $12$ & $+1$ & $95\,001$ & $34\,701$ & $15\,473\,219$ & $34\,701$ & $0$ & $256.88$ & $6.88$ & $202.94$ & $209.82$ \\
$7$ & $-2$ & $12$ & $-1$ & $94\,924$ & $34\,701$ & $15\,473\,226$ & $34\,701$ & $0$ & $257.31$ & $6.97$ & $203.54$ & $210.51$ \\
$7$ & $-2$ & $13$ & $+1$ & $93\,450$ & $33\,735$ & $14\,961\,563$ & $33\,735$ & $0$ & $250.15$ & $6.78$ & $195.94$ & $202.72$ \\
$7$ & $-2$ & $13$ & $-1$ & $93\,374$ & $33\,735$ & $14\,961\,558$ & $33\,735$ & $0$ & $250.94$ & $6.60$ & $194.86$ & $201.46$ \\
$7$ & $-2$ & $14$ & $+1$ & $91\,914$ & $32\,784$ & $14\,459\,974$ & $32\,784$ & $0$ & $240.65$ & $6.39$ & $188.03$ & $194.43$ \\
$7$ & $-2$ & $14$ & $-1$ & $91\,840$ & $32\,784$ & $14\,459\,979$ & $32\,784$ & $0$ & $239.94$ & $6.38$ & $188.28$ & $194.67$ \\
\addlinespace[3pt]
$8$ & $-3$ & $0$ & $+1$ & $145\,342$ & $52\,830$ & $37\,449\,092$ & $52\,830$ & $0$ & $625.90$ & $17.05$ & $520.39$ & $537.44$ \\
$8$ & $-3$ & $0$ & $-1$ & $145\,440$ & $52\,830$ & $37\,449\,159$ & $52\,830$ & $0$ & $630.02$ & $17.05$ & $521.36$ & $538.41$ \\
$8$ & $-3$ & $1$ & $+1$ & $146\,080$ & $53\,949$ & $35\,475\,128$ & $53\,949$ & $0$ & $594.86$ & $16.59$ & $476.68$ & $493.27$ \\
$8$ & $-3$ & $1$ & $-1$ & $146\,184$ & $53\,949$ & $35\,475\,174$ & $53\,949$ & $0$ & $600.33$ & $16.22$ & $476.52$ & $492.74$ \\
$8$ & $-3$ & $2$ & $+1$ & $147\,631$ & $55\,641$ & $34\,169\,564$ & $55\,641$ & $0$ & $575.62$ & $15.61$ & $481.69$ & $497.30$ \\
$8$ & $-3$ & $2$ & $-1$ & $147\,741$ & $55\,641$ & $34\,169\,634$ & $55\,641$ & $0$ & $574.12$ & $15.63$ & $482.05$ & $497.68$ \\
$8$ & $-3$ & $3$ & $+1$ & $149\,747$ & $57\,306$ & $33\,192\,757$ & $57\,306$ & $0$ & $554.55$ & $15.37$ & $503.11$ & $518.48$ \\
$8$ & $-3$ & $3$ & $-1$ & $149\,856$ & $57\,306$ & $33\,192\,795$ & $57\,306$ & $0$ & $558.19$ & $15.46$ & $501.15$ & $516.61$ \\
$8$ & $-3$ & $4$ & $+1$ & $152\,467$ & $58\,944$ & $33\,028\,685$ & $58\,944$ & $0$ & $555.68$ & $15.32$ & $496.38$ & $511.70$ \\
$8$ & $-3$ & $4$ & $-1$ & $152\,574$ & $58\,944$ & $33\,028\,756$ & $58\,944$ & $0$ & $551.52$ & $15.36$ & $496.06$ & $511.42$ \\
$8$ & $-3$ & $5$ & $+1$ & $154\,574$ & $59\,970$ & $34\,056\,595$ & $59\,970$ & $0$ & $559.85$ & $15.76$ & $544.62$ & $560.38$ \\
$8$ & $-3$ & $5$ & $-1$ & $154\,680$ & $59\,970$ & $34\,056\,654$ & $59\,970$ & $0$ & $563.27$ & $15.40$ & $544.18$ & $559.58$ \\
$8$ & $-3$ & $6$ & $+1$ & $157\,090$ & $60\,984$ & $36\,940\,438$ & $60\,984$ & $0$ & $593.43$ & $16.53$ & $583.59$ & $600.12$ \\
$8$ & $-3$ & $6$ & $-1$ & $157\,194$ & $60\,984$ & $36\,940\,509$ & $60\,984$ & $0$ & $592.54$ & $16.44$ & $584.60$ & $601.04$ \\
$8$ & $-3$ & $7$ & $+1$ & $157\,748$ & $60\,981$ & $40\,756\,440$ & $60\,981$ & $0$ & $644.02$ & $18.35$ & $677.48$ & $695.84$ \\
$8$ & $-3$ & $7$ & $-1$ & $157\,851$ & $60\,981$ & $40\,756\,460$ & $60\,981$ & $0$ & $640.51$ & $18.24$ & $676.43$ & $694.67$ \\
$8$ & $-3$ & $8$ & $+1$ & $158\,620$ & $60\,981$ & $46\,749\,519$ & $60\,981$ & $0$ & $785.18$ & $21.16$ & $701.33$ & $722.50$ \\
$8$ & $-3$ & $8$ & $-1$ & $158\,721$ & $60\,981$ & $46\,749\,576$ & $60\,981$ & $0$ & $781.41$ & $21.20$ & $699.49$ & $720.69$ \\
$8$ & $-3$ & $9$ & $+1$ & $156\,389$ & $59\,559$ & $45\,467\,567$ & $59\,559$ & $0$ & $760.99$ & $20.78$ & $646.95$ & $667.72$ \\
$8$ & $-3$ & $9$ & $-1$ & $156\,489$ & $59\,559$ & $45\,467\,637$ & $59\,559$ & $0$ & $761.39$ & $20.99$ & $647.29$ & $668.29$ \\
$8$ & $-3$ & $10$ & $+1$ & $154\,177$ & $58\,155$ & $44\,205\,368$ & $58\,155$ & $0$ & $736.51$ & $20.38$ & $640.17$ & $660.55$ \\
$8$ & $-3$ & $10$ & $-1$ & $154\,275$ & $58\,155$ & $44\,205\,423$ & $58\,155$ & $0$ & $734.83$ & $20.08$ & $641.50$ & $661.58$ \\
$8$ & $-3$ & $11$ & $+1$ & $151\,982$ & $56\,769$ & $42\,962\,870$ & $56\,769$ & $0$ & $718.65$ & $19.49$ & $608.74$ & $628.23$ \\
$8$ & $-3$ & $11$ & $-1$ & $152\,079$ & $56\,769$ & $42\,962\,938$ & $56\,769$ & $0$ & $712.26$ & $19.65$ & $608.42$ & $628.07$ \\
$8$ & $-3$ & $12$ & $+1$ & $149\,806$ & $55\,401$ & $41\,740\,124$ & $55\,401$ & $0$ & $695.34$ & $19.00$ & $579.15$ & $598.15$ \\
$8$ & $-3$ & $12$ & $-1$ & $149\,901$ & $55\,401$ & $41\,740\,176$ & $55\,401$ & $0$ & $691.10$ & $19.24$ & $579.80$ & $599.05$ \\
$8$ & $-3$ & $13$ & $+1$ & $147\,647$ & $54\,051$ & $40\,537\,079$ & $54\,051$ & $0$ & $680.01$ & $18.72$ & $547.73$ & $566.45$ \\
$8$ & $-3$ & $13$ & $-1$ & $147\,741$ & $54\,051$ & $40\,537\,143$ & $54\,051$ & $0$ & $680.42$ & $18.55$ & $548.36$ & $566.92$ \\
$8$ & $-3$ & $14$ & $+1$ & $145\,507$ & $52\,719$ & $39\,353\,787$ & $52\,719$ & $0$ & $655.08$ & $18.06$ & $553.29$ & $571.35$ \\
$8$ & $-3$ & $14$ & $-1$ & $145\,599$ & $52\,719$ & $39\,353\,835$ & $52\,719$ & $0$ & $655.80$ & $18.00$ & $553.37$ & $571.37$ \\
\addlinespace[3pt]
$9$ & $-3$ & $0$ & $+1$ & $221\,806$ & $82\,081$ & $95\,834\,678$ & $82\,081$ & $0$ & $1\,611.46$ & $43.49$ & $1\,439.73$ & $1\,483.22$ \\
$9$ & $-3$ & $0$ & $-1$ & $221\,678$ & $82\,081$ & $95\,834\,675$ & $82\,081$ & $0$ & $1\,601.54$ & $43.84$ & $1\,443.46$ & $1\,487.31$ \\
$9$ & $-3$ & $1$ & $+1$ & $222\,270$ & $83\,100$ & $91\,228\,047$ & $83\,100$ & $0$ & $1\,536.87$ & $42.28$ & $1\,414.42$ & $1\,456.71$ \\
$9$ & $-3$ & $1$ & $-1$ & $222\,137$ & $83\,100$ & $91\,228\,036$ & $83\,100$ & $0$ & $1\,547.35$ & $42.16$ & $1\,409.88$ & $1\,452.04$ \\
$9$ & $-3$ & $2$ & $+1$ & $223\,701$ & $84\,891$ & $87\,587\,643$ & $84\,891$ & $0$ & $1\,472.51$ & $40.42$ & $1\,328.62$ & $1\,369.04$ \\
$9$ & $-3$ & $2$ & $-1$ & $223\,562$ & $84\,891$ & $87\,587\,630$ & $84\,891$ & $0$ & $1\,470.38$ & $39.93$ & $1\,328.75$ & $1\,368.69$ \\
$9$ & $-3$ & $3$ & $+1$ & $226\,009$ & $87\,019$ & $85\,376\,483$ & $87\,019$ & $0$ & $1\,416.78$ & $39.49$ & $1\,284.83$ & $1\,324.31$ \\
$9$ & $-3$ & $3$ & $-1$ & $225\,865$ & $87\,019$ & $85\,376\,473$ & $87\,019$ & $0$ & $1\,427.83$ & $39.47$ & $1\,285.55$ & $1\,325.03$ \\
$9$ & $-3$ & $4$ & $+1$ & $228\,549$ & $88\,749$ & $83\,555\,231$ & $88\,749$ & $0$ & $1\,399.75$ & $38.08$ & $1\,380.53$ & $1\,418.62$ \\
$9$ & $-3$ & $4$ & $-1$ & $228\,407$ & $88\,749$ & $83\,555\,199$ & $88\,749$ & $0$ & $1\,397.35$ & $38.21$ & $1\,379.48$ & $1\,417.69$ \\
$9$ & $-3$ & $5$ & $+1$ & $231\,696$ & $90\,456$ & $83\,884\,151$ & $90\,456$ & $0$ & $1\,405.69$ & $38.30$ & $1\,414.41$ & $1\,452.70$ \\
$9$ & $-3$ & $5$ & $-1$ & $231\,556$ & $90\,456$ & $83\,884\,151$ & $90\,456$ & $0$ & $1\,379.41$ & $38.38$ & $1\,415.41$ & $1\,453.79$ \\
$9$ & $-3$ & $6$ & $+1$ & $233\,830$ & $91\,360$ & $86\,965\,335$ & $91\,360$ & $0$ & $1\,439.56$ & $39.29$ & $1\,500.61$ & $1\,539.90$ \\
$9$ & $-3$ & $6$ & $-1$ & $233\,692$ & $91\,360$ & $86\,965\,327$ & $91\,360$ & $0$ & $1\,427.23$ & $39.21$ & $1\,500.91$ & $1\,540.13$ \\
$9$ & $-3$ & $7$ & $+1$ & $236\,376$ & $92\,256$ & $93\,777\,953$ & $92\,256$ & $0$ & $1\,520.41$ & $41.60$ & $1\,569.40$ & $1\,611.00$ \\
$9$ & $-3$ & $7$ & $-1$ & $236\,240$ & $92\,256$ & $93\,777\,959$ & $92\,256$ & $0$ & $1\,508.94$ & $41.56$ & $1\,551.84$ & $1\,593.41$ \\
$9$ & $-3$ & $8$ & $+1$ & $236\,664$ & $91\,944$ & $101\,734\,586$ & $91\,944$ & $0$ & $1\,619.90$ & $45.34$ & $1\,612.18$ & $1\,657.52$ \\
$9$ & $-3$ & $8$ & $-1$ & $236\,530$ & $91\,944$ & $101\,734\,526$ & $91\,944$ & $0$ & $1\,605.86$ & $45.91$ & $1\,614.66$ & $1\,660.57$ \\
$9$ & $-3$ & $9$ & $+1$ & $237\,169$ & $91\,639$ & $113\,717\,930$ & $91\,639$ & $0$ & $1\,906.07$ & $51.98$ & $1\,770.04$ & $1\,822.02$ \\
$9$ & $-3$ & $9$ & $-1$ & $237\,037$ & $91\,639$ & $113\,717\,939$ & $91\,639$ & $0$ & $1\,930.67$ & $52.27$ & $1\,769.70$ & $1\,821.97$ \\
$9$ & $-3$ & $10$ & $+1$ & $234\,171$ & $89\,721$ & $110\,882\,606$ & $89\,721$ & $0$ & $1\,858.02$ & $50.07$ & $1\,732.32$ & $1\,782.39$ \\
$9$ & $-3$ & $10$ & $-1$ & $234\,041$ & $89\,721$ & $110\,882\,598$ & $89\,721$ & $0$ & $1\,861.13$ & $49.98$ & $1\,721.94$ & $1\,771.92$ \\
$9$ & $-3$ & $11$ & $+1$ & $231\,195$ & $87\,825$ & $108\,088\,248$ & $87\,825$ & $0$ & $1\,809.10$ & $52.70$ & $1\,637.81$ & $1\,690.51$ \\
$9$ & $-3$ & $11$ & $-1$ & $231\,067$ & $87\,825$ & $108\,088\,257$ & $87\,825$ & $0$ & $1\,781.13$ & $49.29$ & $1\,637.71$ & $1\,687.00$ \\
$9$ & $-3$ & $12$ & $+1$ & $228\,241$ & $85\,951$ & $105\,334\,919$ & $85\,951$ & $0$ & $1\,755.04$ & $47.61$ & $1\,617.12$ & $1\,664.73$ \\
$9$ & $-3$ & $12$ & $-1$ & $228\,115$ & $85\,951$ & $105\,334\,910$ & $85\,951$ & $0$ & $1\,763.36$ & $47.60$ & $1\,615.80$ & $1\,663.40$ \\
$9$ & $-3$ & $13$ & $+1$ & $225\,309$ & $84\,099$ & $102\,622\,556$ & $84\,099$ & $0$ & $1\,716.86$ & $46.87$ & $1\,514.95$ & $1\,561.83$ \\
$9$ & $-3$ & $13$ & $-1$ & $225\,185$ & $84\,099$ & $102\,622\,563$ & $84\,099$ & $0$ & $1\,703.57$ & $47.39$ & $1\,515.78$ & $1\,563.17$ \\
$9$ & $-3$ & $14$ & $+1$ & $222\,399$ & $82\,269$ & $99\,951\,223$ & $82\,269$ & $0$ & $1\,654.91$ & $46.32$ & $1\,477.45$ & $1\,523.76$ \\
$9$ & $-3$ & $14$ & $-1$ & $222\,277$ & $82\,269$ & $99\,951\,212$ & $82\,269$ & $0$ & $222.03$ & $124.40$ & $537.83$ & $662.23$ \\
\addlinespace[3pt]
$10$ & $-3$ & $0$ & $+1$ & $324\,672$ & $122\,070$ & $225\,807\,127$ & $122\,070$ & $0$ & $704.54$ & $179.19$ & $3\,602.83$ & $3\,782.02$ \\
$10$ & $-3$ & $0$ & $-1$ & $324\,825$ & $122\,070$ & $225\,807\,222$ & $122\,070$ & $0$ & $703.45$ & $173.98$ & $3\,596.34$ & $3\,770.32$ \\
$10$ & $-3$ & $1$ & $+1$ & $324\,697$ & $122\,881$ & $216\,119\,837$ & $122\,881$ & $0$ & $671.22$ & $162.36$ & $3\,523.03$ & $3\,685.40$ \\
$10$ & $-3$ & $1$ & $-1$ & $324\,856$ & $122\,881$ & $216\,119\,905$ & $122\,881$ & $0$ & $669.27$ & $165.27$ & $3\,523.88$ & $3\,689.14$ \\
$10$ & $-3$ & $2$ & $+1$ & $325\,844$ & $124\,663$ & $207\,546\,926$ & $124\,663$ & $0$ & $646.42$ & $159.10$ & $3\,401.35$ & $3\,560.45$ \\
$10$ & $-3$ & $2$ & $-1$ & $326\,008$ & $124\,663$ & $207\,547\,024$ & $124\,663$ & $0$ & $645.71$ & $162.12$ & $3\,406.27$ & $3\,568.40$ \\
$10$ & $-3$ & $3$ & $+1$ & $328\,021$ & $126\,981$ & $201\,531\,213$ & $126\,981$ & $0$ & $611.42$ & $153.50$ & $3\,275.46$ & $3\,428.96$ \\
$10$ & $-3$ & $3$ & $-1$ & $328\,191$ & $126\,981$ & $201\,531\,280$ & $126\,981$ & $0$ & $610.10$ & $166.68$ & $3\,276.75$ & $3\,443.44$ \\
$10$ & $-3$ & $4$ & $+1$ & $330\,997$ & $129\,265$ & $196\,948\,163$ & $129\,265$ & $0$ & $582.37$ & $151.34$ & $3\,317.67$ & $3\,469.01$ \\
$10$ & $-3$ & $4$ & $-1$ & $331\,165$ & $129\,265$ & $196\,948\,256$ & $129\,265$ & $0$ & $585.13$ & $153.46$ & $3\,315.45$ & $3\,468.91$ \\
$10$ & $-3$ & $5$ & $+1$ & $333\,804$ & $130\,960$ & $194\,534\,047$ & $130\,960$ & $0$ & $552.40$ & $147.09$ & $3\,363.61$ & $3\,510.71$ \\
$10$ & $-3$ & $5$ & $-1$ & $333\,970$ & $130\,960$ & $194\,534\,091$ & $130\,960$ & $0$ & $555.12$ & $150.76$ & $3\,366.01$ & $3\,516.77$ \\
$10$ & $-3$ & $6$ & $+1$ & $337\,222$ & $132\,636$ & $196\,821\,579$ & $132\,636$ & $0$ & $527.45$ & $146.10$ & $3\,609.98$ & $3\,756.08$ \\
$10$ & $-3$ & $6$ & $-1$ & $337\,386$ & $132\,636$ & $196\,821\,669$ & $132\,636$ & $0$ & $532.97$ & $168.16$ & $3\,608.99$ & $3\,777.16$ \\
$10$ & $-3$ & $7$ & $+1$ & $339\,226$ & $133\,318$ & $204\,479\,709$ & $133\,318$ & $0$ & $586.80$ & $151.65$ & $3\,634.62$ & $3\,786.27$ \\
$10$ & $-3$ & $7$ & $-1$ & $339\,388$ & $133\,318$ & $204\,479\,789$ & $133\,318$ & $0$ & $594.78$ & $147.13$ & $3\,633.03$ & $3\,780.17$ \\
$10$ & $-3$ & $8$ & $+1$ & $341\,646$ & $133\,996$ & $219\,073\,605$ & $133\,996$ & $0$ & $704.44$ & $162.43$ & $3\,855.16$ & $4\,017.59$ \\
$10$ & $-3$ & $8$ & $-1$ & $341\,806$ & $133\,996$ & $219\,073\,697$ & $133\,996$ & $0$ & $698.59$ & $191.08$ & $3\,860.83$ & $4\,051.91$ \\
$10$ & $-3$ & $9$ & $+1$ & $341\,407$ & $133\,275$ & $234\,485\,423$ & $133\,275$ & $0$ & $838.64$ & $182.30$ & $3\,671.92$ & $3\,854.23$ \\
$10$ & $-3$ & $9$ & $-1$ & $341\,565$ & $133\,275$ & $234\,485\,455$ & $133\,275$ & $0$ & $837.82$ & $174.74$ & $3\,673.37$ & $3\,848.12$ \\
$10$ & $-3$ & $10$ & $+1$ & $341\,389$ & $132\,565$ & $256\,953\,765$ & $132\,565$ & $0$ & $805.65$ & $202.45$ & $4\,230.30$ & $4\,432.76$ \\
$10$ & $-3$ & $10$ & $-1$ & $341\,545$ & $132\,565$ & $256\,953\,842$ & $132\,565$ & $0$ & $811.21$ & $199.89$ & $4\,228.32$ & $4\,428.22$ \\
$10$ & $-3$ & $11$ & $+1$ & $337\,467$ & $130\,051$ & $251\,123\,790$ & $130\,051$ & $0$ & $797.49$ & $202.16$ & $4\,070.22$ & $4\,272.39$ \\
$10$ & $-3$ & $11$ & $-1$ & $337\,621$ & $130\,051$ & $251\,123\,886$ & $130\,051$ & $0$ & $789.98$ & $188.97$ & $4\,073.40$ & $4\,262.38$ \\
$10$ & $-3$ & $12$ & $+1$ & $333\,571$ & $127\,563$ & $245\,369\,388$ & $127\,563$ & $0$ & $765.65$ & $191.74$ & $4\,061.20$ & $4\,252.95$ \\
$10$ & $-3$ & $12$ & $-1$ & $333\,723$ & $127\,563$ & $245\,369\,462$ & $127\,563$ & $0$ & $770.31$ & $184.48$ & $4\,059.24$ & $4\,243.73$ \\
$10$ & $-3$ & $13$ & $+1$ & $329\,701$ & $125\,101$ & $239\,690\,484$ & $125\,101$ & $0$ & $750.36$ & $184.46$ & $3\,907.11$ & $4\,091.58$ \\
$10$ & $-3$ & $13$ & $-1$ & $329\,851$ & $125\,101$ & $239\,690\,576$ & $125\,101$ & $0$ & $752.87$ & $217.24$ & $3\,909.20$ & $4\,126.44$ \\
$10$ & $-3$ & $14$ & $+1$ & $325\,857$ & $122\,665$ & $234\,087\,154$ & $122\,665$ & $0$ & $728.55$ & $174.52$ & $3\,735.72$ & $3\,910.25$ \\
$10$ & $-3$ & $14$ & $-1$ & $326\,005$ & $122\,665$ & $234\,087\,224$ & $122\,665$ & $0$ & $729.07$ & $179.44$ & $3\,737.23$ & $3\,916.67$ \\
\addlinespace[3pt]
$11$ & $-3$ & $0$ & $+1$ & $459\,925$ & $174\,985$ & $493\,941\,502$ & $174\,985$ & $0$ & $1\,915.08$ & $643.16$ & $4\,787.63$ & $5\,430.78$ \\
$11$ & $-3$ & $0$ & $-1$ & $459\,745$ & $174\,985$ & $493\,941\,501$ & $174\,985$ & $0$ & $1\,923.77$ & $648.94$ & $4\,786.55$ & $5\,435.49$ \\
$11$ & $-3$ & $1$ & $+1$ & $459\,370$ & $175\,480$ & $475\,210\,934$ & $175\,480$ & $0$ & $1\,809.82$ & $628.51$ & $4\,662.61$ & $5\,291.11$ \\
$11$ & $-3$ & $1$ & $-1$ & $459\,185$ & $175\,480$ & $475\,210\,919$ & $175\,480$ & $0$ & $1\,803.07$ & $628.00$ & $4\,663.71$ & $5\,291.72$ \\
$11$ & $-3$ & $2$ & $+1$ & $460\,090$ & $177\,145$ & $457\,339\,582$ & $177\,145$ & $0$ & $1\,691.66$ & $603.08$ & $4\,549.11$ & $5\,152.19$ \\
$11$ & $-3$ & $2$ & $-1$ & $459\,899$ & $177\,145$ & $457\,339\,571$ & $177\,145$ & $0$ & $1\,728.61$ & $598.72$ & $4\,547.46$ & $5\,146.19$ \\
$11$ & $-3$ & $3$ & $+1$ & $461\,995$ & $179\,545$ & $443\,380\,349$ & $179\,545$ & $0$ & $1\,638.10$ & $588.94$ & $4\,491.62$ & $5\,080.58$ \\
$11$ & $-3$ & $3$ & $-1$ & $461\,799$ & $179\,545$ & $443\,380\,335$ & $179\,545$ & $0$ & $1\,623.53$ & $579.58$ & $4\,492.05$ & $5\,071.63$ \\
$11$ & $-3$ & $4$ & $+1$ & $464\,665$ & $181\,915$ & $431\,453\,862$ & $181\,915$ & $0$ & $1\,540.06$ & $565.96$ & $4\,452.86$ & $5\,018.82$ \\
$11$ & $-3$ & $4$ & $-1$ & $464\,471$ & $181\,915$ & $431\,453\,848$ & $181\,915$ & $0$ & $1\,548.85$ & $571.77$ & $4\,451.75$ & $5\,023.53$ \\
$11$ & $-3$ & $5$ & $+1$ & $468\,145$ & $184\,255$ & $423\,720\,263$ & $184\,255$ & $0$ & $1\,496.04$ & $561.65$ & $4\,442.45$ & $5\,004.10$ \\
$11$ & $-3$ & $5$ & $-1$ & $467\,953$ & $184\,255$ & $423\,720\,249$ & $184\,255$ & $0$ & $1\,474.70$ & $555.54$ & $4\,442.52$ & $4\,998.06$ \\
$11$ & $-3$ & $6$ & $+1$ & $471\,055$ & $185\,815$ & $421\,949\,343$ & $185\,815$ & $0$ & $1\,428.70$ & $551.50$ & $4\,476.76$ & $5\,028.26$ \\
$11$ & $-3$ & $6$ & $-1$ & $470\,865$ & $185\,815$ & $421\,949\,292$ & $185\,815$ & $0$ & $1\,433.62$ & $552.87$ & $4\,474.89$ & $5\,027.76$ \\
$11$ & $-3$ & $7$ & $+1$ & $474\,580$ & $187\,360$ & $429\,273\,522$ & $187\,360$ & $0$ & $1\,375.83$ & $560.25$ & $4\,550.88$ & $5\,111.13$ \\
$11$ & $-3$ & $7$ & $-1$ & $474\,392$ & $187\,360$ & $429\,273\,519$ & $187\,360$ & $0$ & $1\,376.77$ & $555.63$ & $4\,551.08$ & $5\,106.71$ \\
$11$ & $-3$ & $8$ & $+1$ & $476\,290$ & $187\,720$ & $445\,975\,205$ & $187\,720$ & $0$ & $1\,527.15$ & $574.65$ & $4\,723.29$ & $5\,297.95$ \\
$11$ & $-3$ & $8$ & $-1$ & $476\,104$ & $187\,720$ & $445\,975\,192$ & $187\,720$ & $0$ & $1\,517.80$ & $586.54$ & $4\,725.94$ & $5\,312.49$ \\
$11$ & $-3$ & $9$ & $+1$ & $478\,420$ & $188\,080$ & $474\,858\,870$ & $188\,080$ & $0$ & $1\,840.64$ & $628.41$ & $5\,015.83$ & $5\,644.24$ \\
$11$ & $-3$ & $9$ & $-1$ & $478\,236$ & $188\,080$ & $474\,858\,876$ & $188\,080$ & $0$ & $1\,818.04$ & $619.23$ & $5\,015.65$ & $5\,634.89$ \\
$11$ & $-3$ & $10$ & $+1$ & $477\,490$ & $186\,850$ & $502\,912\,142$ & $186\,850$ & $0$ & $2\,130.19$ & $657.19$ & $5\,243.63$ & $5\,900.82$ \\
$11$ & $-3$ & $10$ & $-1$ & $477\,308$ & $186\,850$ & $502\,912\,066$ & $186\,850$ & $0$ & $2\,132.65$ & $654.00$ & $5\,244.36$ & $5\,898.36$ \\
$11$ & $-3$ & $11$ & $+1$ & $476\,785$ & $185\,635$ & $542\,697\,008$ & $185\,635$ & $0$ & $2\,063.55$ & $724.00$ & $5\,561.91$ & $6\,285.92$ \\
$11$ & $-3$ & $11$ & $-1$ & $476\,605$ & $185\,635$ & $542\,697\,019$ & $185\,635$ & $0$ & $2\,071.29$ & $710.82$ & $5\,561.06$ & $6\,271.88$ \\
$11$ & $-3$ & $12$ & $+1$ & $471\,775$ & $182\,425$ & $531\,513\,048$ & $182\,425$ & $0$ & $2\,023.35$ & $697.43$ & $5\,347.98$ & $6\,045.41$ \\
$11$ & $-3$ & $12$ & $-1$ & $471\,597$ & $182\,425$ & $531\,513\,035$ & $182\,425$ & $0$ & $2\,038.59$ & $699.09$ & $5\,349.20$ & $6\,048.28$ \\
$11$ & $-3$ & $13$ & $+1$ & $466\,795$ & $179\,245$ & $520\,457\,752$ & $179\,245$ & $0$ & $2\,003.07$ & $693.77$ & $5\,146.23$ & $5\,840.00$ \\
$11$ & $-3$ & $13$ & $-1$ & $466\,619$ & $179\,245$ & $520\,457\,761$ & $179\,245$ & $0$ & $1\,999.07$ & $688.98$ & $5\,145.72$ & $5\,834.70$ \\
$11$ & $-3$ & $14$ & $+1$ & $461\,845$ & $176\,095$ & $509\,531\,208$ & $176\,095$ & $0$ & $1\,936.73$ & $668.97$ & $4\,954.80$ & $5\,623.77$ \\
$11$ & $-3$ & $14$ & $-1$ & $461\,671$ & $176\,095$ & $509\,531\,193$ & $176\,095$ & $0$ & $1\,947.71$ & $671.13$ & $4\,955.18$ & $5\,626.32$ \\
\end{longtable}
\endgroup

\paragraph{Interpretation and complete data.}
All $450$ sector parameter combinations agree with the stated modular nullity
expectations. The accompanying machine-readable data file lists every
sector's parameters, row and column counts, sparse-entry count, modular
rank, column nullity and recorded timings. The group summaries above
sum the corresponding per-sector records; the individual matrix
tables report those records directly.

For an integer coefficient matrix, or a rational coefficient matrix
whose denominators are invertible modulo $32749$, one has
\begin{equation}
 \operatorname{rank}_{\mathbb{F}_{32749}}(A\bmod32749)
 \le \operatorname{rank}_{\mathbb{Q}}(A).
\end{equation}
Accordingly, an established full-column-rank modular result also forces
full column rank over $\mathbb{Q}$. A positive modular column nullity,
however, gives an upper bound on the rational nullity and need not equal
it. The positive column nullities reported here are therefore
finite-field verification results; by themselves they do not prove
the corresponding nullity equalities in characteristic zero.

\clearpage
\subsubsection{The degree-16 compatibility example with root-weight sectors}
\label{app:compact-compat-d15}

For this test we use the degree-$16$ surface
\begin{equation}
 X^{16}+Y^{16}+Z^{16}+T^{16}+Y^7T^9+Z^7T^9=0,
 \qquad d=15,
 \label{eq:compact-compat-d15-surface}
\end{equation}
so that the affine term $U$ in~\eqref{eq:compact-affine-surface} is
$U(x,y,z)=1+y^7+z^7$. The generator identifies this polynomial as
\texttt{compat\_normal\_type}. The tests cover $(m,t)=(3,-2)$ and
$(4,-3)$, every $r=0,\ldots,15$, and both signs $e=\pm1$. The
\texttt{unit\_roots} option is enabled, imposing the congruence
restriction~\eqref{eq:compact-weight} modulo $16$. Thus there are $64$
distinct coefficient matrices, rather than full matrices without an
$r$ subdivision.

The modified generator was rebuilt before these runs. Each matrix was
generated in the \texttt{batch} partition using $4$ MPI processes and
column-oriented HDF5 output. Formula~\eqref{eq:compact-block-count}
gives $B(3)=6$ and $B(4)=8$ blocks per matrix, for $448$ HDF5 blocks in
all. Their file sets, headers and small scalar metadata were checked.
The rank stage used \texttt{WiedemannGlobal} over
$\mathbb{F}_{32749}$ in the \texttt{dcu\_big} partition, with one MPI
process and one DCU per matrix. All $64$ final rank records and their
completed main jobs were verified.

Table~\ref{tab:compact-compat-d15-summary} summarizes the two weighted
degrees. Every matrix has full column rank modulo $32749$, so the sum
of column nullities is zero. The measured generation durations sum to
$36.81$ seconds and the rank benchmarks to $206.83$ seconds over the
$64$ cases. These sums combine independently scheduled jobs and are
not the elapsed time of the parallel batch. Table~\ref{tab:compact-compat-d15-details}
gives each matrix's dimensions, sparse-entry count, modular rank,
column nullity, and logged execution times.

\begin{table}[htbp]
\centering
\caption{Summary of the $64$ \texttt{compat\_normal\_type} matrices for
$d=15$, with $0\le r\le15$ and $e=\pm1$.
Here $K$ is the number of matrices, $N_{\Sigma}$, $C_{\Sigma}$ and
$Z_{\Sigma}$ sum row counts, column counts and sparse-entry counts,
respectively; $R_{\Sigma}$ sums ranks over $\mathbb{F}_{32749}$ and
$\kappa_{\Sigma}=C_{\Sigma}-R_{\Sigma}$.}
\label{tab:compact-compat-d15-summary}
\small
\setlength{\tabcolsep}{3pt}
\begin{tabular}{rrrrrrrr}
\toprule
$m$ & $t$ & $K$ & $N_{\Sigma}$ & $C_{\Sigma}$ & $Z_{\Sigma}$ & $R_{\Sigma}$ & $\kappa_{\Sigma}$ \\
\midrule
$3$ & $-2$ & $32$ & $230\,664$ & $75\,900$ & $2\,135\,737$ & $75\,900$ & $0$ \\
$4$ & $-3$ & $32$ & $548\,357$ & $188\,320$ & $11\,418\,210$ & $188\,320$ & $0$ \\
\midrule
\multicolumn{2}{r}{Total} & $64$ & $779\,021$ & $264\,220$ & $13\,553\,947$ & $264\,220$ & $0$ \\
\bottomrule
\end{tabular}
\end{table}

\begingroup
\footnotesize
\setlength{\tabcolsep}{2pt}
\setlength{\LTcapwidth}{\textwidth}
\renewcommand{\arraystretch}{1.05}
\begin{longtable}{@{}rrrrrrrrrrrrr@{}}
\caption{Individual $d=15$ \texttt{compat\_normal\_type} matrices for
$(m,t)=(3,-2),(4,-3)$, $0\le r\le15$ and $e=\pm1$.
Here $\rho=\operatorname{rank}_{\mathbb{F}_{32749}}(A)$ and
$\kappa=n_{\rm col}-\rho$. Generation, loading, solving and total
rank-stage times are measured in seconds and displayed to two decimal
places; values below $0.01$ seconds are shown as $0$.} 
\label{tab:compact-compat-d15-details}\\
\toprule
$m$ & $t$ & $r$ & $e$ & $n_{\rm row}$ & $n_{\rm col}$ & $\operatorname{nnz}(A)$ & $\rho$ & $\kappa$ & $T_{\rm gen}$ & $T_{\rm load}$ & $T_{\rm solve}$ & $T_{\rm rank}$ \\
\midrule
\endfirsthead
\caption[]{Individual $d=15$ compatibility matrices (continued).
All times are in seconds.}\\
\toprule
$m$ & $t$ & $r$ & $e$ & $n_{\rm row}$ & $n_{\rm col}$ & $\operatorname{nnz}(A)$ & $\rho$ & $\kappa$ & $T_{\rm gen}$ & $T_{\rm load}$ & $T_{\rm solve}$ & $T_{\rm rank}$ \\
\midrule
\endhead
\midrule
\multicolumn{13}{r}{\footnotesize Continued on the next page.}\\
\endfoot
\bottomrule
\endlastfoot
$3$ & $-2$ & $0$ & $+1$ & $6\,332$ & $1\,996$ & $44\,465$ & $1\,996$ & $0$ & $0.19$ & $0.22$ & $1.72$ & $1.94$ \\
$3$ & $-2$ & $0$ & $-1$ & $6\,311$ & $1\,996$ & $44\,455$ & $1\,996$ & $0$ & $0.19$ & $0.20$ & $1.65$ & $1.84$ \\
$3$ & $-2$ & $1$ & $+1$ & $6\,956$ & $2\,476$ & $48\,166$ & $2\,476$ & $0$ & $0.20$ & $0.22$ & $1.74$ & $1.96$ \\
$3$ & $-2$ & $1$ & $-1$ & $6\,930$ & $2\,476$ & $48\,166$ & $2\,476$ & $0$ & $0.20$ & $0.22$ & $1.88$ & $2.10$ \\
$3$ & $-2$ & $2$ & $+1$ & $7\,731$ & $2\,867$ & $61\,285$ & $2\,867$ & $0$ & $0.26$ & $0.29$ & $2.26$ & $2.55$ \\
$3$ & $-2$ & $2$ & $-1$ & $7\,706$ & $2\,867$ & $61\,263$ & $2\,867$ & $0$ & $0.26$ & $0.25$ & $1.87$ & $2.13$ \\
$3$ & $-2$ & $3$ & $+1$ & $8\,735$ & $3\,247$ & $99\,055$ & $3\,247$ & $0$ & $0.40$ & $0.30$ & $2.40$ & $2.69$ \\
$3$ & $-2$ & $3$ & $-1$ & $8\,711$ & $3\,247$ & $99\,053$ & $3\,247$ & $0$ & $0.39$ & $0.32$ & $2.49$ & $2.81$ \\
$3$ & $-2$ & $4$ & $+1$ & $8\,472$ & $3\,080$ & $93\,488$ & $3\,080$ & $0$ & $0.38$ & $0.33$ & $2.30$ & $2.63$ \\
$3$ & $-2$ & $4$ & $-1$ & $8\,449$ & $3\,080$ & $93\,485$ & $3\,080$ & $0$ & $0.37$ & $0.29$ & $2.11$ & $2.40$ \\
$3$ & $-2$ & $5$ & $+1$ & $8\,214$ & $2\,918$ & $88\,133$ & $2\,918$ & $0$ & $0.36$ & $0.28$ & $2.17$ & $2.45$ \\
$3$ & $-2$ & $5$ & $-1$ & $8\,192$ & $2\,918$ & $88\,131$ & $2\,918$ & $0$ & $0.36$ & $0.26$ & $1.97$ & $2.23$ \\
$3$ & $-2$ & $6$ & $+1$ & $7\,961$ & $2\,761$ & $82\,992$ & $2\,761$ & $0$ & $0.33$ & $0.23$ & $1.87$ & $2.10$ \\
$3$ & $-2$ & $6$ & $-1$ & $7\,940$ & $2\,761$ & $82\,989$ & $2\,761$ & $0$ & $0.34$ & $0.23$ & $2.17$ & $2.40$ \\
$3$ & $-2$ & $7$ & $+1$ & $7\,713$ & $2\,609$ & $78\,063$ & $2\,609$ & $0$ & $0.31$ & $0.24$ & $2.16$ & $2.40$ \\
$3$ & $-2$ & $7$ & $-1$ & $7\,693$ & $2\,609$ & $78\,061$ & $2\,609$ & $0$ & $0.32$ & $0.26$ & $2.06$ & $2.32$ \\
$3$ & $-2$ & $8$ & $+1$ & $7\,470$ & $2\,462$ & $73\,348$ & $2\,462$ & $0$ & $0.30$ & $0.25$ & $1.84$ & $2.08$ \\
$3$ & $-2$ & $8$ & $-1$ & $7\,451$ & $2\,462$ & $73\,345$ & $2\,462$ & $0$ & $0.29$ & $0.24$ & $1.80$ & $2.04$ \\
$3$ & $-2$ & $9$ & $+1$ & $7\,232$ & $2\,320$ & $68\,845$ & $2\,320$ & $0$ & $0.28$ & $0.26$ & $1.78$ & $2.04$ \\
$3$ & $-2$ & $9$ & $-1$ & $7\,214$ & $2\,320$ & $68\,843$ & $2\,320$ & $0$ & $0.28$ & $0.23$ & $1.89$ & $2.12$ \\
$3$ & $-2$ & $10$ & $+1$ & $6\,999$ & $2\,183$ & $64\,554$ & $2\,183$ & $0$ & $0.27$ & $0.30$ & $1.75$ & $2.06$ \\
$3$ & $-2$ & $10$ & $-1$ & $6\,982$ & $2\,183$ & $64\,552$ & $2\,183$ & $0$ & $0.27$ & $0.28$ & $1.81$ & $2.10$ \\
$3$ & $-2$ & $11$ & $+1$ & $6\,771$ & $2\,051$ & $60\,477$ & $2\,051$ & $0$ & $0.25$ & $0.26$ & $1.79$ & $2.05$ \\
$3$ & $-2$ & $11$ & $-1$ & $6\,755$ & $2\,051$ & $60\,476$ & $2\,051$ & $0$ & $0.25$ & $0.24$ & $1.81$ & $2.05$ \\
$3$ & $-2$ & $12$ & $+1$ & $6\,548$ & $1\,924$ & $56\,612$ & $1\,924$ & $0$ & $0.23$ & $0.21$ & $1.59$ & $1.80$ \\
$3$ & $-2$ & $12$ & $-1$ & $6\,533$ & $1\,924$ & $56\,611$ & $1\,924$ & $0$ & $0.23$ & $0.21$ & $1.74$ & $1.95$ \\
$3$ & $-2$ & $13$ & $+1$ & $6\,330$ & $1\,802$ & $52\,961$ & $1\,802$ & $0$ & $0.22$ & $0.25$ & $1.60$ & $1.86$ \\
$3$ & $-2$ & $13$ & $-1$ & $6\,316$ & $1\,802$ & $52\,961$ & $1\,802$ & $0$ & $0.22$ & $0.24$ & $1.50$ & $1.75$ \\
$3$ & $-2$ & $14$ & $+1$ & $6\,116$ & $1\,684$ & $49\,430$ & $1\,684$ & $0$ & $0.20$ & $0.23$ & $1.50$ & $1.73$ \\
$3$ & $-2$ & $14$ & $-1$ & $6\,102$ & $1\,684$ & $49\,429$ & $1\,684$ & $0$ & $0.20$ & $0.22$ & $1.51$ & $1.73$ \\
$3$ & $-2$ & $15$ & $+1$ & $5\,906$ & $1\,570$ & $46\,022$ & $1\,570$ & $0$ & $0.19$ & $0.19$ & $1.64$ & $1.83$ \\
$3$ & $-2$ & $15$ & $-1$ & $5\,893$ & $1\,570$ & $46\,021$ & $1\,570$ & $0$ & $0.19$ & $0.20$ & $1.57$ & $1.77$ \\
\addlinespace[3pt]
$4$ & $-3$ & $0$ & $+1$ & $15\,273$ & $4\,949$ & $253\,466$ & $4\,949$ & $0$ & $0.64$ & $0.54$ & $3.30$ & $3.85$ \\
$4$ & $-3$ & $0$ & $-1$ & $15\,304$ & $4\,949$ & $253\,488$ & $4\,949$ & $0$ & $0.64$ & $0.54$ & $3.09$ & $3.63$ \\
$4$ & $-3$ & $1$ & $+1$ & $16\,203$ & $5\,722$ & $258\,808$ & $5\,722$ & $0$ & $0.67$ & $0.51$ & $3.40$ & $3.90$ \\
$4$ & $-3$ & $1$ & $-1$ & $16\,242$ & $5\,722$ & $258\,822$ & $5\,722$ & $0$ & $0.67$ & $0.53$ & $3.43$ & $3.96$ \\
$4$ & $-3$ & $2$ & $+1$ & $17\,496$ & $6\,470$ & $283\,664$ & $6\,470$ & $0$ & $0.67$ & $0.60$ & $4.17$ & $4.78$ \\
$4$ & $-3$ & $2$ & $-1$ & $17\,533$ & $6\,470$ & $283\,688$ & $6\,470$ & $0$ & $0.67$ & $0.60$ & $4.56$ & $5.16$ \\
$4$ & $-3$ & $3$ & $+1$ & $18\,500$ & $6\,929$ & $347\,549$ & $6\,929$ & $0$ & $0.88$ & $0.69$ & $4.20$ & $4.89$ \\
$4$ & $-3$ & $3$ & $-1$ & $18\,536$ & $6\,929$ & $347\,545$ & $6\,929$ & $0$ & $0.88$ & $0.68$ & $4.02$ & $4.70$ \\
$4$ & $-3$ & $4$ & $+1$ & $19\,734$ & $7\,379$ & $496\,877$ & $7\,379$ & $0$ & $1.23$ & $0.85$ & $4.62$ & $5.47$ \\
$4$ & $-3$ & $4$ & $-1$ & $19\,769$ & $7\,379$ & $496\,897$ & $7\,379$ & $0$ & $1.22$ & $0.80$ & $4.80$ & $5.60$ \\
$4$ & $-3$ & $5$ & $+1$ & $19\,254$ & $7\,076$ & $473\,496$ & $7\,076$ & $0$ & $1.15$ & $0.81$ & $4.59$ & $5.40$ \\
$4$ & $-3$ & $5$ & $-1$ & $19\,288$ & $7\,076$ & $473\,516$ & $7\,076$ & $0$ & $1.15$ & $0.83$ & $4.44$ & $5.26$ \\
$4$ & $-3$ & $6$ & $+1$ & $18\,780$ & $6\,780$ & $450\,781$ & $6\,780$ & $0$ & $1.13$ & $0.76$ & $4.13$ & $4.89$ \\
$4$ & $-3$ & $6$ & $-1$ & $18\,813$ & $6\,780$ & $450\,800$ & $6\,780$ & $0$ & $1.10$ & $0.75$ & $3.55$ & $4.30$ \\
$4$ & $-3$ & $7$ & $+1$ & $18\,312$ & $6\,491$ & $428\,728$ & $6\,491$ & $0$ & $1.05$ & $0.74$ & $4.04$ & $4.78$ \\
$4$ & $-3$ & $7$ & $-1$ & $18\,344$ & $6\,491$ & $428\,747$ & $6\,491$ & $0$ & $1.04$ & $0.76$ & $4.23$ & $4.99$ \\
$4$ & $-3$ & $8$ & $+1$ & $17\,850$ & $6\,209$ & $407\,341$ & $6\,209$ & $0$ & $1.01$ & $0.73$ & $4.00$ & $4.73$ \\
$4$ & $-3$ & $8$ & $-1$ & $17\,881$ & $6\,209$ & $407\,359$ & $6\,209$ & $0$ & $0.99$ & $0.73$ & $3.86$ & $4.59$ \\
$4$ & $-3$ & $9$ & $+1$ & $17\,394$ & $5\,934$ & $386\,614$ & $5\,934$ & $0$ & $0.96$ & $0.71$ & $3.11$ & $3.82$ \\
$4$ & $-3$ & $9$ & $-1$ & $17\,424$ & $5\,934$ & $386\,633$ & $5\,934$ & $0$ & $0.95$ & $0.69$ & $3.75$ & $4.44$ \\
$4$ & $-3$ & $10$ & $+1$ & $16\,944$ & $5\,666$ & $366\,555$ & $5\,666$ & $0$ & $0.90$ & $0.70$ & $3.81$ & $4.50$ \\
$4$ & $-3$ & $10$ & $-1$ & $16\,973$ & $5\,666$ & $366\,573$ & $5\,666$ & $0$ & $0.90$ & $0.67$ & $3.60$ & $4.27$ \\
$4$ & $-3$ & $11$ & $+1$ & $16\,500$ & $5\,405$ & $347\,156$ & $5\,405$ & $0$ & $0.85$ & $0.68$ & $3.42$ & $4.10$ \\
$4$ & $-3$ & $11$ & $-1$ & $16\,529$ & $5\,405$ & $347\,175$ & $5\,405$ & $0$ & $0.84$ & $0.68$ & $3.26$ & $3.94$ \\
$4$ & $-3$ & $12$ & $+1$ & $16\,062$ & $5\,151$ & $328\,425$ & $5\,151$ & $0$ & $0.81$ & $0.59$ & $3.46$ & $4.05$ \\
$4$ & $-3$ & $12$ & $-1$ & $16\,090$ & $5\,151$ & $328\,443$ & $5\,151$ & $0$ & $0.79$ & $0.59$ & $2.84$ & $3.43$ \\
$4$ & $-3$ & $13$ & $+1$ & $15\,631$ & $4\,904$ & $310\,353$ & $4\,904$ & $0$ & $0.75$ & $0.60$ & $3.22$ & $3.81$ \\
$4$ & $-3$ & $13$ & $-1$ & $15\,657$ & $4\,904$ & $310\,371$ & $4\,904$ & $0$ & $0.75$ & $0.59$ & $2.94$ & $3.53$ \\
$4$ & $-3$ & $14$ & $+1$ & $15\,207$ & $4\,664$ & $292\,949$ & $4\,664$ & $0$ & $0.72$ & $0.57$ & $3.15$ & $3.71$ \\
$4$ & $-3$ & $14$ & $-1$ & $15\,232$ & $4\,664$ & $292\,965$ & $4\,664$ & $0$ & $0.71$ & $0.55$ & $3.07$ & $3.63$ \\
$4$ & $-3$ & $15$ & $+1$ & $14\,789$ & $4\,431$ & $276\,205$ & $4\,431$ & $0$ & $0.68$ & $0.57$ & $3.03$ & $3.60$ \\
$4$ & $-3$ & $15$ & $-1$ & $14\,813$ & $4\,431$ & $276\,221$ & $4\,431$ & $0$ & $0.67$ & $0.57$ & $2.65$ & $3.23$ \\
\end{longtable}
\endgroup

\clearpage
\section{Computations in the logarithmic case}
\label{app:log-computations}

\subsection{Computational setup}

The logarithmic tests concern the complement of the degree-$11$ plane
curve
\begin{equation}
 C:\quad X^{11}+Y^{11}+T^{11}+T^6X^5+T^6Y^5=0
 \quad\text{in }\mathbb{P}^2.
 \label{eq:log-curve}
\end{equation}
Its affine defining polynomial is
$R(x,y)=x^{11}+y^{11}+1+x^5+y^5$.
The generator identifies this example as \texttt{normal\_type}, with
$d=10$, and uses the twist rule
\begin{equation}
 t(m)=-\left\lceil\frac{13m}{96}\right\rceil,
 \qquad m=3,\ldots,22.
 \label{eq:log-twists}
\end{equation}
The two branches $e=\pm1$ impose the two signs of the compatibility
identity under interchange of $x$ and $y$. The default branch is $e=+1$;
the option \texttt{--negtive} selects $e=-1$. These runs have no additional root-weight parameter $r$. Thus the tested range comprises $40$ parameter combinations,
two for each $m$.

In the logarithmic coefficient ansatz, the polynomial coefficient
indexed by $(a,q)$ has degree bound
\begin{equation}
 D_{a,q}=(m-2q)d-m+a+q+t,
 \qquad 0\le q\le\lfloor m/3\rfloor,
 \qquad 0\le a\le m-3q.
 \label{eq:log-degree-bound}
\end{equation}
It contains all monomials $x^iy^j$ with $i,j\ge0$ and
$i+j\le D_{a,q}$. Consequently the common number of columns of the two
sign matrices is
\begin{equation}
 C(m,t)=\sum_{q=0}^{\lfloor m/3\rfloor}
        \sum_{a=0}^{m-3q}
        \begin{cases}
         \binom{D_{a,q}+2}{2},&D_{a,q}\ge0,\\
         0,&D_{a,q}<0.
        \end{cases}
 \label{eq:log-column-count}
\end{equation}
This formula agrees with the column count in every final rank record
reported below.

The C++ generator distributes substitutions, polynomial expansions and
coefficient extraction among MPI processes. The affine-chart conditions
and the conditions at infinity are included in the same sparse system.
Depending on the problem size, generation uses $2$ to $40$ MPI processes.
The matrices are stored in HDF5. For column-oriented output, the block
count is $B(m)$ in~\eqref{eq:compact-block-count}; each block is written
independently, with a common numbering of the matrix columns. For
example, the negative-sign inputs at $m=17,\ldots,22$ contain
$64,71,78,85,93,101$ blocks, respectively. A complete set of input files
defines one matrix and one rank problem. Depending on the output layout,
files group the data by generating MPI process or by coefficient block,
so the file count and MPI process count need not coincide. Empty
auxiliary files are excluded from the input sets and storage statistics.

The rank stage uses the GPU/DCU Wiedemann procedure described in
Section~\ref{app:compact-methods}, with $1$ to $64$ MPI processes and
one mapped DCU per process. MPI distributes the sparse products among
DCUs; each device parallelizes the products within a Krylov step.
The process counts for both sign branches are reported in
Table~\ref{tab:log-times}. All ranks in this section are over
$\mathbb{F}_{32749}$. The generation parameter \texttt{p0} in an input
filename denotes unreduced integer coefficients and is not the modulus
of the subsequent rank computation.

\subsection{Results}

We computed the ranks of all $40$ logarithmic normal-type matrices for
$d=10$, $m=3,\ldots,22$, and both signs $e=\pm1$, using the twist
rule~\eqref{eq:log-twists}. Every matrix has full column rank over
$\mathbb{F}_{32749}$. Each parameter combination is counted once.

Table~\ref{tab:log-sizes-ranks} gives one row for every $m$ in the
tested range. Here $N_\pm$ are the row counts, $C$ is the common column
count, and $\rho_\pm$ are the measured ranks. The total column nullity is
$\kappa_\Sigma=2C-\rho_+-\rho_-$.

\begin{table}[htbp]
\centering
\caption{Logarithmic normal-type results for $d=10$, one row per weighted degree. All ranks are verified over $\mathbb{F}_{32749}$. The last column is the total column nullity of the two sign branches.}
\label{tab:log-sizes-ranks}
\small
\setlength{\tabcolsep}{3pt}
\begin{tabular}{rrrrrrrr}
\toprule
$m$ & $t$ & $N_+$ & $N_-$ & $C$ & $\rho_+$ & $\rho_-$ & $\kappa_\Sigma$ \\
\midrule
$3$ & $-1$ & $4\,386$ & $4\,368$ & $1\,720$ & $1\,720$ & $1\,720$ & $0$ \\
$4$ & $-1$ & $10\,100$ & $10\,127$ & $4\,034$ & $4\,034$ & $4\,034$ & $0$ \\
$5$ & $-1$ & $20\,088$ & $20\,052$ & $8\,055$ & $8\,055$ & $8\,055$ & $0$ \\
$6$ & $-1$ & $35\,732$ & $35\,785$ & $14\,461$ & $14\,461$ & $14\,461$ & $0$ \\
$7$ & $-1$ & $59\,207$ & $59\,140$ & $24\,062$ & $24\,062$ & $24\,062$ & $0$ \\
$8$ & $-2$ & $90\,629$ & $90\,708$ & $36\,549$ & $36\,549$ & $36\,549$ & $0$ \\
$9$ & $-2$ & $135\,502$ & $135\,398$ & $54\,889$ & $54\,889$ & $54\,889$ & $0$ \\
$10$ & $-2$ & $194\,974$ & $195\,096$ & $79\,343$ & $79\,343$ & $79\,343$ & $0$ \\
$11$ & $-2$ & $272\,250$ & $272\,110$ & $110\,995$ & $110\,995$ & $110\,995$ & $0$ \\
$12$ & $-2$ & $370\,046$ & $370\,220$ & $151\,425$ & $151\,425$ & $151\,425$ & $0$ \\
$13$ & $-2$ & $492\,465$ & $492\,269$ & $201\,910$ & $201\,910$ & $201\,910$ & $0$ \\
$14$ & $-2$ & $642\,216$ & $642\,435$ & $263\,805$ & $263\,805$ & $263\,805$ & $0$ \\
$15$ & $-3$ & $815\,711$ & $815\,453$ & $333\,566$ & $333\,566$ & $333\,566$ & $0$ \\
$16$ & $-3$ & $1\,031\,767$ & $1\,032\,052$ & $422\,884$ & $422\,884$ & $422\,884$ & $0$ \\
$17$ & $-3$ & $1\,288\,413$ & $1\,288\,101$ & $528\,675$ & $528\,675$ & $528\,675$ & $0$ \\
$18$ & $-3$ & $1\,589\,653$ & $1\,590\,015$ & $653\,646$ & $653\,646$ & $653\,646$ & $0$ \\
$19$ & $-3$ & $1\,941\,562$ & $1\,941\,168$ & $799\,316$ & $799\,316$ & $799\,316$ & $0$ \\
$20$ & $-3$ & $2\,347\,459$ & $2\,347\,884$ & $967\,582$ & $967\,582$ & $967\,582$ & $0$ \\
$21$ & $-3$ & $2\,815\,296$ & $2\,814\,812$ & $1\,161\,826$ & $1\,161\,826$ & $1\,161\,826$ & $0$ \\
$22$ & $-3$ & $3\,348\,508$ & $3\,349\,028$ & $1\,383\,586$ & $1\,383\,586$ & $1\,383\,586$ & $0$ \\
\bottomrule
\end{tabular}
\end{table}

The matrix dimensions, sparse-entry counts and storage measurements
were checked against the corresponding generation records and HDF5
metadata. Every final rank was checked against its input dimensions,
global sparse-entry count and completion record. Process/device mappings
and benchmark timings were checked for consistency. Metadata validation
uses HDF5 headers and small scalar datasets, without scanning coefficient
arrays. Row counts come from final rank records or from the sum of the
disjoint block row-ID counts.
In Table~\ref{tab:log-sparse-storage}, $Z_\pm$ are the reported
sparse-entry counts and $S_\pm$ are input file sizes in GiB
($1\,\mathrm{GiB}=2^{30}$ bytes). Each input size and timing refers to
the same generation and rank execution as its corresponding result.
Logs, checkpoints, obsolete inputs and empty auxiliary files are excluded.

\begin{table}[htbp]
\centering
\caption{Sparse-entry counts and HDF5 storage for the logarithmic inputs used in the reported records. Both sign branches are listed; storage is in GiB and excludes logs, checkpoints and empty auxiliary files.}
\label{tab:log-sparse-storage}
\small
\setlength{\tabcolsep}{3pt}
\begin{tabular}{rrrrr}
\toprule
$m$ & $Z_+$ & $Z_-$ & $S_+$ (GiB) & $S_-$ (GiB) \\
\midrule
$3$ & $52\,806$ & $52\,803$ & $0.003$ & $0.003$ \\
$4$ & $276\,746$ & $276\,759$ & $0.015$ & $0.015$ \\
$5$ & $1\,115\,238$ & $1\,115\,234$ & $0.059$ & $0.059$ \\
$6$ & $3\,776\,203$ & $3\,776\,232$ & $0.199$ & $0.199$ \\
$7$ & $11\,216\,720$ & $11\,216\,721$ & $0.597$ & $0.596$ \\
$8$ & $28\,472\,357$ & $28\,472\,420$ & $1.536$ & $1.532$ \\
$9$ & $69\,375\,759$ & $69\,375\,775$ & $3.805$ & $3.792$ \\
$10$ & $156\,716\,231$ & $156\,716\,316$ & $8.729$ & $8.706$ \\
$11$ & $329\,894\,994$ & $329\,895\,014$ & $18.645$ & $18.603$ \\
$12$ & $662\,360\,825$ & $662\,360\,948$ & $38.041$ & $37.960$ \\
$13$ & $1\,264\,644\,169$ & $1\,264\,644\,208$ & $73.791$ & $73.648$ \\
$14$ & $2\,303\,308\,688$ & $2\,303\,308\,856$ & $136.474$ & $136.250$ \\
$15$ & $3\,982\,782\,169$ & $3\,982\,782\,173$ & $239.662$ & $239.328$ \\
$16$ & $6\,789\,102\,006$ & $6\,789\,102\,160$ & $414.776$ & $414.242$ \\
$17$ & $11\,190\,028\,434$ & $11\,190\,028\,435$ & $693.341$ & $692.522$ \\
$18$ & $18\,028\,801\,449$ & $18\,028\,801\,650$ & $1\,133.154$ & $1\,132.109$ \\
$19$ & $28\,285\,684\,643$ & $28\,285\,684\,662$ & $1\,802.386$ & $1\,801.002$ \\
$20$ & $43\,283\,769\,151$ & $43\,283\,769\,404$ & $2\,794.525$ & $2\,792.632$ \\
$21$ & $65\,129\,678\,930$ & $65\,129\,678\,945$ & $4\,263.076$ & $4\,260.365$ \\
$22$ & $96\,034\,279\,576$ & $96\,034\,279\,861$ & $6\,370.636$ & $6\,366.774$ \\
\bottomrule
\end{tabular}
\end{table}

Table~\ref{tab:log-times} reports measured generation and rank-run
times in hours. Each $T_{\rm gen}$ is the interval between the first and
last timestamped messages in the corresponding successful generation
log. Each $T_{\rm run}$ is the wrapper interval around the completed
rank invocation, including MPI startup, input loading, iteration,
recurrence reconstruction and shutdown, as well as result verification
where this is performed within the wrapper. Queueing and matrix generation are
excluded. These wrapper times differ from the in-program
$T_{\rm rank}$ benchmarks in the compact-case tables. The accompanying
per-case data retain the benchmark loading, solve and total times
separately, at their recorded precision.
The process counts $p_{\rm gen}$ and $p_{\rm rank}$ are taken from the
corresponding runs. These measurements use different process counts
and are not a uniform-hardware performance comparison. Missing timings are not estimated by doubling one sign.

\begin{table}[htbp]
\centering
\caption{Logarithmic-case logged times in hours. Process-count entries are ordered as positive/negative branch, with one mapped DCU per rank-stage MPI process. Times are shown to two decimal places, with values below $0.01$ hours written as $0$. A dagger denotes the resumed invocation only; its retained solve work across both segments is $62.48$ hours.}
\label{tab:log-times}
\small
\setlength{\tabcolsep}{3pt}
\begin{tabular}{rrrrrrr}
\toprule
$m$ & $p_{\rm gen}^{+/-}$ & $T_{\rm gen}^{+}$ & $T_{\rm gen}^{-}$ & $p_{\rm rank}^{+/-}$ & $T_{\rm run}^{+}$ & $T_{\rm run}^{-}$ \\
\midrule
$3$ & $2/2$ & $0$ & $0$ & $1/1$ & $0$ & $0$ \\
$4$ & $4/4$ & $0$ & $0$ & $1/1$ & $0$ & $0$ \\
$5$ & $4/4$ & $0$ & $0$ & $1/1$ & $0$ & $0$ \\
$6$ & $6/6$ & $0$ & $0$ & $1/1$ & $0$ & $0$ \\
$7$ & $8/8$ & $0$ & $0$ & $2/2$ & $0.01$ & $0.01$ \\
$8$ & $8/8$ & $0$ & $0$ & $2/2$ & $0.02$ & $0.02$ \\
$9$ & $10/10$ & $0.02$ & $0.02$ & $2/2$ & $0.06$ & $0.06$ \\
$10$ & $12/12$ & $0.05$ & $0.05$ & $4/4$ & $0.10$ & $0.10$ \\
$11$ & $12/12$ & $0.07$ & $0.06$ & $4/4$ & $0.23$ & $0.24$ \\
$12$ & $12/12$ & $0.12$ & $0.12$ & $4/4$ & $0.53$ & $0.54$ \\
$13$ & $12/12$ & $0.29$ & $0.28$ & $8/8$ & $0.74$ & $0.74$ \\
$14$ & $12/12$ & $0.67$ & $0.65$ & $8/8$ & $1.57$ & $1.57$ \\
$15$ & $40/12$ & $1.92$ & $1.06$ & $2/8$ & $12.39$ & $3.22$ \\
$16$ & $40/12$ & $3.26$ & $2.38$ & $8/16$ & $14.61$ & $4.17$ \\
$17$ & $40/12$ & $5.62$ & $2.23$ & $24/24$ & $23.85$ & $6.26$ \\
$18$ & $40/10$ & $8.23$ & $5.49$ & $48/16$ & $36.68$ & $14.69$ \\
$19$ & $10/10$ & $5.67$ & $6.50$ & $48/24$ & $17.68$ & $20.34$ \\
$20$ & $10/8$ & $9.29$ & $13.02$ & $64/32$ & $22.84$ & $29.70$ \\
$21$ & $10/5$ & $15.08$ & $30.11$ & $64/40$ & $47.84$ & $44.84$ \\
$22$ & $5/5$ & $44.11$ & $46.24$ & $64/48$ & $17.18^{\dagger}$ & $64.57$ \\
\bottomrule
\end{tabular}
\end{table}

The positive $m=22$ computation resumed from iteration
$2\,250\,000$ of $2\,767\,172$. Its last invocation took $17.18$ hours
including reloading; the in-program resumed solve time was
$59\,135.90$ seconds. The interval from the initial solver creation to
the retained checkpoint was $165\,787.97$ seconds.
Together these give $62.48$ hours of retained solve work, excluding
input loading, downtime, rolled-back work and older failed attempts.
The $17.18$-hour table entry consequently describes the resumed
invocation and must not be read as the duration of the entire $m=22$
computation.

Both signs for every $m=3,\ldots,22$ have measured column nullity zero
over $\mathbb{F}_{32749}$. Across all $40$ logarithmic matrices, the
total column count and the sum of the modular ranks are both
$14\,404\,658$, and the total column nullity is zero. Every logarithmic
rank entry is supported by a completed computation. Including the $450$ earlier compact-sector tests and the $64$
compatibility sectors in Section~\ref{app:compact-compat-d15}, the
appendix covers $554$ parameter combinations. All $554$ have
completed modular rank records.

\medskip\noindent\textbf{Acknowledgement.}\quad
This work was supported by the ``Two Major'' Construction Project for
Enhancing the Original Innovation Capability of the State Key Laboratory
of Mathematical Sciences. This project began when Lei Hou and Song-Yan Xie attended a 2024
research-group program at the Tianyuan Mathematics Research Center in Kunming,
Yunnan, held in conjunction with the conference ``Progress in Several Complex
Variables and Complex Geometry''.

We thank Yum-Tong Siu for suggesting to Song-Yan Xie the direct-image
argument for the relative-family construction, and the Center for its
excellent research environment.  Valuable discussions with Jo\"el Merker
concerned computational algebra, especially the limits of Maple's modular
arithmetic for systems of this size.  Guoxiong Gao suggested the additional
$\ZZ/15\ZZ$ symmetry in the compact case; this substantially reduced the
systems and made the hardest compact computation feasible with the
available equipment.

\medskip\noindent\textbf{AI-assisted preparation.}\quad
The ideas of the symmetry method and of the argument with pole-three
slanted vector fields are due to the authors.  ChatGPT-5.6-Sol assisted with
the details of the latter proof, writing, \texttt{C++} code design, and
checks of the code's consistency and correctness.  ChatGPT-6.1-Sol assisted
with revision and polishing.  The authors independently checked the code
and take full
responsibility for the work.

\medskip\noindent\textbf{Funding.}\quad 
S.-Y. Xie acknowledges partial support from the National Key R\&D Program of
China under Grants No.~2023YFA1010500 and No.~2021YFA1003100, from the
National Natural Science Foundation of China under Grants No.~12288201 and
No.~12471081, and from the Xiaomi Young Talents Program.
L. Hou is supported by the CAS Project for Young Scientists in Basic
Research, Grant No.~YSBR-033.



\end{document}